\documentclass[reqno]{amsart}

\usepackage{amssymb,a4wide,color,eucal,enumerate,mathrsfs,ifthen}

\numberwithin{equation}{section}

\newtheorem{theorem}{Theorem}[section]

\newtheorem{corollary}[theorem]{Corollary}
\newtheorem{lemma}[theorem]{Lemma}
\newtheorem{proposition}[theorem]{Proposition}

\theoremstyle{definition}
\newtheorem{definition}[theorem]{Definition}

\newtheorem{example}[theorem]{Example}

\theoremstyle{remark}
\newtheorem{remark}[theorem]{Remark}

\definecolor{ddmagenta}{rgb}{0.7,0,1.0}
\definecolor{ddcyan}{rgb}{0,0.1,1.0}
\definecolor{dred}{rgb}{.8,0,0}

 \numberwithin{equation}{section}

\newcommand{\Uu}{\mathrm{U}}
\newcommand{\Ww}{\mathrm{W}}

\newcommand{\N}{\mathbb{N}}

\newcommand{\R}{\mathbb{R}}

\newcommand{\BB}{\mathscr{B}}

\newcommand{\LL}{\mathscr{L}}

\newcommand{\XX}{\mathscr{X}}
\newcommand{\VV}{\mathscr{V}}

\newcommand{\cB}{{\ensuremath{\mathcal B}}}

\newcommand{\cE}{{\ensuremath{\mathcal E}}}

\newcommand{\cP}{{\ensuremath{\mathcal P}}}
\newcommand{\cR}{{\ensuremath{\mathcal R}}}

\newcommand{\dd}{{\mbox{\boldmath$d$}}}

\newcommand{\uu}{{\mbox{\boldmath$u$}}}
\newcommand{\vv}{{\mbox{\boldmath$v$}}}

\newcommand{\xx}{{\mbox{\boldmath$x$}}}
\newcommand{\sfxx}{{\mbox{\boldmath$\sfx$}}}

\newcommand{\sfa}{{\sf a}}
\newcommand{\sfb}{{\sf b}}

\newcommand{\sfm}{{\sf m}}

\newcommand{\sfp}{{\sf p}}

\newcommand{\sfr}{{\sf r}}
\newcommand{\sfs}{{\sf s}}
\newcommand{\sft}{{\sf t}}
\newcommand{\sfu}{{\sf u}}
\newcommand{\sfv}{{\sf v}}
\newcommand{\sfw}{{\sf w}}
\newcommand{\sfx}{{\sf x}}

\newcommand{\sfC}{{\sf C}}

\newcommand{\sfS}{{\sf S}}

\newcommand{\beq}{\begin{equation}}
\newcommand{\beql}[1]{\begin{equation}\label{#1}}
\newcommand{\eeq}{\end{equation}}

\newcommand{\ba}{\begin{array}} \newcommand{\ea}{\end{array}}

\newcommand{\eps}{\varepsilon}

\def\d{\mathrm d}
\def\dd{\;\!\mathrm{d}} 



\def\rmd{{\mathrm d}}

  \def\rmC{{\mathrm C}}
\def\rmD{{\mathrm D}}  
  \def\rmI{{\mathrm I}}
\def\rmJ{{\mathrm J}}

\newcommand{\weaksto}{{\rightharpoonup^*}\,}
\newcommand{\weakto}{\rightharpoonup}
\newcommand{\Restr}[1]{\lower2pt\hbox{$|_{#1}$}}

\newcommand{\argmin}{\mathop\textrm{Argmin}}

\newcommand{\trepar}{{|\kern-1truept|\kern-1truept|}}
\newcommand{\ltt}{{(\kern-2truept(}}
\newcommand{\rtt}{{)\kern-2truept)}}
\newcommand{\la}{\langle}
\newcommand{\ra}{\rangle}

\newcommand{\nchi}{{\raise.4ex\hbox{$\chi$}}}
\newcommand{\down}{\downarrow}
\newcommand{\up}{\uparrow}
\newcommand{\xfin}{X}
\newcommand{\Kx}{K^*}
\newcommand{\xfins}{X^*}
\newcommand{\Diss}{{\Psi_0}}
\newcommand{\DualDiss}{\Psi_0^*}

\newcommand{\rescs}{\mathsf s}
\newcommand{\rescpt}{\mathsf p}
\newcommand{\resct}{\mathsf t}
\newcommand{\rescu}{\mathsf u}

\newcommand{\rescw}{\mathsf w}
\newcommand{\rescT}{\mathsf S}
\newcommand{\rescS}{\mathsf S}

\newcommand{\forae}{\text{for a.a.}}

\newcommand{\ene}[2]{\cE_{#1}(#2)}

\newenvironment{frcseries}{\fontfamily{frc}\selectfont}{}

\newcommand{\Varname}[1]{{\rm Var}_#1}
\newcommand{\pVarname}[1]{{\text{\sl Var}_{#1}}}
\newcommand{\Var}[4]{\mathop{\rm Var}\nolimits_{#1}(#2;[#3,#4])}
\newcommand{\pVar}[4]{\mathop{\text{\sl Var}}\nolimits_{#1}(#2;[#3,#4])}
\newcommand{\JVar}[4]{{\rm Jmp}_{#1}(#2;[#3,#4])}
\newcommand{\BV}{\mathrm{ BV}}
\newcommand{\AC}{\mathrm{AC}}

\newcommand{\Norm}[3]{\cB_{#3}\left(#1;#2\right)}
\newcommand{\FNorm}[3]{\cR_{#3}\left(#1;#2\right)}

\newcommand{\Leb}[1]{\mathscr L^{#1}}
\newcommand{\pt}{p}
\newcommand{\Bip}{\mathfrak P}
\newcommand{\Bipo}{\mathfrak P}
\newcommand{\BipDom}{\BB}

\newcommand{\fl}{{\Lambda}}
\newcommand{\bip}{{\mathfrak p}}
\newcommand{\bipo}{{{\mathfrak p}}}
\newcommand{\bipotential}{vanishing viscosity contact potential\relax}
\newcommand{\shortbipotential}{contact potential\relax}

\newcommand{\Bipotential}{Vanishing viscosity contact potential}

\newcommand{\topref}[2]{\stackrel{\eqref{#1}}#2}
\newcommand{\Contact}[1]{\Sigma_{#1}}
\newcommand{\co}{{\rm co}}
\newcommand{\Ca}{{\rm C}}
\newcommand{\Le}{{\LL}}
\newcommand{\VRIS}{\ensuremath{(\xfin,\cE,\bipo)}}
\newcommand{\PVRIS}{\ensuremath{(\xfin,\cE,\Bip)}}
\newcommand{\RIS}{\ensuremath{(\xfin,\cE,\Diss)}}
\newcommand{\Mint}[4]{\int_{#1}^{#2}\d\Mdiff{#3}{#4}}
\newcommand{\Mdiff}[2]{#1( #2)}

\newcommand{\cost}[3]{\Delta_{#1}(#2,#3)}
\newcommand{\tricost}[4]{\Delta_{#1}(#2,#3,#4)}
\newcommand{\Cost}[4]{\Delta_{#1}(#2;#3,#4)}
\newcommand{\TriCost}[5]{\Delta_{#1}(#2;#3,#4,#5)}
\newcommand{\Dist}[3]{\Delta_{\cB_{#3}}(#1,#2)}
\newcommand{\FDist}[3]{\Delta_{\cR_{#3}}(#1,#2)}
\newcommand{\bipcE}{{\bipo,\cE}}
\newcommand{\MU}[2]{\mu_{#1,#2}}
\newcommand{\essJ}{\mathop{\text{\rm ess-J}}\nolimits}
\newcommand{\fnorm}[1]{\|#1\|}
\newcommand{\snorm}[1]{|#1|}
\newcommand{\dist}{\mathrm{dist}_*}
\newcommand{\taueps}{{\tau,\eps}}

\newenvironment{mydescription}%
{\begin{list}{}{\setlength{\labelwidth}{0pt}
      \setlength{\leftmargin}{12pt}
      \setlength{\itemindent}{-12pt}
      \setlength{\listparindent}{0pt}
      }}%
  {\end{list}}

\newenvironment{mylist}%
{\begin{list}{}{\setlength{\labelwidth}{0pt}
      \setlength{\leftmargin}{0pt}
      \setlength{\itemindent}{4pt}
      \setlength{\listparindent}{0pt}}}%
  {\end{list}}

\begin{document}

\title{BV solutions and viscosity approximations of rate-independent systems}
\date{9 October 2009}

\author{Alexander Mielke}
\address{Weierstra\ss-Institut,
  Mohrenstra\ss{}e 39, 10117 D--Berlin and Institut f\"ur
  Mathematik, Humboldt-Universit\"at zu
  Berlin, Rudower Chaussee 25, D--12489 Berlin (Adlershof), Germany.}
\email{mielke\,@\,wias-berlin.de}

\author{Riccarda Rossi}
\address{Dipartimento di Matematica, Universit\`a di
  Brescia, via Valotti 9, I--25133 Brescia, Italy.}
\email{riccarda.rossi\,@\,ing.unibs.it}

\author{Giuseppe Savar\'e}

\address{Dipartimento di Matematica ``F.\
  Casorati'', Universit\`a di Pavia.
  Via Ferrata, 1 -- 27100 Pavia, Italy.}
\email{ giuseppe.savare\,@\,unipv.it}
\thanks{
    A.M. has been partially supported by DFG, Research Unit FOR 797
    ``MicroPlast'', Mi\,459/5-1.
 R.R. and G.S. have  been partially supported by a  MIUR-PRIN'06 grant for the project
 ``Variational methods in optimal mass transportation and in geometric
 measure theory''.}

\begin{abstract}
   In the nonconvex case solutions of rate-independent systems
   may develop jumps as a function of time. To model such jumps, we
   adopt the philosophy that rate independence should be considered as
   limit of systems with smaller and smaller viscosity.
   For the finite-dimensional case we study the vanishing-viscosity
   limit of doubly nonlinear equations given in terms of a
   differentiable energy functional and a dissipation
   potential which is a viscous regularization of a given
   rate-independent dissipation potential.

   The resulting definition  of `BV solutions' involves, in a
   nontrivial way, both the  rate-independent   and the viscous dissipation
   potential, which play a crucial role in the description of the
   associated jump trajectories.

    We shall prove a general convergence result for the  time-continuous and
    for the time-discretized  viscous approximations and
    establish various properties of the limiting $\BV$ solutions.  In
    particular, we shall provide a careful description of the jumps
    and compare the new notion of solutions with the related concepts
    of energetic and local solutions to rate-independent systems.

\end{abstract}

\maketitle

\section{Introduction}
Rate-independent evolutions occur in several contexts. We refer the
reader to~\cite{Mielke05} and the forthcoming
monograph~\cite{MieRou08?RIST} for a survey of rate-independent
modeling and analysis in a wide variety of applications, which may
pertain to very different and far-apart branches of
mechanics and physics. Rate-independent systems present very distinctive common
features, because of their hysteretic character
\cite{Visintin94,Krej99EVIM}. Driven by external loadings on a
time scale much slower than their internal scale, such systems respond
to changes in the external actions invariantly for
time-rescalings. Thus, they in fact show (almost) no intrinsic
time-scale. This kind of behavior is encoded in the simplest, but
still significant, example of rate-independent evolution, namely the
doubly nonlinear differential inclusion
\begin{equation}
\label{e:ris_intro} \partial\Diss (u'(t)) + \mathrm{D} \cE_t (u(t)) \ni
0 \quad \text{in $\xfins$}\quad \text{for a.a. $t \in (0,T).$}
\tag{DN$_0$}
\end{equation}
For the sake of simplicity, we will consider here the case when
$\xfin$ is a finite dimensional linear
space, $\cE:[0,T]\times \xfin \to \R$
an  energy functional ($\mathrm{D}\cE$ denoting the differential of
$\cE$ with respect to the variable $u\in \xfin$), and $\Diss: \xfin
\to [0,+\infty)$ is a convex, 
  nondegenerate,
dissipation potential, hereafter supposed \emph{positively
homogeneous of degree $1$}. Thus,  \eqref{e:ris_intro} is invariant
for time-rescalings, rendering the system rate independence.

Since the range
  $\Kx$
of $\partial\Diss$ is a proper subset of $\xfins$, when
$\cE(t,\cdot)$ is not strictly convex one cannot expect the
existence of an absolutely continuous solution of
\eqref{e:ris_intro}. Over the past decade,  this fact has motivated
the development of suitable notions of weak solutions to
\eqref{e:ris_1}. In the mainstream of~\cite{Efendiev-Mielke06,
Mielke-Rossi-Savare08, Mielke-Zelik09}, the present paper aims to
contribute to this issue. Relying on the \emph{vanishing-viscosity
approach}, we shall propose the notion of $\BV$ solution
to~\eqref{e:ris_intro} and thoroughly analyze it.

To better motivate the use of
  vanishing viscosity and highlight the features of the concept of
  $\BV$ solution, in the next paragraphs we shall briefly recall the
  other main weak solvability notions for~\eqref{e:ris_intro}.
  For the sake of simplicity,  we shall focus on the particular case
  \begin{equation}
    \label{e:norm-intro} \Diss(v) =
      \|v\|,
     \qquad \text{for some norm
  $\fnorm\cdot$ on $\xfin$.}
\end{equation}

\paragraph{\textbf{Energetic and local solutions.}} \, The first
attempt at a rigorous weak formulation of~\eqref{e:ris_intro} goes
back to~\cite{MieThe99MMRI} and the
subsequent~\cite{Mielke-Theil-Levitas02, MieThe04RIHM}, which advanced
the notion of \emph{global energetic solution} to the rate-independent
system~\eqref{e:ris_intro}. In the simplified
case~\eqref{e:norm-intro}, this solution concept consists of the
following relations, holding for all $t \in [0,T]$:
\begin{equation}
  \label{eq:1-intro}
  \forall\, z\in \xfin:\qquad \ene t{u(t)}\le \ene tz+\fnorm{z-u(t)},
  \tag{$\mathrm{S}$}
\end{equation}
\begin{equation}
  \label{eq:2-intro}
  \ene{t}{u(t)}+\Var{}u{0}{t}=\ene{0}{u(0)}+
  \int_{0}^{t} \partial_t\ene s{u(s)}\,\d s\,.
  \tag{$\mathrm{E}$}
\end{equation}
The energy identity \eqref{eq:2-intro}  balances at every time $t
\in [0,T]$
 the dissipated energy $\Var{}u{0}{t}$ (the latter symbol denotes the total variation  of the
solution $u \in \BV([0,T];\xfin)$ on the interval $[0,t]$),
 with the stored energy $\ene t{u(t)}$, the initial energy,  and the
 work of the external forces. On the other hand,  \eqref{eq:1-intro} is a stability
 condition, for it  asserts that the change from the current state $u(t)$
 to another state $z$ brings about a gain of potential energy smaller
 than the dissipated energy. Since
 the  competitors for  $u(t)$ range in the whole space $\xfin$,
 \eqref{eq:1-intro} is in fact a \emph{global stability} condition.

 The global energetic formulation~\eqref{eq:1-intro}--\eqref{eq:2-intro}
  only involves the (assumedly
 smooth) power of the external forces $\partial_t \cE$, and is
 otherwise \emph{derivative-free}. Thus, it is well suited to jumping
 solutions. Furthermore, as shown in~\cite{MaiMie05EREM, Mielke05},
 it is amenable to analysis in very general ambient spaces, even
 with no underlying linear structure. Because of its flexibility,
 this concept has been exploited in a variety of applicative
 contexts, like, for instance,  shape memory alloys~\cite{Mielke-Theil-Levitas02,
MieRou03RIMI, AuMiSt08RIMI}, crack propagation~\cite{DalToa02MQSG,
DaFrTo05QCGN, DalZan07QSCG}, elastoplasticity~\cite{Miel02FELG,
Miel03EFME, Miel04EMIE, FraMie06ERCR,
DaDeMo06QEPL,DaDeMoMo06,MaiMie08?GERI}, damage in brittle materials
\cite{MieRou06RIDP, BoMiRo07?CDPS, ThoMie09?DNEM, Miel09?CDEB},
delamination~\cite{KoMiRo06RIAD},
ferroelectricity~\cite{MieTim06EMMT}, and
superconductivity~\cite{SchMie05VPSC}.

 On the other hand, in the case of nonconvex energies
condition~\eqref{eq:1-intro} turns out to be a strong requirement,
for it may lead the system to change instantaneously in a very
drastic way, jumping into very far-apart energetic configurations
(see, for instance, \cite[Ex.\,6.1]{Miel03EFME},
\cite[Ex.\,6.3]{KnMiZa07?ILMC}, and
\cite[Ex.\,1]{Mielke-Rossi-Savare08}). On the discrete level, global
stability is reflected in the global minimization scheme giving
raise to approximate solutions by time-discretization. Indeed,  for
a fixed time-step $\tau>0$, inducing a partition $\{
0=t_0<t_1<\ldots<t_{N-1}<t_N=T\}$ of the interval $[0,T]$, one
constructs \emph{discrete solutions} $(\Uu_{\tau}^n)_{n=1}^N$
of~\eqref{eq:1-intro}--\eqref{eq:2-intro} by setting
$\Uu_{\tau}^0:=u_0$ and then solving recursively the variational
incremental scheme
\begin{equation}
  \label{eq:58-intro}
  U^n_\tau\in \argmin_{\Uu \in \xfin}\Big\{\fnorm{\Uu-\Uu^{n-1}_\tau}+\ene{t_n}\Uu\Big\} \qquad \text{for $n=1,\ldots,
  N$.}
  \tag{$\mathrm{IP}_0$}
\end{equation}
However, a scheme based on \emph{local minimization}
 would be preferable, both in  view of numerical analysis and from a
 modeling perspective, see the discussions in
 \cite[Sec.\,6]{Miel03EFME} and, in the realm of crack propagation,
 \cite{DT02MQGB, NegOrt07?QSCP, Larsen09}.

 As pointed out in~\cite{DT02MQGB}, local minimization may be
 enforced by perturbing the variational scheme~\eqref{eq:58-intro}
 with a term, modulated by a \emph{viscosity} parameter $\eps$,
  which penalizes  the squared distance from the previous step
 $\Uu^{n-1}_\taueps$
\begin{equation}
  \label{eq:58-intro-eps}
  \Uu^n_\taueps\in \argmin_{\Uu \in \xfin}\Big\{\fnorm{\Uu-\Uu^{n-1}_\taueps}+
  \eps\frac{\snorm{\Uu-\Uu^{n-1}_\taueps}^2}{\tau}+\ene{t_n}\Uu\Big\}\quad \text{for $n=1,\ldots,
  N$}\,,
  \tag{$\mathrm{IP}_\eps$}
\end{equation}
and depends on a second norm $\snorm\cdot$, typically Hilbertian, on the space $\xfin$.
In a infinite-dimensional setting, one may think of $X=L^2 (\Omega)$, with $\Omega$ a domain
in $\R^d$, $d \geq 1$, and $\fnorm \cdot$, $\snorm\cdot$ the $L^1$ and
$L^2 $ norms, respectively. Notice that, on the time-continuous
level, \eqref{eq:58-intro} corresponds to the  \emph{viscous doubly
nonlinear} equation
\begin{equation}
\label{e:ris_intro-eps}
\begin{gathered}
\partial\Psi_\eps (u_\eps'(t)) + \mathrm{D} \cE_t
(u_\eps(t)) \ni 0 \quad \text{in $\xfins$}\quad \text{for a.a. $t \in
(0,T),$} \tag{DN$_\eps$}
\\
\text{with} \quad  \Psi_\eps(v) = \fnorm v+ \frac{\eps}2\snorm v^2
\end{gathered}
\end{equation}
(see~\cite{Colli-Visintin90, Colli92} for the existence of solutions
$u_\eps \in \AC([0,T];\xfin)$). Then, the idea would be to consider
the solutions to~\eqref{e:ris_intro} arising in the  passage to the
limit, in the discrete scheme \eqref{eq:58-intro-eps}, as $\eps $
and $\tau $ tend to $0$ \emph{simultaneously}, 
  keeping
  $\tau\ll\eps$. One can guess that, at least formally,
  this procedure should be equivalent to considering the limit
  of the solutions to \eqref{e:ris_intro-eps} as $\eps\downarrow 0$.

Vanishing viscosity has by now become an established selection
criterion for mechanically feasible weak solvability notions of
rate-independent evolutions. We refer the reader
to~\cite{KreLie07?RIKP} for rate-independent problems with convex
energies and  discontinuous inputs, and, in more specific applied
contexts, to~\cite{DDMM07?VVAQ} for elasto-plasticity with
softening, to \cite{fiaschi09} for general material models with
nonconvex elastic energies,  the recent
\cite{DaDeSo09?QECP} for cam-clay non-associative plasticity,
and \cite{ToaZan06?AVAQ, KnMiZa07?ILMC, KnMiZa08?CPPM}
for crack propagation. Since the energy functionals involved in such
applications are usually nonsmooth and nonconvex, the passage to the
limit mostly relies on lower semicontinuity arguments. Let us
illustrate the latter in the prototypical
case~\eqref{e:ris_intro-eps}. The key observation is
that~\eqref{e:ris_intro-eps} is equivalent (see the discussion in
Section~\ref{ss:2.1}) to  the \emph{$\eps$-energy identity}
\begin{equation}
\label{eps-enid}
\begin{aligned}
   \ene t{u_\eps(t)}+\int_0^t
   \Big(\fnorm{u_\eps'(s)}\, \d s + \frac\eps2 \snorm{u_\eps'(s)}^2
  & + \frac1{2\eps}
    \dist
  \big({-}\mathrm{D} \cE_s (u_\eps(s)),
   \Kx
   \big)^2 \Big)
  \d s  \\ &  = \ene0{u(0)} + \int_0^t
  \partial_t \ene{s}{u_\eps(s)} \d s
 \end{aligned}
 \end{equation}
for all $t \in [0,T]$, where the term
  \begin{equation}
    \label{e:term}
    \begin{aligned}
      \dist\big({-}\mathrm{D} \cE_t (u(t)), \Kx \big)&: = \min_{z \in \Kx}
      \snorm{-{\rm
          D} \cE_t (u(t)) -z}_{*},\ \
      \text{with}\ \
      \Kx= \big\{z\in \xfin^*:\fnorm z_*\le
      1\big\},
    \end{aligned}
  \end{equation}
measures the distance with respect to the dual norm $\snorm\cdot_*$
of $-\mathrm{D} \cE_t (u(t)) $ from the set $\Kx$.
The term defined in \eqref{e:term} is penalized in~\eqref{eps-enid}
by the coefficient $1/{2\eps}$. Thus, passing to
the limit in~\eqref{eps-enid}
 as $\eps
\down 0$, one finds
\[
\dist({-}\mathrm{D} \cE_t (u(t)), \Kx )=0 \qquad \forae\, t \in
(0,T)\,.
\]
Hence,
\begin{equation}
\label{loc-stab-intro} 
-\rmD\cE_t(u(t))\in \Kx,\quad\text{i.e.}\quad \fnorm {-\mathrm{D} \cE_t
(u(t))}_* \leq 1 \qquad \forae\, t \in (0,T)\,,
\end{equation}
which is a \emph{local version} of the global
stability~\eqref{eq:1-intro}. Furthermore, \eqref{eps-enid} yields,
via lower-semicontinuity, the  \emph{energy inequality}
\begin{equation}
  \label{eq:3-intro}
  \ene{t}{u(t)}+\Var{}u{0}{t} \leq \ene{0}{u(0)}+
  \int_{0}^{t} \partial_t\ene s{u(s)}\,\d s \quad \text{for all $t \in
  [0,T]$}\,.
\end{equation}
Conditions~\eqref{loc-stab-intro}--\eqref{eq:3-intro} give raise to
the notion of~\emph{local solution} of the rate-independent
system~\eqref{e:ris_intro}.

While the local stability~\eqref{loc-stab-intro} is more physically
realistic than~\eqref{eq:1-intro}, its combination with the energy
inequality~\eqref{eq:3-intro} turns out to provide an unsatisfactory
description of the solution at jumps (see the discussion
in~\cite[Sec.~5.2]{Mielke-Rossi-Savare08} and Remark~\ref{rmk:lack}
later on). In order to capture the jump dynamics, the energetic
behavior of the system in a jump regime has to be revealed. From this
perspective, it seems to be crucial to recover from~\eqref{eps-enid},
as $\eps \down 0$, an \emph{energy identity}, rather than an energy
inequality. Thus, the passage to the limit has to somehow keep track
of the limit of the term
\[
 \int_0^t
 \left(\frac\eps2 \snorm{u_\eps'(s)}^2 + \frac1{2\eps} \dist\big ({-}\mathrm{D}
   \cE_s (u_\eps(s)), \Kx \big)^2 \right) \d s\,,
\]
 which in fact  encodes the contribution of the viscous
 dissipation $\frac\eps2|u_\eps'|^2$,  completely missing in~\eqref{eq:3-intro}.

 \paragraph{\textbf{BV solutions.}} Moving from these considerations,
 it is natural to introduce the \emph{{\bipotential}} (which is
 related to the bipotential discussed in
 \cite{Buliga-deSaxce-Vallee08}, see Section
 \ref{subsec:bipotentials})  induced by $\Psi_\eps$, i.e.~the
 quantity
\begin{equation}
 \label{bipot}
\begin{aligned}
  \bipo(v,w) := \inf_{\eps>0}\Big(\Psi_\eps(v)+\Psi_\eps^*(w)\Big) & =\inf_{\eps>0} \left(\fnorm v +
    \frac\eps2\snorm v^2 +\frac1{2\eps}\dist^2 (w, \Kx ) \right)\\
  &=
  \fnorm v +
  \snorm v\,\dist (w, \Kx )
  \quad \text{for $v \in \xfin$, $w \in \xfin^*$}\,.
\end{aligned}
\end{equation}
Then,  the  $\eps$-energy identity~\eqref{eps-enid} 
  yields the inequality
\begin{equation}
\label{eps-enid-bipot}
 \ene t{u_\eps(t)}+\int_0^t \bipo\left( u_\eps'(s),-\mathrm{D}
  \cE_s (u_\eps(s)) \right) \d s
\le \ene0{u(0)} + \int_0^t \partial_t \ene{s}{u_\eps(s)} \d s\,,
 \end{equation}
 see Section~\ref{ss:3.1}.
Passing to the limit in~\eqref{eps-enid-bipot},  in
Theorem~\ref{thm:convergence1} we shall
prove that, 
 up to a
subsequence, the solutions $(u_\eps)$  of the viscous equation
\eqref{e:ris_intro-eps}  converge, as $\eps \down 0$, to a curve $u
\in \BV([0,T];\xfin)$ satisfying the local
stability~\eqref{loc-stab-intro} and the following \emph{energy
inequality}
\begin{equation}
    \label{eq:73-intro}
    \ene{t}{u(t)}+\pVar{{\bipcE}}u{0}{t}\le\ene{0}{u(0)}+
    \int_{0}^{t} \partial_t\ene s{u(s)}\,\d s\,.
\end{equation}
Without going into details (see Definition~\ref{def:Finsler_var} later
on), we may point out that \eqref{eq:73-intro} features a notion of
(pseudo)-total variation (denoted by $\pVarname{\bipcE}$) induced by
the {\bipotential} $\bipo$~\eqref{bipot} and the energy $\cE$.
The main novelty is that a $\BV$-curve obeying the local stability
condition \eqref{loc-stab-intro} always satisfies the opposite
inequality in \eqref{eq:73-intro}, thus yielding the energy balance
\begin{equation}
    \label{eq:186}
    \ene{t}{u(t)}+\pVar{{\bipo,\cE}}u{t}{t}=\ene{0}{u(0)}+
    \int_{0}^{t} \partial_t\ene s{u(s)}\,\d s\,.
    \tag{$\mathrm{E}_{\bipo,\cE}$}
\end{equation}
In fact, $\pVarname\bipcE$ provides a finer description of the
dissipation $\Delta_\bipcE$ of $u$, along a jump between two values
$u_-$ and $u_+$ at time $t$: it involves not only the quantity
$\fnorm{u_+-u_-}$ related to the dissipation potential
\eqref{e:norm-intro}, but also the viscous contribution induced by the
{\bipotential} $\bip$ through the formula
  \begin{equation}
    \label{eq:185}
    \begin{aligned}
      \Cost{\bipcE}t{u_-}{u_+}:=\inf\Big\{&\int_{r_0}^{r_1} \bipo(\dot
      \vartheta(r),-\rmD\ene t{\vartheta(r)})\,\d r:\\& \vartheta\in
      \AC([r_0,r_1];\xfin),\ \vartheta(r_0)=u_-,\
      \vartheta(r_1)=u_+\Big\}.
    \end{aligned}
  \end{equation}
  By a rescaling technique, it is possible to show that, in  a jump
  point, the system may switch to a \emph{viscous behavior}, which
  is in fact reminiscent of the viscous
  approximation~\eqref{e:ris_intro-eps}. In particular, when the jump
  point is of viscous type,
  the infimum in \eqref{eq:185} is attained and the states
  $u_-$ and $u_+$ are connected by some transition curve $\vartheta:
  [r_0,r_1]\to \xfin$, fulfilling the \emph{viscous} doubly nonlinear
  equation
\[
    \partial\Psi_0(\vartheta'(r)) +
    \vartheta'(r) +
    \mathrm{D}\ene{t}{\vartheta(r)} \ni 0 \quad \text{in
      $\xfin^*$} \quad \forae\, r \in (r_0,r_1)
      \]
  (in the case the norm $\snorm\cdot$ is Euclidean and we use its differential to identify $\xfin$ with
  $\xfin^*$).
 The combination
of~\eqref{loc-stab-intro} and~\eqref{eq:73-intro} yields the notion
of \emph{$\BV$ solution} to the rate-independent system~\VRIS. This
concept was first introduced in~\cite{Mielke-Rossi-Savare08}, in the
case the ambient space $\xfin$ is a finite-dimensional
\emph{manifold} $\mathcal{X}$, and both the rate-independent and the
viscous approximating dissipations depend on one single  Finsler
distance on $\mathcal{X}$. In this paper, while keeping to a Banach
framework,  we shall considerably broaden the class of
rate-independent and viscous dissipation functionals, cf.
Remark~\ref{ex:viscous}. Moreover, the notion of $\BV$ solution
shall be presented here in a more compact form than
in~\cite{Mielke-Rossi-Savare08}, amenable to a finer analysis and,
hopefully, to further generalizations.

Let us now briefly comment on our main results. First of all, we are
going to show in Theorems~\ref{thm:def-bv-solutions},
\ref{thm:filling-jumps}, and~\ref{thm:careful} that the concept of
$\BV$ rate-independent evolution completely encompasses the solution
behavior in both a purely rate-independent, non-jumping regime, and
in jump regimes, where the competition between dry-friction and
viscous effects is highlighted. Indeed,  from \eqref{loc-stab-intro}
and~\eqref{eq:73-intro} it is possible to deduce suitable energy
balances  at jumps (cf. conditions~\eqref{eq:67} in
Theorem~\ref{thm:def-bv-solutions}).

Then, in Theorem~\ref{thm:convergence1} we shall prove that, along a
subsequence, the viscous approximations arising
from~\eqref{e:ris_intro-eps} converge as $\eps \downarrow 0$ to a
$\BV$ solution. Next, our second main result,
Theorem~\ref{thm:convergence1bis}, states that, up to a subsequence,
also the discrete solutions $\Uu_{\tau,\eps}$ constructed via the
$\eps$-discretization scheme~\eqref{eq:58-intro-eps} converge to a
$\BV$ solution $u \in \BV([0,T];\xfin)$ of~\eqref{e:ris_1} as $\eps
\down 0$ and $\tau \down 0$ simultaneously, provided that the
respective convergence rates are such that
\[
 \lim_{\eps,\,\tau  \down 0} \frac{\eps}{\tau}=+\infty\,.
\]

Finally, in Section~\ref{sec:parametrized} we shall develop a
different approach to $\BV$ solutions, via the rescaling technique
advanced in~\cite{Efendiev-Mielke06} and refined
in~\cite{Mielke-Rossi-Savare08, Mielke-Zelik09}. The main idea is to
suitably reparametrize the approximate viscous curves $(u_\eps)$ in
order to capture, in the vanishing viscosity limit,  the viscous
transition paths at jumps points. This leads to performing an
asymptotic analysis as $\eps \down 0$ of the graphs of the functions
$u_\eps$,  in the extended phase space $[0,T]\times \xfin$. For
every $\eps>0$ the graph of $u_\eps$ can be parametrized by a couple
of functions $(\sft_\eps,\sfu_\eps) $, $\sft_\eps$ being the
(strictly increasing) rescaling function and $\sfu_\eps: = u_\eps
\circ \sft_\eps$ the rescaled solution. In
Theorem~\ref{thm:convergence2} we assert that, up to a subsequence,
the functions $(\sft_\eps,\sfu_\eps)$ converge as $\eps \down 0$ to
a \emph{parametrized rate-independent solution}. By the latter
terminology we mean a curve $(\sft,\sfu): [0,\sfS] \to [0,T] \times
\xfin$ fulfilling
\begin{subequations}
\label{mprif}
\begin{align}
\label{mprifa} &\begin{aligned}
&  \text{$\sft: [0,\sfS] \to [0,T]$ is nondecreasing, } \\
&  \sft'(s) + \fnorm{\sfu'(s)} >0 \quad \forae\, s \in
(0,\sfS),
\end{aligned}
\\[0.5em]
\label{mprifb} &\begin{aligned} & \left.\ba{@{}ccc}
    \sft'(s)>0&\Longrightarrow&
      \fnorm{-\mathrm{D} \cE_{\sft(s)} (\sfu(s))}\le 1
,\\
    \fnorm{\sfu'(s)}>0&\Longrightarrow&
    \fnorm{-\mathrm{D} \cE_{\sft(s)} (\sfu(s))}\ge 1
    \ea \right\}\quad \text{for a.a.\ }s \in (0,\sfS)\,,
\end{aligned}
\\[0.1em]
\intertext{and the \emph{energy identity}} \label{mprifc}
&\begin{aligned}
  \frac{\rmd}{\rmd s}\cE (\sft(s),\sfu(s) )
  &-\partial_t\cE (\sft(s),\sfu(s) )\, \sft'(s)
  \\ & = - \fnorm{\sfu'(s)} - \snorm{\sfu'(s)} \dist({-}\mathrm{D} \cE_{\sft(s)} (\sfu(s)),\Kx )
   \quad \forae \, s\in(0,\sfS)\,,
\end{aligned}
\end{align}
\end{subequations}
As already pointed out in~\cite{Efendiev-Mielke06,
Mielke-Rossi-Savare08}, like the notion of $\BV$ solution,
relations~\eqref{mprif} as well comprise both the purely
rate-independent evolution as well as the viscous transient regime
at jumps. The latter regime in fact corresponds to the case 
  $-\mathrm{D} \cE_{\sft} (\sfu)\not\in\Kx$
: the system
does not obey the local stability constraint~\eqref{loc-stab-intro}
any longer, and switches to viscous behavior, see also
Remark~\ref{rmk:mech} later on.

 As a matter of fact, Theorem~\ref{thm:equivalence} shows
that parametrized rate-independent solutions may be viewed as the
``continuous counterpart'' to $\BV$ evolutions.  With a suitable
transformation, it is possible to associate with every parametrized
rate-independent solution a $\BV$ one, and conversely. One
advantage of the parametrized notion is that it avoids  the
technicalities related to $\BV$ functions. Hence, it is for instance
more easily amenable to a stability analysis (cf. \cite[Rmk.
6]{Mielke-Rossi-Savare08}). Furthermore, in~\cite{Mielke-Zelik09}
 a highly refined vanishing
viscosity analysis has been developed, with this reparametrization
technique, in the  infinite-dimensional $(L^1, L^2)$-framework,
where~\eqref{e:ris_intro-eps} is replaced by a general quasilinear
evolutionary PDE.

\paragraph{\textbf{Generalizations and future developments.}}
So far we have focused on dissipation functionals of the
type~\eqref{e:norm-intro} and $\Psi_\eps(v)=\fnorm v+\frac\eps2\snorm
v^2$ as in~\eqref{e:ris_intro-eps} for expository reasons only, in
order to highlight the main variational argument leading to the notion
of $\BV$ solution.  Indeed, the analysis
developed in this paper is targeted to a general
\[
 \text{positively $1$-homogeneous, convex dissipation $\Diss:\xfin\to [0,+\infty)$,}
\]
(cf.\ \eqref{e:2.1}), and considers a fairly wide class of approximate
viscous dissipation functionals $\Psi_\eps$, defined by
conditions~\eqref{eq:19}--\eqref{eq:21} in Section~\ref{ss:2.3}.
Furthermore, at the price of just technical complications, our results
could be extended to the case of a \emph{Finsler-like} family of
dissipation functionals $\Diss(u,\cdot)$, depending on the state
variable $u\in \xfin$, and satisfying uniform bounds and
Mosco-continuity with respect to $u$, see
\cite[Sect.\,2]{Mielke-Rossi-Savare08} and
\cite[Sect.~6,\,8]{Rossi-Mielke-Savare08}.

The extension to \emph{infinite-dimensional} ambient spaces and
\emph{nonsmooth} energies is  crucial for  application of the
concept of $\BV$ solution to the PDE systems modelling
rate-independent evolutions in continuum mechanics. A first step in
this direction is to generalize the known existence results for
doubly nonlinear equations, driven by a viscous dissipation, to
nonconvex and nonsmooth energy functionals
in infinite dimensions. As shown in~\cite{Rossi-Savare06,
Rossi-Mielke-Savare08}, in the nonsmooth and nonconvex case one can
replace the energy differential $\mathrm{D}\cE_t$ with a suitable
notion of \emph{subdifferential} $\partial \cE_t $. Accordingly,
instead of continuity of $\mathrm{D}\cE_t$, one asks for closedness of
the multivalued subdifferential $\partial \cE_t $ in the sense of
graphs. These ideas shall be further advanced in the forthcoming
work~\cite{Mielke-Rossi-Savare-viscous_progress09}. Therein,
exploiting  techniques from nonsmooth analysis, we shall also tackle
energies which do not depend smoothly on time (this is relevant for
rate-independent applications, see e.g.~\cite{KnMiZa08?CPPM}
and~\cite{KreLie07?RIKP}).

On the other hand, the requirement that the ambient space is
finite-dimensional could be replaced by suitable compactness (of the
sublevels of the energy) and reflexivity assumptions on the ambient
space $\xfin$. The latter topological requirement in fact ensures
that $\xfin$ has the so-called Radon-Nikod\'ym property, i.e.~that
absolutely continuous curves with values in $\xfin$ are almost
everywhere differentiable. The vanishing viscosity analysis in
spaces which do not enjoy this property requires a subtler approach,
involving metric arguments (see e.g.\ \cite[Sect.
7]{Rossi-Mielke-Savare08}), or ad-hoc stronger estimates
\cite{Mielke-Zelik09}.  See also \cite{Miel08?DEMF} for some
preliminary approaches to BV solutions for PDE problems.


\subsection*{Plan of the paper}
 Section~\ref{s:2} is devoted to an extended presentation of
 energetic and local solutions to rate-independent systems. In
 particular, after fixing the setup of the paper in
 Section~\ref{subsec:setting},  in Sec.~\ref{ss:2.2} we recall the
 definition of global energetic solution, show its differential
 characterization and the related  variational time-incremental
 scheme. We develop  the vanishing-viscosity approach in
 Secs.~\ref{ss:2.3}  and~\ref{ss:2.1}, thus arriving at the notion
 of local solution (see Section~\ref{ss:2.5}), which also admits a
 differential characterization.

In Section~\ref{s:3} we  introduce the concept of {\bipotential} and
thoroughly analyze its properties, as well as the induced
(pseudo)-total variation.   With these ingredients, in
Sec.~\ref{sec:BV-solutions}  we present the notion of $\BV$
solution. We show that $\BV$ rate-independent evolutions admit, too,
a differential characterization, and, in Sec.~\ref{ss:4.2},  that
they provide a careful description of the energetic behavior of the
system. Then, in Section~\ref{subsec:viscous-limit}, we state our
main results on $\BV$ solutions.

 While Section~\ref{sec:parametrized} is focused on the alternative notion
 of
 parametrized rate-independent solutions, the last Sec.~\ref{sec:technical} contains some
technical results which lie at the core of our theory.



\section{Global energetic versus  local solutions, and their viscous regularizations}
\label{s:2}
 \noindent In this section, we will briefly recall the notion of \emph{energetic solutions}
and   show that  their viscous regularizations give raise to
\emph{local solutions}.

\subsection{Rate-independent setting: dissipation and energy functionals}
\label{subsec:setting}
We let
\[
(\xfin,\|\cdot\|_\xfin) \ \ \ \text{be a finite-dimensional normed
vector space,}
\]
endowed with a gauge function $\Diss$, namely a
\begin{equation}
   \label{e:2.1}
   \text{non-degenerate, positively $1$-homogeneous, convex dissipation $\Diss:\xfin\to [0,+\infty)$,}
\end{equation}
 i.e.~$\, \Diss$ satisfies  $\Diss(v)>0$ if $v\neq 0$, and
 \begin{displaymath}
   \Diss(v_1+v_2)\le \Diss(v_1)+\Diss(v_2),\quad
   \Diss(\lambda v)=\lambda \Diss(v)
   \quad
   \text{for every }\lambda \ge0,\ v,v_1,v_2\in \xfin.
 \end{displaymath}
 In particular, there exists a constant $\eta>0$ such that
 \begin{displaymath}
   \eta^{-1}\|v\|_{\xfin}\le \Diss(v)\le \eta\|v\|_{\xfin}\quad\text{for every }v\in \xfin.
 \end{displaymath}
 Since $\Diss$ is $1$-homogeneous, its
 subdifferential $\partial\Diss\, : \xfin \rightrightarrows \xfins$ can be characterized by
\begin{equation}
  \label{eq:18}
  \partial\Diss(v):=\Big\{w\in \xfin:
  \la w,z\ra \le \Diss(z)\text{ for every }z\in \xfin,\quad
  \la w,v\ra=\Diss(v)\Big\}\subset \xfins;
\end{equation}
$\partial\Diss$ takes its values
in the convex set $\Kx\subset \xfins$, given by
\begin{equation}
  \label{e:2.3}
  \Kx=\partial\Diss(0):=\big\{w\in \xfins\,: \langle w,z\rangle\le \Psi_0(z)\quad\forall\, z\in \xfin\}\supset
  \partial\Diss(v)\quad\text{for every }v\in \xfin,
\end{equation}
which enjoys some useful (and well-known, see e.g.\
\cite{Rockafellar70}) properties. For the reader's convenience we
list them here:
\begin{enumerate}[\bf K1.]
\item
  $\Kx$ is the proper domain of the Legendre transform $\DualDiss $ of
  $\Diss$, since
  \begin{equation}
    \label{eq:104}
    \DualDiss (w)=\rmI_{\Kx}(w)=
    \begin{cases}
      0&\text{if }w\in \Kx,\\
      +\infty&\text{otherwise.}
    \end{cases}
  \end{equation}
\item $\Diss$ is the support function of $\Kx$, since
  \begin{equation}
    \label{eq:38}
    \Diss(v)=\sup_{w\in \Kx}\la w,v\ra\quad\text{for every }v\in \xfin,
  \end{equation}
  and $\Kx$ is the polar set of the unit ball $K:=\big\{v\in \xfin:
  \Diss(v)\le 1\big\}$ associated with $\Diss$.
\item $\Kx$ is the unit ball
  of the support function $\Psi_{0*}$ of $K$:
  \begin{equation}
    \label{eq:105}
    \Kx=\big\{w\in \xfin^*:\Psi_{0*}(w)\le 1\big\}, \quad
    \text{with}\quad
    \Psi_{0*}(w)=\sup_{v\in K} \la w,v\ra=\sup_{v\neq 0}\frac{\la w,v\ra}{\Diss( v)}.
  \end{equation}
\item In the even case (i.e., when $\Diss(v) = \Diss({-}v)$ for all $v \in \xfin$),  we have that
 $\Diss$ is an
  equivalent norm for $\xfin$, $\Psi_{0*}$ is its dual norm, $K$ and
  $\Kx$ are their respective unit balls.
\end{enumerate}
Further, we consider a \emph{smooth} energy functional
\[
\cE \in {\rm C}^1 ([0,T] \times \xfin)\,,
\]
which we suppose
bounded from below and with energy-bounded time derivative
\begin{equation}
  \label{assene1} \exists\, C >0  \ \forall\, (t,u) \in
[0,T] \times \xfin\, :  \qquad \ene{t}{u} \geq -C\,,\qquad
|\partial_t\ene{t}{u}| \leq  C \left( 1+
  \ene{t}{u}^+\right),
\end{equation}
where $(\cdot)^+$ denotes the positive part.
The rate-independent
system associated with the energy functional $\cE$ and the
dissipation potential $\Diss$ can be formally described by the
\emph{rate-independent doubly nonlinear} differential inclusion
\begin{equation}
\label{e:ris_1} \partial\Diss (u'(t)) + \mathrm{D} \cE_t (u(t)) \ni 0
\quad \text{in $\xfins$}\quad
\text{for a.a. $t \in (0,T).$}
\tag{DN$_0$}
\end{equation}

 As already mentioned in the Introduction, for nonconvex energies
solutions to \eqref{e:ris_1} may exhibit discontinuities in time.
The first weak solvability notion for~\eqref{e:ris_1} is the concept
of \emph{(global)
 energetic solution}   to the rate-independent
system~\eqref{e:ris_1} (see~\cite{Mielke-Theil-Levitas02,
MieThe99MMRI, MieThe04RIHM} and the survey~\cite{Mielke05}), which
we recall in the next section.
\subsection{Energetic solutions and variational incremental scheme}
\label{ss:2.2}
\begin{definition}[Energetic solution]
  \label{def:energetic}
  A curve $u\in \BV([0,T];\xfin)$ is an energetic solution of the \emph{rate independent system}
  {\RIS}
  if for all $t\in [0,T]$ the \emph{global stability} (S) and the
  \emph{energy balance} (E) holds:
\begin{equation}
  \label{eq:1}
  \forall\, z\in \xfin:\qquad \ene t{u(t)}\le \ene tz+\Psi_0(z-u(t)),
  \tag{S}
\end{equation}
\begin{equation}
  \label{eq:2}
  \ene{t}{u(t)}+\Var{\Psi_0}u{0}{t}=\ene{0}{u(0)}+
  \int_{0}^{t} \partial_t\ene s{u(s)}\,\d s.
  \tag{E}
\end{equation}
\end{definition}
\paragraph{\textbf{$\BV$ functions.}}
  Hereafter,  we shall  consider functions of bounded variation \emph{pointwise defined
  in every point $t\in [0,T]$},
  such that the \emph{pointwise} total variation
  with respect to $\Diss$ (any equivalent norm of $\xfin$ can be chosen)
  $\Var{\Psi_0}u0T$ is finite, where
   \[
    \Var{\Psi_0}uab:=
    \sup\Big\{\sum_{m=1}^M\Psi_0\big(u(t_m)-u(t_{m-1})\big):a=t_0< t_1<\cdots<t_{M-1}<t_M=b\Big\}.
    \]
 Notice that a function $u$ in $\BV([0,T];\xfin)$ admits left and right limits at every  $t\in [0,T]:$
  \begin{equation}
    \label{eq:41}
    u(t_-):=\lim_{s\up t}u(s),\ \
    u(t_+):=\lim_{s\down t}u(s),\ \ \text{with the convention }u(0_-):=u(0),\ u(T_+):=u(T),
  \end{equation}
  and its \emph{pointwise} jump set $\rmJ_u$ is the at most countable set defined
  by
  \begin{equation}
    \label{eq:88}
    \rmJ_u:=\big\{t\in [0,T]:u(t_-)\neq u(t)\text{ or }u(t)\neq u(t_+)\big\}\supset
    \essJ_u:=
    \big\{t\in [0,T]:u(t_-)\neq u(t_+)\big\}.
  \end{equation}
 We denote by $u'$ the distributional derivative  of
$u$, and recall that $u'$  is a Radon vector measure with finite
total variation $|u'|$. It is well known
\cite{Ambrosio-Fusco-Pallara00} that $u'$ can be decomposed into the
sum of the three mutually singular measures
\begin{equation}
  \label{eq:87}
  u'=u'_\Le+u'_\Ca+u'_\rmJ,\quad u'_\Le=\dot u\,\Leb 1,\quad
  u'_\co:=u'_\Le+u'_\Ca\,.
\end{equation}
Here,  $u'_\Le$ is the absolutely continuous part with
respect to the Lebesgue measure $\Leb1$, whose Lebesgue density
$\dot u$ is the usual pointwise (and $\Leb 1$-a.e.\ defined)
derivative, $u'_\rmJ$ is a discrete measure concentrated on
$\essJ_u\subset \rmJ_u$, and $u'_\Ca$ is the so-called Cantor part,
still satisfying $u'_\Ca(\{t\})=0$ for every $t\in [0,T]$. Therefore
$u'_\co=u'_\Le+u'_\Ca$ is the diffuse part of the measure, which
does not charge $\rmJ_u$. 
  In the following, it will be useful to use a nonnegative and diffuse reference measure $\mu$ on $(0,T)$
  such that $\Leb 1$ and $u'_\Ca$ are absolutely continuous w.r.t.\ $\mu$: just to fix our ideas, we set
  \begin{equation}
    \label{eq:156}
    \mu:=\Leb 1+|u'_\Ca|.
  \end{equation}
  With a slight abuse of notation, for every $(a,b)\subset (0,T)$ we denote by $\Mint ab{ \Psi_0}{u'_\co}$
  the integral
\begin{equation}
  \label{eq:187}
  \Mint ab{ \Psi_0}{u'_\co}:=\int_a^b \Psi_0\left(\frac{\d u'_\co}{\d \mu}\right)\,\d \mu=
  \int_a^b \Psi_0(\dot u)\,\d \Leb 1+
  \int_a^b \Psi_0\left(\frac{\d u_\Ca'}{\d|u_\Ca'|}\right)\,\d |u_\Ca'|.
\end{equation}
Since $\Psi_0$ is $1$-homogeneous, the above integral  is
independent of $\mu$, provided $u_\co'$ is absolutely continuous
w.r.t.\ $\mu$. 
\paragraph{\textbf{Towards a differential characterization of energetic solutions.}}
Let us first of all point out that \eqref{eq:1} is stronger than the
local stability condition
\begin{equation}
  \label{eq:65}
  \tag{$\mathrm{S}_{\mathrm{loc}}$}
  -\rmD\ene t{u(t)}\in K^*\quad\text{for every } t\in [0,T]\setminus \rmJ_u,
\end{equation}
which can be formally deduced from  \eqref{e:ris_1} and
\eqref{e:2.3}. Indeed, the global stability  \eqref{eq:1} yields for
every $z=u(t)+hv\in \xfin$ and $h>0$
\[
  \la-\rmD\ene t{u(t)},h v\ra+o(|h|)\le \ene t{u(t)}-\ene t{u(t)+hv}\le h\Psi_0(v)
  \]
and therefore, dividing by $h$ and passing to the limit as
$h\downarrow0$, one gets
 \[
  \la -\rmD\ene t{u(t)},v\ra\le \Psi_0(v)\quad\text{for every } z\in \xfin,
  \]
so that \eqref{eq:65} holds.
We obtain more insight into \eqref{eq:2} by representing the
$\Psi_0$ variation $\Var{\Psi_0}uab$ in terms of the
distributional derivative $u'$ of $u$.
In fact, recalling \eqref{eq:156} and \eqref{eq:187}, we have
 \[
  \Var{\Psi_0}uab:=
  \Mint ab{\Psi_0}{u'_\co}+\JVar{\Psi_0}uab,
  \]
where the jump contribution $\JVar{\Psi_0}uab$ can be described, in
terms of the quantities
\begin{equation}
  \label{eq:64}
  \cost{\Psi_0}{v_0}{v_1}:=\Psi_0(v_1-v_0),\qquad
  \tricost{\Psi_0}{v_-}{v}{v_+}:=\Psi_0(v-v_-)+\Psi_0(v_+-v),
\end{equation}
by
\begin{equation}
  \label{eq:37}
   \JVar{\Psi_0}uab:=
    \cost{\Psi_0}{u(a)}{u(a_+)}+\cost{\Psi_0}{u(b_-)}{u(b)}+
    \kern-6pt\sum_{t\in\rmJ_u\cap (a,b)}\kern-6pt\tricost{\Psi_0}{u(t_-)}{u(t)}{u(t_+)}.
\end{equation}
Also notice that, as usual in rate-independent evolutionary
problems, $u$ is pointwise \emph{everywhere} defined and the jump
term $\JVar{\Psi_0}u\cdot\cdot$ takes into account the value of $u$
at every time $t\in J_u$. Therefore, if $u$ is not continuous at
$t$, this part may  yield a strictly bigger contribution than the
total mass of the distributional jump measure $u'_\rmJ$ (which gives
rise to the so-called \emph{essential} variation).

  The following result provides an equivalent characterization of
  energetic solutions: besides the global stability condition
  \eqref{eq:1}, it involves a $\BV$ formulation of the differential inclusion
  \eqref{e:ris_1}  (cf.\  the
  \emph{subdifferential formulation} of~\cite{MieThe04RIHM}) and a
  \emph{jump condition} at any jump point of~$u$.
\begin{proposition}
 \label{le:global-diff}
   A curve $u\in \BV([0,T];\xfin)$ \emph{satisfying the global
     stability condition \eqref{eq:1}} is an energetic solution of the
   rate-independent system {\RIS} if and only if it satisfies the
   differential inclusion
   \begin{equation}
     \label{eq:56}
     \partial\Psi_0\Big(\frac {\d u'_\co}{d \mu}(t)\Big)+\rmD\ene t{u(t)}\ni 0\quad \text{for $\mu$-a.e.\ $t\in [0,T]$},
     \quad \mu:=\Leb 1+|u_\Ca'|,
     \tag{DN$_{0,\BV}$}
   \end{equation}
and the jump conditions
\begin{equation}
  \label{eq:57}
  \begin{gathered}
    \ene{t}{u(t)}-\ene t{u(t_-)}=-\cost{\Psi_0}{u(t_-)}{u(t)},\quad
    \ene{t}{u(t_+)}-\ene t{u(t)}=-\cost{\Psi_0}{u(t)}{u(t_+)},\\
    \ene{t}{u(t_+)}-\ene t{u(t_-)}=-\cost{\Psi_0}{u(t_-)}{u(t_+)}.
  \end{gathered}
  \tag{J$_{\text{ener}}$}
\end{equation}
for every $t\in\rm J_u$ (recall convention \eqref{eq:41} in the case
$t=0,T$).
\end{proposition}
\noindent We shall simply sketch the proof, referring to the
arguments  for the forthcoming Proposition~\ref{le:local-diff} for
all details.
\begin{proof}
By the additivity property of the total variation $\Var{\Psi_0}
u\cdot\cdot$, \eqref{eq:2} yields for every $0\le t_0<t_1\le T$
\begin{equation}
  \label{eq:74}
  \Var{\Psi_0}u{t_0}{t_1}+\ene{t_1}{u(t_1)}=\ene{t_0}{u(t_0)}+
  \int_{t_0}^{t_1} \partial_t\ene t{u(t)}\,\d t\,.
  \tag{E'}
\end{equation}
Arguing as in the proof of Proposition~\ref{le:local-diff} later on,
one can see that the global stability \eqref{eq:1} and \eqref{eq:74}
yield the differential inclusion~\eqref{eq:56} and
conditions~\eqref{eq:57}.

Conversely, repeating the arguments of Proposition~\ref{le:local-diff}
one can verify that \eqref{eq:56} and \eqref{eq:57} imply
\eqref{eq:2}.
\end{proof}

\paragraph{\textbf{Incremental minimization scheme}}
Existence of energetic solutions can be proved by solving a minimization
scheme, which is also interesting as construction of an
effective approximation of the solutions.

For a given time-step $\tau>0$ we consider a uniform partition (for
simplicity) $0=t_0<t_1<\cdots<t_{N-1}< T\le t_N$, $t_n:=n\tau$, of
the time interval $[0,T]$, and an initial value $\Uu^0_\tau\approx
u_0$. In order to find good approximations of $\Uu^n_\tau\approx
u(t_n)$ we solve the incremental minimization scheme
\begin{equation}
  \label{eq:58}
  \text{find $\Uu^1_\tau,\cdots,\Uu^N_\tau$}\quad\text{such that}\quad
  \Uu^n_\tau\in \argmin_{\Uu \in \xfin}\Big\{\Psi_0(\Uu-\Uu^{n-1}_\tau)+\ene{t_n}\Uu\Big\}.
  \tag{IP$_0$}
\end{equation}
Setting
\begin{equation}
  \label{eq:59}
  \overline{\Uu}_\tau(t):=\Uu^n_\tau\quad\text{if }t\in (t_{n-1},t_n],
\end{equation}
it is possible to find a suitable vanishing sequence of step sizes
$\tau_k\downarrow0$ (see, e.g., \cite{MieThe04RIHM,Mielke05} for all
calculations),
 such that
 \[
  \exists \ \lim_{k\to+\infty}\overline{\Uu}_{\tau_k}(t)=:u(t)\quad \text{for every }
     t\in [0,T],
     \]
and $u$ is an energetic solution of \eqref{e:ris_1}.

\subsection{Viscous approximations of rate-independent systems}
\label{ss:2.3} In the present paper we want to study a different
approach to approximate and solve \eqref{e:ris_1}: the main idea is
to replace the linearly growing dissipation potential $\Psi_0$ with
a suitable \emph{convex and superlinear} ``viscous'' regularization
$\Psi_\eps:\xfin\to[0,+\infty)$ of $\Diss$, depending on a ``small''
parameter $\eps>0$ and ``converging'' to $\Diss$ in a suitable sense
as $\eps\down0$. Solving the doubly nonlinear differential inclusion
(we use the notation $\dot u$ for the time derivative when $u$ is
absolutely continuous)
\begin{equation}
  \label{e:ris_eps} \partial\Psi_\eps(\dot u_\eps(t))
  + \mathrm{D} \cE_t (u_\eps(t)) \ni 0 \quad \text{in $\xfins$ \quad for a.a. $t \in (0,T),$}
  \tag{DN$_\eps$}
\end{equation}
one can consider the sequence  $(u_\eps)$ as
a good approximation of the solution $u$ of \eqref{e:ris_1} as
$\eps\down0$.

There is also a natural discrete counterpart to \eqref{e:ris_eps},
which regularizes the incremental minimization problem
\eqref{eq:58}. We simply substitute $\Psi_0$ by $\Psi_\eps$ in
\eqref{eq:58}, recalling that now the time-step $\tau$ should
explicitly appear, since $\Psi_\eps$ is not $1$-homogeneous any
longer. The viscous incremental problem is therefore
\begin{equation}
  \label{eq:61}
  \text{find $\Uu^1_{\tau,\eps},\cdots,\Uu^N_{\tau,\eps}$}\quad
  \text{such that}\quad
  \Uu^n_{\tau,\eps}\in \argmin_{\Uu \in \xfin}
  \Big\{\tau\Psi_\eps\Big(\frac{\Uu-\Uu^{n-1}_{\tau,\eps}}\tau\Big)+\ene{t_n}\Uu\Big\}.
  \tag{IP$_\eps$}
\end{equation}
Setting as in \eqref{eq:59}
 \[
  \overline{\Uu}_{\tau,\eps}(t):=\Uu^n_{\tau,\eps}\quad\text{if }t\in (t_{n-1},t_n],
\]
one can study the limit of the discrete solutions when $\tau\down0$
and $\eps\down0$, under some restriction on the behavior of the
quotient $\eps/\tau$ (see Theorem~\ref{thm:convergence1bis} later
on).

\paragraph{\textbf{The choice of the viscosity approximation $\Psi_\eps$.}}
Here we consider the particular case when
the potential $\Psi_\eps$ can be obtained starting from a given
\begin{equation}
  \label{eq:19}
  \text{convex function $\Psi:\xfin\to [0,+\infty)$}\quad
  \text{such that}\quad
  \Psi(0)=0,\quad
  \lim_{\|v\|_\xfin\uparrow+\infty}\frac{\Psi(v)}{\|v\|_\xfin}=+\infty,
  \tag{$\Psi.1$}
\end{equation}
by the canonical rescaling
\begin{equation}
  \label{eq:20}
  \Psi_\eps(v):=\eps^{-1}\Psi(\eps v)\quad \text{for every }v\in \xfin,\ \eps>0,
  \tag{$\Psi.2$}
\end{equation}
and $\Psi_\eps$ is linked to $\Diss$ by the relation
\begin{equation}
  \label{eq:21}
  \Diss(v)=\lim_{\eps\down0}\Psi_\eps(v)=\lim_{\eps\down0}\eps^{-1}\Psi(\eps v)\quad
  \text{for every }v\in \xfin.
  \tag{$\Psi.3$}
\end{equation}
\begin{remark}
\label{rem:prop-psi} Notice that, by convexity of $\Psi$ and the
fact that $\Psi(0)=0$, the map $\eps \mapsto \eps^{-1}\Psi(\eps v)$
is nondecreasing for all $v \in \xfin$. Hence,
\begin{equation}
\label{eq:21-BIS} \Diss(v) \leq \Psi_\eps(v) \quad \text{for all $v
\in \xfin$, for all $\eps>0$.}
\end{equation}
Furthermore, by the coercivity condition \eqref{eq:19},
 \[
  \partial\Psi_\eps(v):=\partial\Psi(\eps v)\quad\text{is a surjective map}.
  \]
\end{remark}
Here are some examples, showing that \eqref{eq:20} still provides a
great flexibility and covers  several interesting cases.
\begin{example}
  \label{ex:viscous}
  \
  \begin{mydescription}
  \item[$\Psi_0$-viscosity]   The simplest example, still absolutely non trivial
    \cite{Mielke-Rossi-Savare08},
    is to consider
    \begin{equation}
      \label{eq:23}
      \begin{gathered}
      \Psi(v):=\Diss(v)+\frac 12\big( \Diss(v)\big)^2,\quad
      \Psi_\eps(v):=\Diss(v)+\frac \eps2\big(\Diss(v)\big)^2,\\
      \partial\Psi_\eps(v)=\Big(1+\eps\Diss(v)\Big)\partial\Diss(v).
      \end{gathered}
    \end{equation}
    A similar regularization can be obtained by choosing a real convex
    and superlinear function $F_V:[0,+\infty)\to [0,+\infty)$,
    with $F_V(0)=F_V'(0)=0$, and setting
    \begin{equation}
      \label{eq:92}
      \Psi(v):=\Diss(v)+F_V(\Diss(v))=F(\Diss(v)),\quad \text{with} \ \  F(r):=r+F_V(r).
    \end{equation}
  \item[Quadratic or $p$-viscosity induced by a norm $\|\cdot\|$] The most interesting case
    involves an arbitrary norm $\|\cdot\|$ on $\xfin$ and considers for $p>1$
    \begin{equation}
      \label{eq:24}
      \Psi(v)=\Diss(v)+\frac 1p\|v\|^p,\quad
      \Psi_\eps(v)=\Diss(v)+\frac{\eps^{p-1}} p\|v\|^p,\quad
      \partial\Psi_\eps(v)=\partial\Diss(v)+\eps^{p-1} J_p(v),
    \end{equation}
    where $J_p$ is the $p$-duality map associated with $\|\cdot\|$.
    In particular, if $\|\cdot\|$ is a Hilbertian norm and $p=2$,
    then $J_2$ is the Riesz isomorphism and
    we can choose $J_2(v)=v$ by identifying $\xfin$ with $\xfins$.
    Hence,
    \eqref{e:ris_eps} reads
     \[
      \partial\Psi_\eps(\dot u_\eps(t))+\eps \dot u_\eps(t)+\rmD\ene
      t{u_\eps(t)}\ni0 \quad \text{in $\xfins$ \quad for a.a. $t \in (0,T),$}
      \]
    and the incremental problem \eqref{eq:61} looks for $U^n_{\tau,\eps}$ which recursively minimizes
     \[
      U\mapsto \Psi_0(U-U^{n-1}_{\tau,\eps})+\frac \eps{2\tau}\|U-U^{n-1}_{\eps,\tau}\|^2+\ene{t_n}U.
      \]
    This is the typical situation which motivates our investigation.
  \item[Additive viscosity]
    More generally, we can choose a convex ``viscous'' potential
    $\Psi_V:\xfin\to [0,+\infty)$ satisfying
    \begin{equation}
      \label{eq:25}
      \lim_{\eps\downarrow0}\eps^{-1}\Psi_V(\eps v)=0,\quad
      \lim_{\lambda\up+\infty}\lambda^{-1}\Psi_V(\lambda v)=+\infty\quad
      \text{for all $\, v\in \xfin,$}
    \end{equation}
    and set
    \begin{equation}
      \label{eq:5}
      \Psi(v):=\Diss(v)+\Psi_V(v),\quad
      \Psi_\eps(v):=\Diss(v)+\eps^{-1}\Psi_V(\eps v),\quad
      \partial\Psi_\eps(v)=\partial\Diss+\partial\Psi_V(\eps v).
    \end{equation}
  \end{mydescription}
\end{example}

\subsection{Viscous energy identity}
\label{ss:2.1}
Since $\Psi$ has a superlinear growth,
the results of~\cite{Colli-Visintin90, Colli92} ensure
that for every $\eps>0$ and initial datum $u_0\in \xfin$ there exists at least one solution
 $u_\eps \in \AC([0,T];\xfin)$ to equation~\eqref{e:ris_eps}, fulfilling the Cauchy condition $u_\eps(0)=u_0$.

In order to capture its asymptotic behavior as $\eps\down0$, we
split equation \eqref{e:ris_eps} in a simple system of two
conditions,
 involving an auxiliary variable $w_\eps:
[0,T]\to \xfins$ and a scalar function $\pt_\eps:[0,T]\to \R$
\begin{subequations}
  \label{eq:7}
  \begin{align}
    \label{eq:9}
    \partial\Psi_\eps(\dot u_\eps(t))&\ni w_\eps \qquad \forae\, t \in (0,T)\,,\\
    \label{eq:10}
    \rmD\ene t{u_\eps(t)}&=-w_\eps(t),\qquad
    \partial_t\ene t{u_\eps(t)}=-\pt_\eps(t) \qquad \text{for all $t \in [0,T]$.}
  \end{align}
\end{subequations}
Denoting by  $\Psi^*,\Psi_\eps^*$ the conjugate functions of $\Psi$
and $\Psi_\eps$, we have
\begin{equation}
  \label{eq:6}
  0=\Psi^*(0)\le \Psi^*(\xi)<+\infty,
  \qquad
  \Psi_\eps^*(\xi)=\eps^{-1}\Psi^*(\xi)\quad
  \text{for every }\xi\in \xfin^*.
\end{equation}
Due to~\eqref{eq:21-BIS}, there holds
\begin{equation}
\label{eq:21-TER}  \Psi_\eps^*(\xi) \leq \Psi_0^*(\xi) \quad
\text{for all $v \in \xfin$, $\eps>0$.}
\end{equation}
 The classical characterization of
the subdifferential of $\Psi_\eps$ yields that the first condition
\eqref{eq:9} is equivalent to
\begin{equation}
  \label{eq:8}
  \Psi_\eps(\dot u_\eps(t))
  + \Psi_\eps^* (w_\eps(t)) = \langle w_\eps(t), \dot u_\eps(t)
  \rangle \qquad \forae\, t \in (0,T)\,.
\end{equation}
On the other hand, the chain rule for the $\mathrm{C}^1$ functional
$\cE$ shows that along the absolutely continuous curve $u_\eps$
\begin{equation}
  \label{e:classical-chain_rule}
  \frac{\rm d}{\rm dt}\ene{t}{u_\eps(t)}=\langle \mathrm{D}
  \ene{t}{u_\eps(t)}, {\dot{u}_\eps}(t) \rangle+
  \partial_t \ene{t}{u_\eps(t)}=
  -\la w_\eps(t),\dot u_\eps(t)\ra-\pt_\eps(t) \ \  \text{for a.a.~$t \in (0,T)$.}
\end{equation}
Thus, if $w_\eps(t)=-\mathrm{D} \ene{t}{u_\eps(t)} $,
equation~\eqref{eq:9} is equivalent to  the \emph{energy identity}
\begin{equation}
\label{e:enid-eps}
\begin{aligned}
  \int_{t_0}^{t_1} \Big(\Psi_\eps
  \left({\dot u_\eps}(r)\right)
  +\Psi_\eps^*
  \left(w_\eps(r)\right)+\pt_\eps(r)\Big) \dd r +\ene {t_1}{u_\eps(t_1)}
  &= \ene{t_0}{u_\eps(t_0)},
  \end{aligned}
\end{equation}
for every $ 0 \leq t_0 \leq t_1 \leq T.$
\begin{remark}[The role of $\Psi^*_\eps$]
  \label{rem:ex}
  \upshape
  In the general, \emph{additive-viscosity} case (see~\eqref{eq:5}),
  when $\Psi(v)=\Diss(v)+\Psi_V(v)$ the inf-sup convolution formula
  yields
   \[
    \Psi_\eps^*(\xi)= \inf_{\stackrel{\scriptstyle \xi_1+\xi_2=\xi}{\xi_1,\xi_2 \in \xfins\vphantom{\big(}}}
    \left\{\rmI_{\Kx} (\xi_1) +\frac1{\eps}\Psi_V^*(\xi_2)\right\}=\eps^{-1}
    \min_{z \in \Kx}  \Psi_V^*(\xi -z).
    \]
  In particular, when $\Psi_V(\xi):=\frac 12|v|^2$ for some norm $|\cdot|$ of $\xfin$,
  one finds
   \[
    \Psi_\eps^*(\xi)=\frac 1{2\eps} \min_{z \in \Kx}  |\xi -z|_*^2,
    \]
  where $|\cdot|_*$ is the dual norm of $|\cdot|$. Thus, for all
  $\xi \in \xfin^*$ the functional
  $\Psi_\eps^*(\xi)$ is  the squared distance of $\xi$ from $\Kx$, with
  respect to $|\cdot|_*$. This shows that, in the viscous regularized equation
  \eqref{e:ris_eps},
  the (local) stability condition $w(t)=-\rmD\ene t{u(t)}\in \Kx$ has been replaced by
  the contribution of the penalizing term
  \begin{displaymath}
    \frac 1{2\eps}\int_0^T \min_{z\in \Kx}|-\rmD \ene t{u_\eps(t)}-z|_*^2\,\d t
  \end{displaymath}
  in the energy identity \eqref{e:enid-eps}.
\end{remark}

\subsection{Pointwise limit of viscous approximations and local solutions}
\label{ss:2.5} \label{subsec:pointwise}
Using~\eqref{assene1}, it is not difficult to show that the viscous
solutions $u_\eps$ of \eqref{e:ris_eps} satisfy the \emph{a priori}
bound
\begin{equation}
\label{e:est-1}
 \int_0^T \Big(\Psi_\eps(\dot u_\eps(t))+\Psi_\eps^*(w_\eps(t))\Big)\,\dd t\le C,
 \quad \text{with} \ \  w_\eps(t)=-\rmD\ene t{u_\eps(t)} \ \ \text{for all $t \in [0,T]$.}
\end{equation}
Therefore,  Helly's compactness theorem shows that, up to the
extraction of a suitable subsequence,  the sequence $(u_\eps)$
pointwise converges to a $\BV$ curve $u$.  From the convergence
$w_\eps(t)\to w(t)=-\rmD\ene t{u(t)}$  as $\eps \down 0$ and the
fact that for all $t \in [0,T]$
\begin{equation}
  \label{eq:70}
  \liminf_{\eps\downarrow0}\eps^{-1}\Psi^*(w_\eps(t))\topref{eq:6}\ge \Psi_0^*(w(t))=\rmI_\Kx(w(t))=
  \begin{cases}
    0&\text{if }w(t)\in K^*,\\
    +\infty&\text{otherwise,}
  \end{cases}
\end{equation}
we infer that the limit curve  $u$ satisfies the (local) stability
condition \eqref{eq:65}. On the other hand, passing to the limit in
\eqref{e:enid-eps} one gets the energy inequality
\begin{equation}
  \label{eq:2bis}
  \ene{t_1}{u(t_1)}+\Var{\Psi_0}u{t_0}{t_1}\le \ene{t_0}{u(t_0)}+
  \int_{t_0}^{t_1} \partial_t\ene t{u(t)}\,\d t
  \quad\text{for }0\le t_0<t_1\le T.
  \tag{E$'_{\text{ineq}}$}
\end{equation}
The above discussion motivates the concept of \emph{local solution}
  (see also~\cite[Sec.~5.2]{Mielke-Rossi-Savare08} and the references
  therein).
\begin{definition}[Local solutions]
  \label{def:local}
  A curve $u\in \BV([0,T];\xfin)$ is called a \emph{local solution} of the rate independent
  system {\RIS}
 if it satisfies the \emph{local stability} condition
 \begin{equation}
  \label{eq:65biss}
  \tag{$\mathrm{S}_{\mathrm{loc}}$}
  -\rmD\ene t{u(t)}\in K^*\quad\text{for every } t\in [0,T]\setminus \rmJ_u,
\end{equation}
and
  the \emph{energy dissipation inequality} \eqref{eq:2bis}.
\end{definition}
Local solutions admit the following differential characterization.
\begin{proposition}[Differential characterization of local solutions]
  \label{le:local-diff}
  A curve $u\in \BV([0,T];\xfin)$ is a \emph{local solution} of the rate independent
  system {\RIS}
  if and only if it satisfies the $\BV$ differential inclusion
     \begin{equation}
     \label{eq:56bis}
     \partial\Psi_0\Big(\frac {\d u'_\co}{d \mu}(t)\Big)+\rmD\ene t{u(t)}\ni 0\quad \text{for $\mu$-a.e.\ $t\in [0,T]$},
     \quad \mu:=\Leb 1+|u_\Ca'|,
     \tag{DN$_{0,\BV}$}
   \end{equation}
   and the jump inequalities
  \begin{equation}
  \label{eq:57bis}
  \begin{gathered}
    \ene{t}{u(t)}-\ene t{u(t_-)}\le -\cost\Diss{u(t_-)}{u(t)},\quad
    \ene{t}{u(t_+)}-\ene t{u(t)}\le -\cost\Diss{u(t)}{u(t_+)},\\
    \ene{t}{u(t_+)}-\ene t{u(t_-)}\le -\cost\Diss{u(t_-)}{u(t_+)},
  \end{gathered}
  \tag{J$_{\text{local}}$}
\end{equation}
at each jump time $t\in\rm J_u$.
\end{proposition}
\begin{proof}
  Notice that at every point $t\in (0,T)$ where $\d
  u'_\co(t)/\d\mu=0$, the differential inclusion \eqref{eq:66bis}
  reduces to the local stability condition \eqref{eq:65biss}. In the
  general case, \eqref{eq:56} follows by differentiation of
  \eqref{eq:2bis}. Indeed, the latter procedure provides the following
  inequality between the distributional derivative $\frac\d{\d t}\ene
  t{u(t)}$ of the map $t\mapsto \ene t{u(t)}$ and the $\Psi_0$-total
  variation measure $\Mdiff{\Psi_0}{u_\co'}:=\Psi_0\big(\d u_\co'/\d
  \mu\big)\mu$ for $\mu:=u'_\Ca+\Leb 1$
\begin{equation}
    \label{eq:75}
    \frac\d{\d t}\ene t{u(t)}+\Mdiff{\Psi_0}{u_\co'}-\partial_t \ene
    t{u(t)}\Leb 1
    \le 0\,.
\end{equation}
Applying the chain rule formula for the composition of the $\rmC^1$
functional $\cE$ and the $\BV$ curve $u$ (see~\cite{AmbDMa90AGCR}
and~\cite[Thm.~3.96]{Ambrosio-Fusco-Pallara00}) and taking into
account the fact that $u'_\co$ and $u'_\rmJ$ are mutually singular, we
obtain from~\eqref{eq:75} that
\begin{equation}
\label{eq:75-bis}
    \left\la -\rmD\ene t{u(t)},\frac {\d u'_\co}{\d\mu}\right\ra\mu
    \ge  \Mdiff{\Psi_0}{u'_\co}=\Psi_0\Big(\frac{\d
    u'_\co}{\d\mu}\Big)\mu\,.
    \end{equation}
Combining~\eqref{eq:75-bis} with the local stability
  condition \eqref{eq:65}, in view of the
  characterization~\eqref{eq:18} of $\partial \Diss$ and
  of~\eqref{e:2.3} we finally conclude~\eqref{eq:56}.
  Localizing \eqref{eq:2bis} around a jump point $t$ we get the inequalities \eqref{eq:57bis}.

  Conversely, let us suppose that a $\BV$ curve $u$ satisfies \eqref{eq:56} and \eqref{eq:57bis}.
  The local stability condition is an immediate consequence of \eqref{eq:56}, which yields
  $-\rmD\ene t{u(t)}\in \Kx$ for $\Leb 1$-a.e.~$t\in [0,T]$ and therefore, by continuity,
  at every point of $[0,T]\setminus \rmJ_u$.

  In order to get \eqref{eq:2bis}, we  again apply the chain rule for the composition
$\cE$ and $u$, obtaining
\begin{equation}
    \label{eq:89}
    \begin{aligned}
    \ene {t_1}{u(t_1)}+
    \int_{t_0}^{t_1} \left\la -\rmD\ene t{u(t)},\frac{\d u'_\co}{\d \mu}\right\ra \,\d\mu(t)
   &  -\JVar{}\cE{t_0}{t_1}\\ & =\ene {t_0}{u(t_0)}+\int_{t_0}^{t_1} \partial_t \ene t{u(t)}\,\d t,
    \end{aligned}
  \end{equation}
  where
  \begin{gather}\nonumber
    \JVar{}\cE{t_0}{t_1}=E_+(t_0)+E_-(t_1)+
    \sum_{t\in\rmJ_u\cap (t_0,t_1)} \Big(E_-(t)+E_+(t)\Big),\intertext{and}
    E_-(t):=\ene t{u(t)}-\ene t{u(t_-)},\quad
    E_+(t):=\ene t{u(t_+)}-\ene t{u(t)}. \nonumber
  \end{gather}
  By \eqref{eq:56} we have
  \begin{equation}
    \label{eq:108}
     \int_{t_0}^{t_1} \left\la -\rmD\ene t{u(t)},\frac{\d u'_\co}{\d \mu}\right\ra \,\d\mu(t)=
     \int_{t_0}^{t_1} \Psi_0\Big(\frac{\d u'_\co}{\d \mu}(t)\Big)\,\d\mu(t)=
     \Mint{t_0}{t_1}{\Psi_0}{u_\co'}\,,
  \end{equation}
  whereas \eqref{eq:57bis} yields for every $t\in \rmJ_u$
  \begin{equation}
    \label{eq:109}
    E_-(t)\le-\cost\Diss{u(t_-)}{u(t)},\quad
    E_+(t)\le -\cost\Diss{u(t)}{u(t_+)},
  \end{equation}
  so that $-\JVar{}\cE{t_0}{t_1}\ge\JVar{\Diss}u{t_0}{t_1}$ and therefore \eqref{eq:2bis} follows
  from \eqref{eq:89}.
\end{proof}
\begin{remark}
\label{rmk:lack} \upshape
 Unlike the case of energetic solutions
(cf. Proposition~\ref{le:global-diff}), a precise description of the
behavior of local solutions at jumps in missing here. In fact,
 the jump inequalities \eqref{eq:57bis} are not
sufficient to get an energy balance and do not completely capture
the jump dynamics, see the discussion of \cite[Sec.
5.2]{Mielke-Rossi-Savare08}.
\end{remark}

In order to get more precise insight into the jump properties and to
understand the correct energy balance along them, we have to
introduce a finer description of the dissipation. It is related to
an extra contribution to the jump part of $\Var{\Psi_0}
u\cdot\cdot$, which can be better described by using the
\emph{{\bipotential}} induced by the coupling $\Psi,\Psi^*$. We
describe this notion in the next section.

\section{{\bipotential}s and Finsler dissipation costs}
\label{s:3} \label{subsec:bipotentials} \subsection{
\textbf{Heuristics for the concept of {\bipotential}.}} \label{ss:3.1}
Suppose for the  moment being that, in a given time interval
$[r_0,r_1]$, the energy $\ene t\cdot=\cE(\cdot)$ does not change
w.r.t.\ time. If $\vartheta\in \AC([r_0,r_1];\xfin)$ is a solution
of \eqref{e:ris_eps} connecting $u_0=\vartheta(r_0)$ to
$u_1=\vartheta(r_1)$, then the energy release between the initial
and the final state is, by  the energy identity \eqref{e:enid-eps},
\begin{equation}
  \label{eq:76}
  \begin{aligned}
  \cE(u_0)-\cE(u_1)=&\int_{r_0}^{r_1}\Big(\Psi_\eps(v)  +\Psi_\eps^*(w)\Big)\,\d
  t,\\ &
  \text{with} \ \
  v(t)=\dot\vartheta(t)\quad\text{and}\quad
  w(t)= -\rmD\cE(\vartheta(t)) \quad
  \forae\, t \in (0,T).
  \end{aligned}
\end{equation}
If one looks for a lower bound of the right-hand side  in the above
energy identity \emph{which is independent of $\eps>0$}, it is
natural to recur to the functional $\bipo:\xfin\times\xfin^*\to [0,+\infty)$
defined by
\[
   \bipo(v,w):=\inf_{\eps>0}\left( \Psi_\eps(v)+\Psi_\eps^*(w)\right)=
   \inf_{\eps>0} \left(\eps^{-1}\Psi(\eps v)+\eps^{-1}\Psi^*(w)\right)\quad
   \text{for } v\in \xfin,\ w\in \xfin^*.
\]
We obtain
\begin{equation}
  \label{eq:77}
  \cE(u_0)-\cE(u_1)\ge \int_{r_0}^{r_1}\bipo(v,w)\,\d t \quad
  \text{with} \quad
  v(t)=\dot\vartheta(t)\text{ \ and \ } w(t)= -\rmD\cE(\vartheta(t)).
\end{equation}
Since $\bipo(\cdot,\cdot)$ is positively $1$-homogeneous with
respect to its first variable, the right-hand side expression in
\eqref{eq:77} is in fact independent of (monotone) time rescalings.
On the other hand, the \emph{\bipotential} $\bipo(\cdot,\cdot)$ has
the remarkable properties
\begin{equation}
  \label{eq:78}
  \bipo(v,w)\ge \la w,v\ra,\qquad
  \bipo(v,w)\ge \Psi_0(v)\quad\text{for every } v\in \xfin,\ w\in \xfin^*.
\end{equation}
Therefore, if $\tilde\vartheta\in \AC([r_0,r_1];\xfin)$ is another
arbitrary curve connecting $u_0$ to $u_1$, the chain
rule~\eqref{e:classical-chain_rule} for $\cE$ yields
\[
  \cE(u_0)-\cE(u_1)=\int_{r_0}^{r_1} \la \tilde w(t),\tilde v(t)\ra\,\d t
  \leq \int_{r_0}^{r_1}\left( \Psi_\eps (v(t)) + \Psi_\eps^* (\tilde
  w(t)) \right)\,\d t
\]
(where
 $\tilde v$ denotes the time derivative of $\tilde\vartheta$ and
$\tilde w=-\rmD\cE(\tilde\vartheta)$),  whence
\begin{equation}
  \label{eq:79}
 \cE(u_0)-\cE(u_1) \le \int_{r_0}^{r_1} \bipo(\tilde v(t),\tilde w(t))\,\d
  t\,.
  \end{equation}
 It follows that, in a time
regime in which the energy functional $\cE$ does not change with
respect to time, for every $\eps>0$ any viscous solution of
\eqref{e:ris_eps} (and, therefore, any suitable limit of viscous
solutions) should attain the minimum dissipation, measured in terms
of the {\bipotential} $\bipo$. Moreover, this dissipation always
provides an upper bound for the energy release, reached exactly
along viscous curves and their limits.
\begin{remark}\label{rem:bip}
   \upshape
   In some of the cases discussed in Example \ref{ex:viscous},
    the {\bipotential} $\bipo$ admits a more explicit
   representation.
   \begin{mylist}
   \item[(1)]
 We first consider the \emph{$\Psi_0$-viscosity} case \eqref{eq:92},
   where $\Psi(v):=F(\Diss(v))$,  $F:[0,+\infty)\to [0,+\infty)$ being a real convex superlinear function
   with $F(0)=0, F'(0)=1$. We introduce the $1$-homogeneous support function $\Psi_{0*}$ of the set
    \[
     K:=\big\{v\in \xfin:\Diss(v)\le 1\big\},\quad
     \Psi_{0*}(w):=\sup_{v\in K}\la w,v\ra.
     \]
   It is not difficult to show that $\Psi^*(w)=F^*(\Psi_{0*}(w))$
   and that for all $(v,w) \in \xfin \times \xfin^*$
   \begin{equation}
     \label{eq:30}
     \bipo(v,w)=\Diss(v)\,\max(1,\Psi_{0*}(w))=
     \begin{cases}
       \Diss(v)&\text{if }w\in \Kx,\\
       \Diss(v)\,\Psi_{0*}(w)&\text{if }w\not\in \Kx.
     \end{cases}
   \end{equation}
   \item[(2)]
   In the \emph{additive viscosity} case of \eqref{eq:5} one has  for all $(v,w) \in \xfin \times \xfin^*$
   \begin{equation}
     \label{eq:34}
     \bipo(v,w)=\Diss(v)+\bip_V(v,w),\ \text{where}\
     \bip_V(v,w)=\inf_{\eps>0} \left(\eps^{-1}\Psi_V(\eps v)+\eps^{-1}\inf_{z\in \Kx}\Psi_V^*(w-z)\right).
   \end{equation}
   In particular, when $\Psi_V(v)=F_V(\|v\|)$ for some norm $\|\cdot\|$ of $\xfin$ and a real convex and superlinear
   function $F_V:[0,+\infty)\to[0,+\infty)$ with $F_V(0)=F_V'(0)=0$, we
   have for all $(v,w) \in \xfin \times \xfin^*$
   \begin{equation}
     \label{eq:36}
     \bipo(v,w)=\Diss(v)+\bip_V(v,w),\quad \text{with} \ \
     \bip_V(v,w)=\|v\|\,\min_{z\in \Kx}\|w-z\|_*.
   \end{equation}
   Notice that in \eqref{eq:30} and  \eqref{eq:36}  the form of the {\bipotential} $\bipo$
   \emph{does not depend} on the choice of $F$ and  $F_V$, respectively, but only on the chosen viscosity norm.
   \end{mylist}
 \end{remark}

By the $1$-homogeneity of $\bipo(\cdot,w)$ and these variational
properties, it is then natural to introduce the following Finsler
dissipation.
\begin{definition}[Finsler dissipation]
  \label{def:Finsler_diss}
  For a fixed $t\in [0,T]$,    the Finsler cost 
  induced by $\bipo$ and (the differential of) $\cE$ at the time $t$ is given by
  \begin{equation}
    \label{eq:69}
    \begin{aligned}
      \Cost{\bipcE}t{u_0}{u_1}:=\inf\Big\{&\int_{r_0}^{r_1}
      \bipo(\dot \vartheta(r),-\rmD\ene t{\vartheta(r)})\,\d r:\\&
      \vartheta\in \AC([r_0,r_1];\xfin),\ \vartheta(r_0)=u_0,\
      \vartheta(r_1)=u_1\Big\}
    \end{aligned}
  \end{equation}
  for every $u_0,u_1\in
      \xfin$.
  We also consider the induced ``triple'' cost
  \[
  \TriCost{\bipcE}t{u_-}{u}{u_+}:=\Cost{\bipcE}t{u_-}{u}+\Cost{\bipcE}t{u}{u_+}.
  \]
\end{definition}
\begin{remark}
  \label{rem:inf-attained}
  Since $\bipo(v,w)\ge \Diss(v)$ by \eqref{eq:78}, a simple time rescaling argument shows that the infimum
  in \eqref{eq:69} is always attained by a \emph{Lipschitz curve}
$\vartheta \in \AC([r_0,r_1];\xfin)$
  with constant $\bipo$-speed,  in particular such that
  \[
  \bipo(\dot\vartheta(r),-\rmD\ene t{\vartheta(r)}\equiv 1\qquad
  \forae\,r \in (r_0,r_1)\,.
  \]
\end{remark}
By the heuristical discussion developed
throughout~\eqref{eq:76}--\eqref{eq:79}, the cost $\Delta_{\bipcE}$
is the natural candidate to substituting the potential $\Psi_0$ and
the related cost $\Delta_{\Psi_0}$ of \eqref{eq:64} in the jump
contributions \eqref{eq:37} and in the jump conditions
\eqref{eq:57}. Notice that the second relation of \eqref{eq:78}
implies
\begin{equation}
  \label{eq:82}
  \Cost{\bipcE}t{u_0}{u_1}\ge \cost{\Psi_0}{u_0}{u_1}\quad \text{for every } u_0,u_1\in \xfin.
\end{equation}
The  notion of jump variation arising from such replacements is
precisely stated as follows.
\begin{definition}[The total variation induced by $\Delta_{\bipcE}$]
  \label{def:Finsler_var}
Let $u\in \BV([0,T];\xfin)$ a given curve, let $u'_\co$ be the
diffuse part of its distributional derivative $u'$,
and let $\rmJ_u$ be its pointwise jump set \eqref{eq:88}. For every
subinterval $[a,b]\subset [0,T]$ the Jump variation of $u$ induced
by $(\bipcE)$ on $[a,b]$ is
\begin{equation}
  \label{eq:37bis}
  \begin{aligned}
    \JVar{\bipcE}uab:=
    &
    \Cost{\bipcE}a{u(a)}{u(a_+)}+\Cost{\bipcE}b{u(b_-)}{u(b)}+
    \\+&
    \sum_{t\in \rmJ_u\cap (a,b)}\TriCost{\bipcE}t{u(t_-)}{u(t)}{u(t_+)},
  \end{aligned}
\end{equation}
and the (pseudo-)total variation induced by $(\bipcE)$ is
\begin{equation}
  \label{eq:81}
  \pVar{\bipcE}uab:=\Mint ab{\Psi_0}{u'_\co}+\JVar{\bipcE}uab.
\end{equation}
\end{definition}
\begin{remark}[The (pseudo-)total variation $\pVarname\bipcE$]
  \label{rem:lsc-pseudo}
  Let us mention that $\pVarname{\bipcE}$ enjoys some of the
  properties of the usual total variation functionals,
  but it is not lower semicontinuous w.r.t.\ pointwise convergence. In fact,  it is not difficult to see that
  its lower semicontinuous envelope is simply $\Varname{\Diss}$.
Furthermore, $\pVarname{\bipcE}$
   is not induced
  by any distance on $\xfin$. Indeed,  we have used \textsl{slanted} fonts in
the notation
   $\pVarname{{}}$ to stress this fact.
  In order to recover a more standard total variation in a metric setting, one has to work in the extended space
  $\XX:=[0,T]\times \xfin$ and  add the local stability constraint $-\rmD\cE_t\in\Kx$ on the
  ``continuous'' part of the trajectories. We shall discuss this point of view in
  Section \ref{sec:technical}.
\end{remark}
In view of inequality~\eqref{eq:82} between the Finsler dissipation
$\Delta_{\bipcE}$ and $\Delta_{\Diss}$, the notion of total
variation associated with $\Delta_{\bipcE}$ provides an upper
bound for $\Varname{\Diss}$, namely
\begin{equation}
  \label{eq:83}
  \forall\, u\in \BV([0,T];\xfin),\ [a,b]\subset [0,T]:\qquad
  \pVar{\bipcE}uab\ge \Var{\Psi_0}uab.
\end{equation}
\subsection{{\Bipotential}s}
\label{subsec:bipotentials2} While postponing the definition of
$\BV$ solutions  related  to $\pVarname{\bipcE}$ to the next
section, let us add a few remarks about the {\bipotential} $\bipo$
  \begin{equation}
  \label{eq:68bis}
   \bipo(v,w):=\inf_{\eps>0}\left( \Psi_\eps(v)+\Psi_\eps^*(w)\right)=
   \inf_{\eps>0} \left(\eps^{-1}\Psi(\eps v)+\eps^{-1}\Psi^*(w)\right)\quad
   \text{for } v\in \xfin,\ w\in \xfin^*.
\end{equation}
which partly matches the definition introduced by
\cite{Buliga-deSaxce-Vallee08}.
 We first list a set of \emph{intrinsic} properties of $\bipo$, which we shall prove
  at the end of this section.
  \begin{theorem}[Intrinsic properties of $\bipo$]
  \label{thm:bipotential}
  The continuous functional $\bip:\xfin\times\xfins\to [0,+\infty)$
  defined by \eqref{eq:68bis} satisfies the following properties:
\begin{enumerate}[\upshape ({I}1)]
\item For every $v\in \xfin,w\in \xfin^*$ the maps $\bipo(v,\cdot)$
  and $\bipo(\cdot,w)$ have convex sublevels.
\item $\bip(v,w)\ge \la w,v\ra$ for every $v\in \xfin, w\in \xfin^*$.
\item For every $w\in \xfins$ the map $v\mapsto \bip(v,w)$ is
  $1$-homogeneous and thus convex in $\xfin$, with $\bipo(v,w)>0$ if
  $v\neq 0$.
\item For every $v\in \xfin,w\in \xfin^*$ the map $\lambda\mapsto
  \bipo(v,\lambda w)$ is nondecreasing in $[0,+\infty)$.
\item If for some $v_0\in \xfin$ and $\bar w,\, w \in \xfins$ we have
  $\bipo(v_0,\bar w)<\bipo(v_0, w)$, then the inequality $\bipo(v,\bar
  w)\le \bipo(v,w)$ holds for every $v\in \xfin$, and there exists
  $v_1\in \xfin$ such that $\bipo(v_1,\bar w)<\la w, v_1 \ra$.
\end{enumerate}
\end{theorem}

\begin{remark}[A dual family of convex sets]
  \upshape
  Property (I5) has a dual geometric counterpart:
  let us first observe that for every $w\in \xfin^*$ the map
  $v\mapsto \bipo(v,w)$ is a gauge function and therefore it is the
  support function of the convex set
\[
K_w^*:=\Big\{z\in \xfins:\la z,v \ra\le \bipo(v,w)\text{ for every }
v\in \xfin\Big\}, \quad\text{i.e.\ } \bipo(v,w)=\sup \Big\{\la
 z,v\ra:z\in K_w^*\Big\}.
 \]
Assertion (I5) then says that for every couple $w,\bar w\in \xfin$
\begin{equation}
  \label{eq:191}
  \text{we always have}\quad
  \bar w\in K_w^*\text{ or }w\in K_{\bar w}^*\quad
  \text{and, moreover,}\quad
  \bar w\in K_w^*\quad\Leftrightarrow\quad
  \bipo(\cdot,\bar w)\le \bipo(\cdot,w).
\end{equation}
Suppose in fact that $w\not\in K_{\bar w}^*$: this means that an
element $ v_0\in \xfin$ exists such that $\la w, v_0\ra>\bipo(
v_0,\bar w)$; by (I2) we get $\bipo(v_0, w)>\bipo( v_0,\bar w)$, and
therefore by (I5) $\bipo(v,w)\ge \bipo(v,\bar
w)\ge \la \bar w, v \ra$ for every $v\in \xfin$, so that $\bar w\in
K_w^*$.  The second statement of \eqref{eq:191} is an immediate
consequence of the second part of (I5).
\end{remark}

Property (I2) suggests that the set where equality holds in plays a
crucial role:
\begin{definition}[Contact set]
  The \emph{contact set} $\Contact\bip\subset \xfin\times \xfin^*$
  is defined as
  \begin{equation}
    \label{eq:53b}
    \Contact\bip:=\Big\{(v,w)\in \xfin\times\xfin^*: \bip(v,w)=\langle w,v\rangle\Big\}.
  \end{equation}
\end{definition}
Here are some other useful consequences of (I1--I5)
\begin{lemma}
  If $\bipo:\xfin\times\xfin^*\to [0,+\infty)$ satisfies {\upshape (I1--I5)}, then
  \begin{enumerate}[\upshape({I}1)]
    \setcounter{enumi}{5}
  \item for every $v\in \xfin,w\in \xfin^*$ we have
    \begin{equation}
      \label{eq:192}
      \bipo(v,0)+I_{K_0^*}(w)\ge \bipo(v,w)\ge \bipo(v,0).
    \end{equation}
\item The contact set can be characterized by
  \begin{equation}
  \label{eq:49}
  (v,w)\in \Contact\bipo\quad\Leftrightarrow\quad
  w\in \partial \bipo(\cdot,w)(v) \quad
  \Leftrightarrow \quad v\in \partial I_{K_w^*}(w).
  \end{equation}
  More generally, if $\bar w \in \partial \bipo(\cdot,w)(v)$ then
  $(v,\bar w)\in \Contact\bipo$, $\bar w\in K_w^*$, and
  $\bipo(v,w)=\bipo(v,\bar w)$.  In particular, if $\bar w\in \partial
  K_w^*$ then $w\in K_{\bar w}^*$.
\end{enumerate}
\end{lemma}
\begin{proof}
  The chain of inequalities in~\eqref{eq:192} is an immediate
  consequence of (I4) and of \eqref{eq:191}. \eqref{eq:49} is a direct
  consequence of the fact that $v\mapsto \bipo(v,w)$ is a gauge
  function and $I_{K_w^*}$ is its Legendre transform.

  In order to check the last statement,
  given $v\in \xfin,w\in \xfin^*$ let us take $\bar w\in \partial\bipo(\cdot,w)(v)$ so that
  $\bar w\in K_w^*$ and $\bipo(v,w)=\la \bar w ,v\ra$. Combining (I2) with \eqref{eq:191} we get
  $\bipo(v,w)=\bipo(v,\bar w)$, so that $(v,\bar w)\in \Contact\bipo$.
\end{proof}
\begin{remark}
  \label{rem:sublevels}
  Properties (I1,I2,I5) suggest a strong analogy between $\bipo$ and
  the notion of bipotential introduced by
  \cite{Buliga-deSaxce-Vallee08}: according to
  \cite{Buliga-deSaxce-Vallee08}, a bipotential is a functional $\frak
  b:\xfin\times \xfin^*\to ({-}\infty,+\infty]$ which is \emph{convex}
  and lower semicontinuous in each argument, satisfies (I2), and whose
  contact set fulfils a condition similar to \eqref{eq:49}
\[
(v,w)\in \Contact{\frak b}\quad\Leftrightarrow\quad w\in \partial
{\frak b}(\cdot,w)(v)\quad \Leftrightarrow\quad v\in \partial {\frak
  b}(v,\cdot)(w).
\]
  In our situation, \eqref{eq:49} is a direct consequence of
  the homogeneity of $\bipo$, but the convexity condition with respect to $w$ looks too
  restrictive, as shown by
  this simple example.
  Consider the case $\xfin=\xfin^*=\R^2$, with $\Psi(v):=\|v\|_1+\Psi_V(v)$, $\|v\|_1:=|v_1|+|v_2|$,
  and
  \begin{displaymath}
     \Psi_V(v):=\frac 12 v_1^2+\frac 14 v_2^4,\quad v=(v_1,v_2)\in \R^2;\qquad
    \Psi_V^*(w)=\frac 12 w_1^2+\frac 34 w_2^{4/3}\quad
    w=(w_1,w_2)\in \R^2.
  \end{displaymath}
  By \eqref{eq:34} we have
   $ \bipo(v,w)=\|v\|_1{+}\bipo_V(v,w)$ with $ \bipo_V(v,w)=\inf_{\eps>0}\tfrac1\eps\Big(\Psi_V(\eps v){+}
    \Psi^*(w)\Big)$ and find
\[
 \Psi^*(w)=\frac 12(|w_1|-1)_+^2+\frac 34(|w_2|-1)_+^{4/3 }\, .
\]
Considering the special case $v=(v_1,0),\ w=(0,w_2)$, we obtain
   \[
    \bip_V((v_1,0),(0,w_2))=\sqrt{ 3/2}\,|v_1|\Big((|w_2|-1)_+
    \Big)^{2/3}.
    \]
  The map $w_2\mapsto \bipo((v_1,0),(0,w_2))$ is therefore not convex.
\end{remark}

Let us  now  consider some properties of $\bipo$ and its contact set
$\Contact\bipo$ involving explicitly the functional $\Psi$. Since
the {\bipotential} $\bipo$ is  defined through the minimum
procedure~\eqref{eq:68bis}, the contact set is strictly related to
the set of optimal $\eps>0$ attaining the minimum in
\eqref{eq:68bis}.
\begin{definition}[Lagrange multipliers]
  For every $(v,w)\in \xfin\times \xfins$
  we introduce the multivalued function
  $\fl$ (with possibly empty values)
  \begin{equation}
    \label{eq:111}
    \fl(v,w):=\Big\{\eps\ge 0:\bipo(v,w)=\Psi_\eps(v)+\Psi^*_\eps(w)\Big\}\subset [0,+\infty).
  \end{equation}
\end{definition}
Notice that for every $(v,w)\in \xfin\times\xfin^*$ the function
$\eps\mapsto \eps^{-1}\Psi(\eps v)+\eps^{-1}\Psi^*(w)$ is convex on
$(0,+\infty)$. Since $\Psi$ has superlinear growth at infinity, it
goes to $+\infty$ as $\eps\uparrow +\infty$ if $v\neq 0$, so that
\begin{equation}
\label{e:new-entry} \text{the set $\fl(v,w)$ is always a bounded
closed interval if $v\neq 0$.}
\end{equation}
\begin{theorem}[Properties of $\bipo,\Psi$ and $\Contact\bipo$]
  \label{thm:auxiliary}
  \
  \begin{enumerate}[\rm (P1)]
\item The {\bipotential} $\bipo$ satisfies $\bipo(v,0)=\Psi_0(v)$,
    $K_0^*=\Kx$, and in particular
    \begin{gather}
    \label{eq:13b}
    \bipo(v,w)\ge \la w,v\ra,\quad
    \Diss(v)+\rmI_{\Kx}(w)\ge\bipo(v,w)\ge \Diss(v)\ge 0\quad\text{for every }v\in \xfin,\ w\in \xfins,\\
    \label{eq:112}
    \bipo(v,w)=\Diss(v)\quad \Leftrightarrow\quad
    w\in \Kx.
  \end{gather}
\item For every $w \in \xfins$, the convex sets $K_w^*$ are the
  sublevels of $\Psi^*$
  \begin{equation}
      \label{eq:188}
      K_w^*=\Big\{z\in \xfin^*:\Psi^*(z)\le \Psi^*(w)\Big\},
    \end{equation}
  and $\bipo$ admits the dual representation
  \begin{equation}
    \label{eq:47}
    \bipo(v,w)=\sup\Big\{\la z,v\ra\,:\,z \in \xfin^*, \ \Psi^*(z)\le
    \Psi^*(w)\Big\}.
  \end{equation}
    In particular, $\Psi^*(w_1)\le \Psi^*(w_2)$ for
    some $w_1,\, w_2 \in \xfins$ if and only if
    $\bipo(v,w_1)\le\bipo(v,w_2)$ for every $v\in \xfin$.
\item The multivalued function $\fl$ defined in \eqref{eq:111} is upper semicontinuous, i.e.
  \begin{equation}
    \label{eq:44}
    \text{if }(v_n,w_n)\to (v,w)\in \xfin\times \xfins\text{ and }
    \eps_n\in \fl(v_n,w_n)\to \eps,\quad
    \text{then}\quad
    \eps\in \fl(v,w).
  \end{equation}
\item  The contact set $\Contact\bipo$ \eqref{eq:53b} can be characterized by
  \begin{equation}
    \label{eq:32b}
    w\in \partial\Diss(v)\subset \Kx\quad\text{or,}\quad \text{if} \
    w\not\in \Kx,\quad
    \exists\, \eps>0: w\in \partial\Psi(\eps v),
  \end{equation}
  and the last inclusion holds exactly for $\eps\in \fl(v,w)$. 
  Equivalently,
   \[
    (v,w)\in \Contact\bipo\quad\Leftrightarrow\quad
    w\in \partial\Psi_\eps(v)\quad\text{for every }\eps\in \fl(v,w).
    \]
  In particular, in the case of \emph{additive viscosity}, with  $\Psi(v)=
  \Diss(v)+\Psi_V(v)$ and $\Psi_V$ satisfying \eqref{eq:25}, we simply have
  \begin{equation}
    \label{eq:33}
    (v,w)\in \Contact\bipo\quad
     \Longleftrightarrow\quad
     \exists\,\lambda\ge 0:\quad
     w\in \partial\Diss(v)+\partial\Psi_V(\lambda v).
   \end{equation}
\end{enumerate}
\end{theorem}
\begin{proof}[Proofs of Theorems \ref{thm:auxiliary} and \ref{thm:bipotential}]
\ \\
\textbf{Ad (P1).} \, Inequalities \eqref{eq:13b} are
  immediate consequences of the definition of $\bipo$.
  The equality $\Diss(v)=\bipo(v,w)$ is equivalent to the existence of a sequence $\eps_k>0$ such that
  (recall that $\eps^{-1}\Psi_\eps(\eps v)\ge \Diss(v)$)
  \begin{displaymath}
    \lim_{k\to\infty} \eps_k^{-1}\Psi(\eps_k v)=\Diss(v),\quad \lim_{k\to\infty}\eps_k^{-1}\Psi^*(w)=0.
  \end{displaymath}
  Since the first inequality prevents $\eps_k$ from diverging to $+\infty$ (being $\Psi$ superlinear),
  from the second limit we get
  $\Psi^*(w)=0$, i.e.
  \begin{displaymath}
    \la w,z\ra\le \Psi(z)\quad \forall\, z\in \xfin.
  \end{displaymath}
  Replacing $z$ with $\eps z$, multiplying the previous inequality by $\eps^{-1}$, and passing to the limit as $\eps\downarrow 0$,
 in view of~\eqref{eq:21} we conclude
  \begin{displaymath}
    \la w,z\ra\le \Diss(z)\quad\forall\, z\in \xfin,\quad\text{so that}\quad
    w\in \Kx.
  \end{displaymath}
  The converse implication in~\eqref{eq:112} is immediate.
  \\
  \textbf{Ad (P2).}  Since the sublevels of $\Psi^*$ are closed and
  convex, a duality argument shows that \eqref{eq:188} is equivalent
  to \eqref{eq:47}. In order to prove the latter formula, let us
  observe that, if $\Psi^*(z)\le \Psi^*(w)$, then $\la z,v\ra \le
  \bipo(v,w)$, because the Fenchel inequality yields
  \begin{displaymath}
    \la z,v\ra=\eps^{-1}\la z, \eps v\ra \le
    \eps^{-1}\Psi(\eps v)+\eps^{-1} \Psi(z)
    = \Psi_\eps(v)+\Psi_\eps^*(z)
    \le
    \Psi_\eps(v)+\Psi_\eps^*(w)\quad\text{for every }\eps>0.
  \end{displaymath}
  We show that there exists $z\in \xfins$ such that $\Psi^*(z)\le
  \Psi^*(w)$ and $\bipo(v,w)=\la z,v\ra.$ Due to~\eqref{eq:112}, if
  $w\in \Kx$, then $\bipo(v,w)=\Diss(v)$ and the thesis follows from
  \eqref{eq:38} Hence, let us suppose that $w\not\in \Kx$ and $v\neq
  0$; then we can choose $\eps_0\in\fl(v,w)$, $ \eps_0>0$, such that
  \begin{equation}
  \label{e:added}
    \bipo(v,w)=\eps_0^{-1}\Psi(\eps_0 v)+\eps_0^{-1}\Psi^*(w)\le
    \eps^{-1}\Psi(\eps v)+\eps^{-1}\Psi^*(w)\quad
    \text{for every }\eps>0.
  \end{equation}
  Choosing $z_\eps\in \partial\Psi(\eps v)$
  we have
  \begin{displaymath}
    \Psi(\eps v)-\Psi(\eps_0 v)\le \la z_\eps, (\eps-\eps_0) v\ra\quad\text{for every }\eps>0
  \end{displaymath}
  so that, in view of inequality~\eqref{e:added},
  \begin{align*}
    \big(\eps^{-1}-\eps_0^{-1}\big)\Big(\Psi(\eps_0 v)+\Psi^*(w)\Big)+\eps^{-1}\la  z_\eps , (\eps-\eps_0) v \ra
    \ge 0\quad\text{for every }\eps>0.
  \end{align*}
  Dividing by $\eps-\eps_0$ and passing to the limit first as $\eps\down\eps_0$
  and then as $\eps\up\eps_0$, we thus find $z_\pm\in \partial\Psi(\eps_0 v)$
  (accumulation points of the sequences $(z_\eps:\eps>\eps_0)$ and
  $(z_\eps:\eps<\eps_0)$, respectively),  such that
  \begin{equation}
    \label{eq:51}
    \la z_- ,v\ra\le \bipo(v,w)=
    \eps_0^{-1}\Big(\Psi(\eps_0 v)+\Psi^*(w)\Big)\le \la z_+, v \ra.
  \end{equation}
  On the other hand,  the Fenchel identity of convex analysis  yields
  \begin{equation}
    \label{eq:52}
    \eps_0^{-1}\Psi^*(z)=\la z,v\ra-\eps_0^{-1}\Psi(\eps_0 v)\quad \text{for every }z\in \partial
    \Psi(\eps_0 v)
  \end{equation}
  so that the map $z\mapsto \Psi^*(z)$ is affine on $\partial\Psi(\eps_0 v)$ and a comparison
  between \eqref{eq:51} and \eqref{eq:52} yields
    \[
    \Psi^*(z_-)\le \Psi^*(w)\le \Psi^*(z_+).
    \]
 Using formula~\eqref{eq:52} we can thus find $\theta\in [0,1]$ and $z_\theta:=(1-\theta)z_-+\theta z_+\in
  \partial\Psi(\eps_0 v)$ such that
   \[
    \Psi^*(z_\theta)=\Psi^*(w),\quad
    \la z_\theta, v \ra= \bipo(v,w)=
    \eps_0^{-1}\Big(\Psi(\eps_0 v)+\Psi^*(w)\Big).
    \]
    The last statement of (P2) follows easily. One implication is
    immediate. On the other hand,  if $\Psi^*(w_1)>\Psi^*(w_2)$, then by the Hahn-Banach
    separation theorem we can find $\bar v\in \xfin$ and $\delta>0$
    such that
    \begin{displaymath}
      \la w_1,  \bar v \ra \ge \delta+\la z, \bar v\ra \quad
     \text{for every $z\in \xfin^*$ such that $\Psi^*(z)\le \Psi^*(w_2)$},
    \end{displaymath}
    and, therefore, by~\eqref{eq:47} we conclude $\bipo(\bar v,w_1)\ge
    \la w_1, \bar v \ra\ge \delta+\bipo(\bar v,w_2)$.

  \noindent
  \textbf{Ad (I1,2,3,4,5)} These properties  directly follow from
  (P2).

  \noindent
  \textbf{Ad (P3) and continuity of $\bipo$.} Notice that $\bipo$ is
  upper semicontinuous, being defined as the infimum of a family of
  continuous functions.  Take now converging sequences $(v_n),\,
  (w_n),\, (\eps_n)$ as in~\eqref{eq:44}: we have that
  \[
    \liminf_{n\to\infty}\left(\eps_{n}^{-1}\Psi(\eps_{n}v_{n})+\eps^{-1}_{n}\Psi^*(w_{n})\right)\ge
    \Psi_\eps(v)+\Psi_\eps^*(w)=
    \begin{cases}
      \eps^{-1}\Psi(\eps v)+\eps^{-1}\Psi^*(w)&\text{if }\eps>0,\\
      \Psi_0(v)+\rmI_\Kx(w)&\text{if }\eps=0.
    \end{cases}
    \]
  Since
  \begin{equation}
  \begin{aligned}
  \label{eq:added2}
    \bipo(v,w)
    \ge \liminf_{n\to\infty}\bipo(v_{n},w_{n}) & \ge
    \liminf_{n\to\infty}\left(\eps_{n}^{-1}\Psi(\eps_{n}v_{n})+\eps^{-1}_{n}\Psi^*(w_{n})\right)
   \\ &  \ge   \Psi_\eps(v)+\Psi_\eps^*(w)\ge
     \bipo(v,w),
  \end{aligned}
  \end{equation}
  we obtain $\eps\in \fl(v,w)$.
  \\
  Inequality~\eqref{eq:added2} shows that $\bipo$ is also lower semicontinuous,
   since,
   if $v\neq0$, any sequence $\eps_n\in\fl(v_n,w_n)$ admits a converging subsequence,
   in view of~\eqref{e:new-entry}.
  \\
\textbf{Ad (P4).} Concerning the characterization \eqref{eq:32b} of $\Contact\bipo$,
   it is easy to check that,
  if $(v,w)$ satisfies \eqref{eq:32b}, then by the  Fenchel identity
  and formula~\eqref{eq:18}
  we have, when $w\in \Kx,$
  \begin{displaymath}
    \bipo(v,w)\ge\la w, v \ra=\Psi_0(v)=\bipo(v,w),
  \end{displaymath}
  and, when $w\not\in\Kx$,
  \begin{displaymath}
    \bipo(v,w)\ge \la w, v \ra=\eps^{-1}\la w, \eps v \ra=\eps^{-1}\Psi(\eps v)+\eps^{-1}\Psi^*(w)\ge \bipo(v,w)
  \end{displaymath}
  so that $(v,w)\in \Contact\bipo$ and $\eps\in \fl(v,w)$.
  Conversely, if $\bipo(v,w)=\la w, v \ra$ and
  $w\in \Kx$, then by \eqref{eq:13b} $\Diss(v)=\la w, v \ra$ and therefore $w\in \partial\Diss(v)$.
  If $w\not\in \Kx$, then, choosing $\eps\in \fl(v,w)$, we have
  \begin{displaymath}
    \Psi(\eps v)+\Psi^*(w)=\eps \bipo(v,w)=\la w, \eps v \ra,\quad\text{so that}\quad
    w\in \partial\Psi(\eps v).
  \end{displaymath}
  In the particular case of~\eqref{eq:5},
  \eqref{eq:33} follows now from  \eqref{eq:32b} by the sum rule of the subdifferentials and the
  $0$-homogeneity of $\partial\Diss$.
\end{proof}
\section{$\BV$ solutions and energy-driven dissipation}
\label{sec:BV-solutions}
 \subsection{$\BV$ solutions}
 \label{ss:4.1}
 We can now give our precise definition of $\BV$ solution of
 the rate-independent system {\VRIS}, driven by the {\bipotential} $\bipo$
 \eqref{eq:68bis} and the energy $\cE$.
 From a formal point of view, the definition simply replaces
 the global stability condition \eqref{eq:1} by the local one \eqref{eq:65}, and
 the $\Psi_0$-total variation in the energy balance \eqref{eq:2} by
 the ``Finsler'' total variation \eqref{eq:81}, induced by $\bipo$ and $\cE$.
\begin{definition}[$\BV$ solutions, variational characterization]
  \label{def:BV-solution}
  A curve $u\in \BV([0,T];\xfin)$ is a $\BV$ solution of the rate independent
  system {\VRIS}
   the \emph{local stability} \eqref{eq:65bis}
  and the
  $(\bipcE)$-\emph{energy balance} hold:
  \begin{equation}
    \label{eq:65bis}
    \tag{$\mathrm{S}_\mathrm{loc}$}
    -\rmD\ene t{u(t)}\in K^*\quad \forae\, t \in [0,T] \setminus J_u
  \end{equation}
  \begin{equation}
    \label{eq:84}
    \pVar{\bipcE}u{0}{t}+\ene{t}{u(t)}=\ene{0}{u(0)}+
    \int_{0}^{t} \partial_t\ene s{u(s)}\,\d s \quad \text{ for all } t\in [0,T].
    \tag{E$_{\bipcE}$}
  \end{equation}
\end{definition}
We shall see in the next Section \ref{subsec:viscous-limit} that any
pointwise limit, as $\eps \down 0$, of the solutions $(u_\eps)$ of
the viscous equation \eqref{e:ris_eps} or, as $\tau,\, \eps \down
0$, of the discrete solutions $(\overline{\Uu}_{\tau,\eps})$ of the
viscous incremental problems \eqref{eq:61},  is a  $\BV$ solutions
induced by the {\bipotential} $\bipo$. Let us first get more insight
into  Definition~\ref{def:BV-solution}. 
\paragraph{\textbf{Properties of $\BV$ solutions.}}
 As in the case of
energetic solutions, it is not difficult to see that the energy
balance \eqref{eq:84} holds on any subinterval $[t_0,t_1]\subset
[0,T]$; moreover, if the local stability condition \eqref{eq:65bis}
holds, to check \eqref{eq:84} it is sufficient to prove the
corresponding inequality.
\begin{proposition}
  \label{prop:useful}
  If $u\in \BV([0,T];\xfin)$ satisfies \eqref{eq:84}, then for every subinterval $[t_0,t_1]$
  there holds
  \begin{equation}
    \label{eq:73}
    \pVar{\bipcE}u{t_0}{t_1}+\ene{t_1}{u(t_1)}=\ene{t_0}{u(t_0)}+
    \int_{t_0}^{t_1} \partial_t\ene s{u(s)}\,\d s.
    \tag{E$'_{\bipcE}$}
  \end{equation}
  Moreover, if $u$ satisfies \eqref{eq:65bis}, then \eqref{eq:84} is equivalent
  to the energy inequality
  \begin{equation}
    \label{eq:85}
    \pVar{\bipcE}u{0}{T}+\ene{T}{u(T)}\le \ene{0}{u(0)}+
    \int_{0}^{T} \partial_t\ene s{u(s)}\,\d s.
    \tag{E$_{\bipcE;\text{ineq}}$}
  \end{equation}
\end{proposition}
\begin{proof}
  \eqref{eq:73} easily follows from the additivity property
  \begin{equation}
    \label{eq:86}
    \forall\, 0\le t_0<t_1<t_2\le T:\quad
    \pVar{\bipcE}u{t_0}{t_1}+\pVar{\bipcE}u{t_1}{t_2}=\pVar{\bipcE}u{t_0}{t_2}.
  \end{equation}
  In order to prove the second inequality we argue as in \cite[Prop.~4]{Mielke-Rossi-Savare08}, taking
  \eqref{eq:65bis} into account.
\end{proof}
Notice that, by \eqref{eq:83}, any $\BV$ solution is also a local
solution according to Definition \ref{def:local}, i.e.~it  satisfies
the local stability condition and energy inequality \eqref{eq:2bis}.
In fact, one has a more accurate description of the jump conditions,
as the following Theorem shows (cf.~with
Propositions~\ref{le:global-diff} and~\ref{le:local-diff}).
\begin{theorem}[Differential characterization of $\BV$ solutions]
  \label{thm:def-bv-solutions}
  A curve $u\in \BV([0,T];\xfin)$ is a $\BV$ solution of the rate-independent
  system {\VRIS}
  if and only if
  it satisfies the doubly nonlinear differential inclusion in the $\BV$ sense
 \begin{equation}
    \label{eq:66bis}
    \tag{DN$_{0,\BV}$}
    \partial\Psi_0\Big(\frac {\d u'_\co}{d \mu}(t)\Big)+\rmD\ene t{u(t)}\ni 0\quad \text{for $\mu$-a.e.\ $t\in [0,T]$},
     \quad \mu:=\Leb 1+|u_\Ca'|,
  \end{equation}
  and the following jump conditions at each point $t\in\rmJ_u$ of the jump set
  \eqref{eq:88}
  \begin{equation}
    \label{eq:67}
      \begin{gathered}
    \ene{t}{u(t)}-\ene t{u(t_-)}=-\Cost{\bipcE}t{u(t_-)}{u(t)},\\
    \ene{t}{u(t_+)}-\ene t{u(t)}=-\Cost{\bipcE}t{u(t)}{u(t_+)},\\
    \ene{t}{u(t_+)}-\ene t{u(t_-)}=-\Cost{\bipcE}t{u(t_-)}{u(t_+)}.
  \end{gathered}
    \tag{J$_{\BV}$}
  \end{equation}
\end{theorem}
\begin{proof}
  We have already seen (see Lemma \ref{le:local-diff})
  that local solutions satisfy
  \eqref{eq:66bis}. The jump conditions~\eqref{eq:67} can be obtained by
  localizing \eqref{eq:73} around any jump time $t\in \rmJ_u$.

  Conversely,
to prove~\eqref{eq:85} (as seen in the proof of
Lemma~\ref{le:local-diff}, \eqref{eq:65bis} ensues
from~\eqref{eq:66bis}),
  we argue as in the second part of the proof of Lemma
  \ref{le:local-diff}, still applying \eqref{eq:89} and \eqref{eq:108}, but replacing
   inequalities \eqref{eq:109} with the
following identities,
  \begin{displaymath}
    E_-(t)=-\Cost{\bipcE}t{u(t_-)}{u(t)},\quad
    E_+(t)=-\Cost{\bipcE}t{u(t)}{u(t_+)} \quad \text{for all $t \in \rmJ_u$,}
  \end{displaymath}
which are  due to~\eqref{eq:67}.
  Hence,  $-\JVar{}\cE0T=\JVar{\bipcE}u0T$. Then, \eqref{eq:85}  follows
  from \eqref{eq:89}.
\end{proof}
The next section is devoted to a refined description of the behavior
of a $\BV$ solution along the jumps.
\subsection{Jumps and optimal transitions}
\label{ss:4.2} Let us first introduce the notion of \emph{optimal
transition}.
\begin{definition}
  \label{def:optimal-transition}
  Let $t\in [0,T]$, $u_-,u_+\in \xfin$ with $-\rmD\ene t{u_-},\,  -\rmD\ene t{u_+}
  \in \Kx$,  and
  $-\infty\le r_0<r_1\le +\infty$.
  An absolutely continuous curve
  $\vartheta:[r_0,r_1]\to \xfin$ connecting $u_-=\vartheta(r_0)$ and $u_+=\vartheta(r_1)$
  is an \emph{optimal $(\bipo,\cE_t)$-transition} between $u_-$ and
  $u_+$   if
  \begin{equation}
    \label{eq:66}
    \tag{O.1}
    \dot \vartheta(r)\neq 0\quad \text{for a.a.\ }r\in (r_0,r_1);\quad
    \Psi_{0*}({-}\rmD\ene t{\vartheta(r)})\ge 1\quad \forall\, r\in [r_0,r_1],
  \end{equation}
  \begin{equation}
    \label{eq:91}
    \ene t{u_-}-\ene t{u_+}=\Cost{\bipcE}t{u_-}{u_+}=
    \int_{r_0}^{r_1} \bipo(\dot\vartheta(r),-\rmD\ene t{\vartheta(r)})\,\d r.
    \tag{O.2}
  \end{equation}
  We also say that an optimal transition $\vartheta$ is of
  \begin{align}
    \label{eq:93}
    \text{\emph{sliding} type if}\qquad
    &-\rmD\ene t{\vartheta(r)}\in \Kx\quad\text{for every $r\in [r_0,r_1]$,}
    \tag{O$_{\text{sliding}}$}\\
    \text{\emph{viscous} type if}\qquad
    &-\rmD\ene t{\vartheta(r)}\not \in \Kx\quad\text{for every $r\in (r_0,r_1)$,}
    \tag{O$_{\text{viscous}}$}
    \label{eq:94}\\
    \label{eq:116}
    \text{\emph{energetic} type if}\qquad
    &\ene t{u_+}-\ene t{u_-}=-\Diss(u_+-u_-).
    \tag{O$_{\text{ener}}$}
  \end{align}
  We denote by $\Theta(t;u_-,u_+)$ the (possibly empty) collection of such optimal transitions, with
  normalized domain $[0,1]$ and  \emph{constant Finsler velocity}
\begin{equation}
\label{e:constant-finsler}
 \bipo(\dot\vartheta(r),-\rmD\cE_t(\vartheta(r)))
  \equiv \ene t{u_-}-\ene t{u_+} \qquad \forae\, r \in (0,1)\,.
  \end{equation}
\end{definition}
\begin{remark}
\upshape \label{rem:newrmk}
 Notice that the notion of optimal
transition is invariant by absolutely continuous (monotone) time
rescalings with absolutely continuous inverse; moreover, any optimal
transition $\vartheta$ has finite length, it admits a
reparametrization with constant Finsler velocity
$\bipo(\dot\vartheta(\cdot),-\rmD\cE_t(\vartheta(\cdot)))$, and is a
minimizer of \eqref{eq:69}, so that it is not restrictive to assume
$\vartheta\in \Theta(t,u_-,u_+).$
\end{remark}
\begin{theorem}
  \label{thm:filling-jumps}
  A local solution $u\in \BV([0,T];\xfin)$
  is a $\BV$ solution according to Definition~\ref{def:BV-solution} if and only if
  at every jump time $t\in\rmJ_u$ the initial and final values $u(t_-)$ and $u(t_+)$ can be connected
  by an optimal transition curve
  $\vartheta^t\in \Theta(t;u(t_-),u(t_+))$, and there exists $r\in [0,1]$ such that
  $u(t)=\vartheta^t(r)$.
  Any optimal transition curve $\vartheta$ satisfies the \emph{contact condition}
  \begin{equation}
    \label{eq:106}
    (\dot\vartheta(r),-\rmD\ene t{\vartheta(r)})\in \Contact\bipo\quad
    \text{for a.a.\ }r\in (0,1).
  \end{equation}
\end{theorem}
\begin{proof}
  Taking into account Theorem~\ref{thm:def-bv-solutions}, the proof of the first part of the statement
is immediate. To prove~\eqref{eq:106},
  let $t$ be a jump point of $u$ and let us first suppose that $u(t_-)=u(t)\neq u(t_+)$.
  By Remark \ref{rem:inf-attained}, we can find a
  Lipschitz curve
  $\vartheta_{01}\in \AC([r_0,r_1];\xfin)$
  with normalized speed $\bipo(\dot\vartheta,-\rmD\ene t{\vartheta})\equiv
  1$,
  connecting $u(t_-)$ to $u(t_+)$, so that the jump condition \eqref{eq:67} yields
  \begin{displaymath}
    \int_{r_0}^{r_1} \la -\rmD\ene t{\vartheta(r)}, \dot\vartheta(r)\ra\,\d r=
    \ene t{u(t_-)}-\ene t{u(t_+)}=\int_{r_0}^{r_1} \bipo(\dot
    \vartheta(r),-\rmD\ene t{\vartheta(r)})\, \d r.
  \end{displaymath}
  This shows that $\vartheta$ is an optimal transition curve and satisfies
   \[
    \int_{r_0}^{r_1} \Big(\bipo(\dot \vartheta,-\rmD\ene t{\vartheta(r)})\, \d r-
    \la -\rmD\ene t{\vartheta(r)}, \dot\vartheta(r)\ra \Big)\,\d r=0.
    \]
  Since the integrand is always nonnegative, it follows that
  \eqref{eq:106} holds.

  In the general case, when $u$ is not left or right continuous at $t$, we join two (suitably rescaled)
  optimal transition curves $\vartheta_{01}\in \Theta(t;u(t_-),u(t))$ and
  $\vartheta_{12}\in \Theta(t;u(t),u(t_+))$.
\end{proof}
The next result provides  a careful
description of $(\bipo,\cE_t)$-optimal transitions.
\begin{theorem}
\label{thm:careful}
  Let
$ t \in [0,T]$, $u_-$, $u_+ \in \xfin$, and $\vartheta :[0,1] \to
\xfin$ be an optimal transition
  curve in $\Theta(t;u_-,u_+)$. Then,
  \begin{enumerate}[\rm (1)]
  \item $\vartheta$ is a constant-speed minimal geodesic for the (possibly asymmetric)
    Finsler cost $\Cost{\bipcE}t{u_-}{u_+}$,
    and for every $0\le \rho_0<\rho_1\le 1$ it satisfies
    \begin{equation}
      \label{eq:97}
      \begin{aligned}
        \ene t{\vartheta(\rho_0)}-\ene t{\vartheta(\rho_1)}={}&
        \Cost{\bipcE}t{\vartheta(\rho_0)}{\vartheta(\rho_1)}\\={}&(\rho_1-\rho_0)
        \Cost{\bipcE}t{u_-}{u_+}=(\rho_1-\rho_0)\Big(\ene t{u_-}-\ene t{u_+}\Big);
      \end{aligned}
    \end{equation}
    In particular,  the map
    $\rho\mapsto \ene t{\vartheta(\rho)}$ is affine.
  \item An optimal transition $\vartheta$ is of \emph{sliding} type \eqref{eq:93} if and only if
    it satisfies
    \begin{equation}
      \label{eq:98}
      \partial\Diss(\dot\vartheta(r))+\rmD\ene t{\vartheta(r)}\ni 0\quad
      \forae\ r \in (0,1),
    \end{equation}
    \begin{equation}
      \label{eq:95}
      \Psi_{0*}({-}\rmD\ene t{\vartheta(r)})=1\quad \text{for every }r\in [0,1].
   \end{equation}
 \item An optimal transition $\vartheta$ is of \emph{viscous}  type \eqref{eq:94}  if and only if there holds
  for every
   selection
   $(0,1)\ni r\mapsto\eps(r)$ in $\fl(\dot\vartheta(r),-\rmD\ene
   t{\vartheta(r)}$
    \begin{equation}
      \label{eq:99}
      \partial\Psi(\eps(r)\dot\vartheta(r))+\rmD\ene t{\vartheta(r)}\ni0\quad
      \forae\, r \in (0,1).
    \end{equation}
    Equivalently, there exists
    an absolutely continuous, surjective time rescaling $\sfr:(\rho_0,\rho_1)\to (0,1)$, with
    $-\infty\le \rho_0 <\rho_1\le \infty$ and
    $\dot\sfr(s)>0$ for
    $\Leb 1$ a.e.~$s\in (\rho_0,\rho_1)$,
    such that the rescaled transition $\theta(s):=\vartheta(\sfr(s))$ satisfies the viscous differential inclusion
    \begin{equation}
      \label{eq:100}
      \partial\Psi(\dot\theta(s))+\rmD \ene t{\theta(s)}\ni0\quad
      \forae\, s \in (\rho_0,\rho_1)\,, \quad \text{with} \ \
      \lim_{s\downarrow \rho_0}\theta(s)=u_-, \ \ \lim_{s\uparrow
      \rho_1}\theta(s)=u_+\,.
    \end{equation}
  \item Any optimal transition $\vartheta$ can be decomposed in a canonical way
    into an (at most) countable collection of
    optimal \emph{sliding and viscous} transitions. In other words, there exists (uniquely determined) disjoint open intervals
    $(S_j)_{j\in \sigma}$ and $(V_k)_{k\in \upsilon}$ of $(0,1)$, with
     $\sigma,\upsilon\subset \N$, such that
    $(0,1)\subset \big(\cup_{j\in\sigma} \overline{S_j})\cup\big(\cup_{k\in \upsilon}V_k\big)$
    and
     \[
      \vartheta\Restr {\overline{S_j}}\quad\text{is of sliding type,}\quad
      \vartheta\Restr {V_k}\quad\text{is of viscous type.}
      \]
  \item An optimal transition $\vartheta$ is of \emph{energetic} type \eqref{eq:116}
    if and only if $\vartheta$ is of sliding type and it is a $\Psi_0$-minimal
    geodesic, i.e.
    \begin{equation}
      \label{eq:96}
      \Diss(\vartheta(r_1)-\vartheta(r_0))=
      (r_1-r_0)\Diss(u_1-u_0)\quad\text{for every }0\le r_0<r_1\le 1.
    \end{equation}
    If $\Diss$  has strictly convex sublevels,  then
    $\vartheta$ is linear and $r\mapsto
    (\vartheta(r),\ene t{\vartheta(r)})$ is a linear segment contained in the graph of $\cE_t$. \\
    If $\Diss$ is G\^{a}teaux-differentiable at $\xfin\setminus \{0\}$
    then
     \[
       -\rmD\ene t{\vartheta(r)}=\rmD\Diss(u_+-u_-)\quad\text{for every }r\in
      [0,1].
      \]
    In particular, the map $r \mapsto -\rmD\ene t{\vartheta(r)}$ is
    constant.
  \end{enumerate}
\end{theorem}
\begin{remark}
  \label{rem:obvious}
  \upshape
  It follows from the characterization in~\textbf{(2)} of Theorem~\ref{thm:careful}
    (cf. with  \eqref{eq:98}--\eqref{eq:95})
  that sliding optimal transitions are independent of
  the form of the {\bipotential} $\bipo$, and thus on the particular viscosity potential
  $\Psi$.

  Instead, as one may expect,
    $\Psi$
occurs in the doubly nonlinear equation~\eqref{eq:99} (equivalently,
in~\eqref{eq:100}), which in fact  describes the viscous transient
regime. Hence,   different choices of the viscous dissipation $\Psi$
shall give raise to a different behavior in the viscous jumping
regime, see also the example in~\cite[Sec.~2.2]{stefanelli09}. The
latter paper sets forth a different characterization of
rate-independent evolution, still oriented towards local stability,
but derived from a global-in-time variational principle and  not  a
vanishing viscosity approach.
\end{remark}
\begin{proof}
  \textbf{Ad (1).} The geodesic property follows from the minimality of $\vartheta$
   (cf. with~\eqref{eq:91} in Definition~\ref{def:optimal-transition}).
Then, there holds
\begin{equation}
    \label{eq:110}
    \frac\d{\d r}\ene  t{\vartheta(r)}=
    -\bipo(\dot\vartheta(r),-\rmD\ene t
    {\vartheta(r)})\equiv \ene t{u_+}-\ene t{u_-}
    \quad \text{for a.a. }r\in (0,1),
  \end{equation}
  where the first identity ensues from the chain
  rule~\eqref{e:classical-chain_rule} for $\cE $ and the contact
  condition \eqref{eq:106}, and the second one
  from~\eqref{e:constant-finsler}. Clearly, \eqref{eq:110} implies
  \eqref{eq:97}.
\\
  \textbf{Ad (2).} If $\vartheta$ is of sliding type, then the contact
  condition \eqref{eq:106}, with \eqref{eq:32b},
  yields \eqref{eq:98}; \eqref{eq:95} follows since $\dot\vartheta\neq
  0$ a.e.\ in $(0,1)$.
\\
  \textbf{Ad (3).} Equation~\eqref{eq:99} still follows from
  \eqref{eq:32b}. Choosing $r_0\in (0,1)$ and a Borel selection
  $\eps(r)\in \fl(\dot\vartheta(r),-\rmD\ene t{\vartheta(r)})$ (which
  is therefore locally bounded away from $0$), we set
  \begin{equation}
    \label{eq:113}
    \sfs(r):=\int_{r_0}^r \eps^{-1}(\rho)\,\d \rho,\quad \sfr:=\sfs^{-1},
  \end{equation}
  so that $\sfr$ is defined in a suitable interval of $\R$ and satisfies
   \[
    \dot\sfr(s)=\eps(\sfr(s)),\quad
    \dot\theta(s)=\eps(\sfr(s))\vartheta(\sfr(s)).
    \]
  \textbf{Ad (4).} We simply introduce the disjoint open sets
   \[
    V:=\Big\{r\in (0,1):-\rmD\ene t{\vartheta(r)}\not\in \Kx\Big\},\quad
    S:=(0,1)\setminus \overline V
    \]
  and we consider their canonical decomposition in connected components.  \\
  \textbf{Ad (5).}
  If $\vartheta$ is energetic, then
  by~\eqref{eq:116} and~\eqref{eq:97} there holds
  $\Cost{\bipcE}t{u_-}{u_+}=\Diss(u_+-u_-)$. Thus, taking into
  account~\eqref{e:constant-finsler} and \eqref{eq:78} as well, we
  find
  $\bipo(\dot\vartheta,-\rmD\ene t{\vartheta(r)})=\Diss(\dot\vartheta(r))$ for a.a.\
  $r\in (0,1)$. Since its $\Psi_0$-velocity is constant and the total length is
  $\Diss(u_+-u_-)$,
  we deduce that $\vartheta$ is a constant speed minimal geodesic for $\Diss$.
  Conversely, the constraint $-\rmD\ene t{\vartheta(r)}\in \Kx$
 satisfied by sliding transitions
  yields, in view of~\eqref{eq:112}, that
  $\bipo(\dot\vartheta,-\rmD\ene t{\vartheta(r)})=\Diss(\dot\vartheta(r))$ for a.a.\
  $r\in (0,1)$.   Therefore,
  \begin{displaymath}
    \Cost{\bipcE}t{u_-}{u_+}=\int_0^1 \Diss(\dot\vartheta(r))\,\d r=\Diss(u_+-u_-)
\end{displaymath}
by the geodesic property \eqref{eq:96}.

It is well known that, if $\Psi_0$  has strictly convex sublevels,
the  related geodesics are linear segments. In order to prove the
last statement, let us observe that for every $\xi\in
\partial\Psi_0(u_+-u_-)\subset \Kx$ there holds
\begin{displaymath}
    \int_0^1 \la \xi,\dot\vartheta(r)\ra \,\d r=
    \la \xi,u_+-u_-\ra=\Diss(u_+-u_-)=\int_0^1 \Psi_0(\dot\vartheta(r))\,\d
    r\,,
  \end{displaymath}
  where the second equality follows from the
  characterization~\eqref{eq:18} of $\partial \Diss (u_+-u_-)$.
  Hence,
  \begin{displaymath}
    \int_0^1 \Big(\Diss(\dot\vartheta(r))-\la \xi,\dot\vartheta(r)\ra\Big)\,\d r=0.
  \end{displaymath}
  Since the above integrand is nonnegative (being $\xi\in \Kx$),
again by~\eqref{eq:18} we deduce that
    $\xi\in \partial\Diss(\dot\vartheta(r))$
  for a.a.\ $r\in (0,1)$.
  On the other hand, if $\Diss$ is G\^{a}teaux-differentiable outside $0$,
  its subdifferential contains just one point.  Ultimately,  \eqref{eq:98} (recall that
    $\vartheta$ is of sliding type)  shows that
  $-\rmD\ene t{\vartheta(r)}=\xi$ for every $r\in [0,1]$.
\end{proof}

The next result clarifies the relationships between energetic and
$\BV$ solutions.
\begin{corollary}[Energy balance and comparison with energetic solutions]
  \label{cor:comparison}
  \
  \begin{enumerate}[\rm (1)]
  \item
    A $\BV$ solution $u$ of the rate-independent system {\VRIS}
    satisfies the energy balance \eqref{eq:2} if and
    only if every  optimal transition associated with
    its jump set is of energetic type \eqref{eq:116}.
  \item
    A $\BV$ solution $u$ is an energetic solution
    if and only if it satisfies the global stability condition \eqref{eq:1}.
    In that case, all of its optimal transition curves are of
    energetic type.
  \item
    Conversely, an energetic solution $u$ is a $\BV$ solution if and only
    if, for every $t\in \rmJ_u$,
    any jump couple $(u(t_-),u(t_+))$  can be connected by
    a sliding optimal transition.
  \end{enumerate}
\end{corollary}
\begin{proof}
  \textbf{Ad (1).}  Let $u$ be a $\BV$ solution such that every
  optimal transition is of energetic type~\eqref{eq:116}. Now, taking
  into account~\eqref{eq:67}, one sees that \eqref{eq:116} is
  equivalent to the jump conditions \eqref{eq:57}. Then, equation
  \eqref{eq:66bis} (which holds by Theorem~\ref{thm:def-bv-solutions})
  and \eqref{eq:57} yield the energy balance \eqref{eq:2} (cf. the
  proofs of Propositions~\ref{le:global-diff}
  and~\ref{le:local-diff}). The converse implication ensues by
  analogous arguments.
  \\
  \textbf{Ad (2).} The necessity is obvious; for the sufficiency we
  observe that, for every jump point $t \in \rmJ_u$, the global
  stability condition \eqref{eq:1} (written first for $u(t_-)$ with
  test functions $v=u(t)$ and $v=u(t_+)$, and then for $u(t)$ with
  $v=u(t_+)$), yields
\begin{align*}
 \Diss(u(t)-u(t_-))&\ge& \ene t{u(t_-)}-\ene t{u(t)}
     &=&\Cost{\bipcE}t{u(t_-)}{u(t)}&\ge&\Diss(u(t)-u(t_-)),\\
 \Diss(u(t_+)-u(t_-))&\ge& \ene t{u(t_-)}-\ene t{u(t_+)}
     &=&\Cost{\bipcE}t{u(t_-)}{u(t_+)}&\ge&\Diss(u(t_+)-u(t_-)),\\
 \Diss(u(t_+)-u(t))&\ge& \ene t{u(t)}-\ene t{u(t_+)}
&=&\Cost{\bipcE}t{u(t)}{u(t_+)}&\ge& \Diss(u(t_+)-u(t)),
\end{align*}
where the intermediate equalities are due to~\eqref{eq:91} and the
subsequent inequalities to~\eqref{eq:82}. The resulting identities
ultimately show that the transition is energetic, by the very
definition \eqref{eq:116}.
\\
\textbf{Ad (3).} The condition is clearly sufficient. It is also
necessary by the previous point, since energetic transitions are in
particular of sliding type.
  \end{proof}
\subsection{Viscous limit}
\label{subsec:viscous-limit}
We conclude this section by our main asymptotic results:
\begin{theorem}[Convergence of viscous approximations to $\BV$ solutions]
  \label{thm:convergence1}
  Consider a sequence \\ $(u_\eps)\subset \AC([0,T];\xfin)$ of  solutions of the viscous
  equation \eqref{e:ris_eps}, with $u_\eps(0)\to u_0$ as
  $\eps\downarrow0$.

  Then, every vanishing sequence $\eps_k\downarrow0$ admits a further subsequence (still denoted by
  $(\eps_k)$), and a limit function $u\in \BV([0,T];\xfin)$ such that
\begin{equation}
\label{e:conve}
\begin{gathered}
   u_{\eps_k}(t)\to u(t) \quad
  \text{for every $t\in [0,T]$ as $k\up+\infty$,}
  \end{gathered}
  \end{equation}
 and $u$ is a $\BV$ solution of \eqref{e:ris_1},
   induced by the {\bipotential} $\bipo$ according to Definition \ref{def:BV-solution}.
\end{theorem}
\begin{proof}
It follows from the discussion developed in
Section~\ref{subsec:pointwise} that for every sequence
$\eps_k\downarrow0 $ there exists a not relabeled subsequence
$(u_{\eps_k})$ such that~\eqref{e:conve} holds, and $u$ complies
with the local stability condition
  \eqref{eq:65bis}.
In view of  Proposition \ref{prop:useful}, it is then sufficient to
check  that \eqref{eq:85} holds. The latter energy inequality is
  a direct consequence of the $\eps$-energy identity
  \eqref{e:enid-eps} and the lower semicontinuity property stated in Lemma \ref{le:main-lsc}
  later on.
\end{proof}

Our next result  concerns the convergence of the discrete solutions
to the \emph{viscous time-incremental} problem~\eqref{eq:61}, as
both the viscosity parameter $\eps$ and the time-step $\tau$ tend to
$0$.
\begin{theorem}[Convergence of discrete solutions of the viscous incremental problems]
  \label{thm:convergence1bis}
  Let $\overline{\Uu}_{\tau,\eps}:[0,T]\to\xfin$ be the left-continuous
  piecewise constant interpolants
  of the discrete solutions of the viscous incremental problem
  \eqref{eq:61},
  with $\Uu_{\tau,\eps}^0\to u_0$ as $\eps,\tau\downarrow0$.

  Then, all vanishing sequences $\tau_k,\eps_k\downarrow0$ satisfying
  \begin{equation}
    \label{eq:103}
    \lim_{k\downarrow0} \frac{\eps_k}{\tau_k}=+\infty
  \end{equation}
  admit  further subsequences (still denoted by
  $(\tau_k) $ and $(\eps_k$)) and a limit function $u\in \BV([0,T];\xfin)$ such that
\[
\overline{\Uu}_{\tau_k,\eps_k}(t)\to u(t) \quad \text{
  for every $t\in [0,T]$ as $k\up+\infty$,}
  \]
   and $u$ is a $\BV$ solution of \eqref{e:ris_1}
  induced by the {\bipotential} $\bipo$ according to Definition \ref{def:BV-solution}.
\end{theorem}
\noindent
 The reader may compare this result
to~\cite{DT02MQGB, KnMiZa07?ILMC, KnMiZa08?CPPM,  Roub08?RIPV},
where the same double passage to the limit was performed for
specific applied problems and conditions analogous to~\eqref{eq:103}
were imposed.
\begin{proof}
  The standard energy estimate associated with the variational problem \eqref{eq:61} yields
  \begin{equation}
    \label{eq:144}
    \frac{\tau}\eps\Psi\Big(\frac \eps\tau(U^n_{\tau,\eps}-U^{n-1}_{\tau,\eps})\Big)+
    \ene{t_n}{U^n_{\tau,\eps}}\le \ene{t_n}{U^{n-1}_{\tau,\eps}}=
    \ene{t_{n-1}}{U^{n-1}_{\tau,\eps}}+\int_{t_{n-1}}^{t_n}\partial_t\ene t{U^{n-1}_{\tau,\eps}}\,\d
    t\,.
  \end{equation}
  Thanks to \eqref{assene1}, we easily get from~\eqref{eq:144} the following uniform bounds for every $1\le n\le N$
  (here $C$ is a constant independent of $n,\tau,\eps$)
   \[
    \ene{t_n}{U^n_{\tau,\eps}}\le C,\quad
     \sum_{n=1}^N \frac{\tau}\eps\Psi\Big(\frac \eps\tau(U^n_{\tau,\eps}-U^{n-1}_{\tau,\eps})\Big)\le C,\quad
     \sum_{n=1}^N \Diss(U^n_{\tau,\eps}-U^{n-1}_{\tau,\eps})\le C\,,
     \]
   the latter estimate thanks to~\eqref{eq:20}.

   Denoting by  ${\underline{\Uu}}_{\tau,\eps}$ (resp.~$\Uu_{\tau,\eps}$)
   the
   \emph{right-continuous}
   piecewise constant interpolants
   (resp.~piecewise linear interpolant)
   of the discrete values $(U^n_\taueps)$
   which take the
   value $U^n_{\tau,\eps}$
   at $t=t_n$, we have
   \begin{subequations}
 \begin{gather}
    \label{eq:48}
    \ene t{\overline{\Uu}_{\tau,\eps}(t)}\le C,\quad
    \Var{\Psi_0}{\Uu_{\tau,\eps}}0T\le C\\
    \label{eq:181}
    \|\Uu_{\tau,\eps}-\overline{\Uu}_{\tau,\eps}\|_{L^\infty (0,T;\xfin)},\,
    \|\Uu_{\tau,\eps}-\underline{\Uu}_{\tau,\eps}\|_{L^\infty (0,T;\xfin)}, \le
    \sup_n\|U^n_{\tau,\eps}-U^{n-1}_{\tau,\eps}\|_\xfin\le C\omega(\tau /C\eps),\quad
  \end{gather}
  \end{subequations}
  where
   \[
    \omega(r):=\sup_{x\in \xfin}\big\{\|x\|_\xfin:r\Psi(r^{-1}x)\le 1\big\}
    \]
  satisfies $\lim_{r\down0}\omega(r)=0$ thanks to \eqref{eq:19}.
  By Helly's theorem, these bounds show that (up to the extraction of suitable
  subsequences $(\tau_k)$ and $( \eps_k )$
  satisfying \eqref{eq:103}), the sequences
  $(\Uu_{\tau_k,\eps_k}),$
  $ (\overline{\Uu}_{\tau_k,\eps_k})$ and $ (\underline{\Uu}_{\tau_k,\eps_k})$   pointwise converge to
  the same limit $u$.

  By differentiating the variational characterization of $U^n_{\tau,\eps}$ given by
  \eqref{eq:61} we obtain
   \[
    \partial\Psi_\eps\Big(\frac{U^n_{\tau,\eps}-U^{n-1}_{\tau,\eps}}\tau\Big)+W^n_{\tau,\eps}\ni0,\quad
    W^n_{\tau,\eps}:=-\rmD\ene{t_n}{U^n_{\tau,\eps}},
    \]
  which yields in each interval $(t_{n-1},t_n]$ (here,  $\overline {\Ww}_{\tau,\eps}$
denotes the left-continuous piecewise constant interpolant of the
values
  $(W^n_{\tau,\eps})_{n=1}^N$)
  \begin{align*}
    &\tau \Psi_\eps\big(\dot{\Uu}_{\tau,\eps}\big)+\tau\Psi_\eps^*(\overline{\Ww}_{\tau,\eps})=
    -\left\la \rmD \ene{t_n}{\Uu_{\tau,\eps}(t_n)}, \Uu_{\tau,\eps}(t_n)-\Uu_{\tau,\eps}(t_{n-1}) \right\ra\\&=
    \ene{t_{n-1}}{\Uu_{\tau,\eps}(t_{n-1})}-\ene{t_n}{\Uu_{\tau,\eps}(t_n)}+
    \int_{t_{n-1}}^{t_n} \partial_t \ene t{\underline{\Uu}_{\tau,\eps}(t)}\,\d t
    -R(t_n;\Uu_{\tau,\eps}(t_{n-1}),\Uu_{\tau,\eps}(t_{n}))
  \end{align*}
  where
   \[
    R(t;x,y):=\ene t{y}-\ene tx-
    \la \rmD \ene{t}{y}, y-x\ra.
    \]
  Since $\cE$ is of class $\mathrm{C}^1$, for every convex and bounded set $B\subset \xfin$ there exists a concave modulus of continuity
  $\sigma_B:[0,+\infty)
  \to [0,+\infty)$ such that $\lim_{r\downarrow0}\sigma_B(r)=\sigma_B(0)=0$ and
  \[
    R(t;x,y)\le \sigma_B(\|y-x\|_\xfin)\|y-x\|_{\xfin}\quad
    \text{for every }t\in [0,T],\ x,y\in B.
    \]
  We thus obtain
  \begin{equation}
    \label{eq:180}
    \begin{aligned}
      &\int_0^{T}\Big(\Psi_\eps(\dot{\Uu}_{\tau,\eps}(t))
      +\Psi_\eps^*(\overline{\Ww}_{\tau,\eps}(t))\Big)\,\d t+ \ene{t_N}{\Uu_{\tau,\eps}(t_N)}\le \ene
      0{u_0}+\int_0^{t_N}\partial_t \ene t{\underline{\Uu}_{\tau,\eps}(t)}\,\d
      t\\&+ \sup_{1\le n\le
        N}\sigma_B(\|U^n_{\tau,\eps}-U^{n-1}_{\tau,\eps}\|)\sum_{n=1}^N
      \|U^n_{\tau,\eps}-U^{n-1}_{\tau,\eps}\|,\quad
      \overline{\Ww}_{\tau,\eps}(t))=-\rmD\ene{\bar \sft_\tau(t)}{\overline{\Uu}_{\tau,\eps}(t)}.
\end{aligned}
\end{equation}
We pass to the limit along suitable subsequences $(\tau_k)$ and
$(\eps_k)$ such that ${{\Uu}}_{\tau_k,\eps_k},\,
{\overline{\Uu}}_{\tau_k,\eps_k}\to u$ pointwise; since
$\Uu_{\tau,\eps}$ and $\overline{\Uu}_{\tau,\eps}$ are uniformly
bounded, \eqref{eq:181} and \eqref{eq:103} yield the convergence to
$0$ of the third term on the right-hand side of \eqref{eq:180}, which
thus tends to $\ene0{u_0}+\int_0^T \partial_t \ene t{u(t)}\,\d t$.
Since $\overline{\Ww}_{\tau_k,\eps_k}(t)\to w(t)=-\rmD\ene t{u(t)}$,
applying the lower semicontinuity result of Lemma \ref{le:main-lsc} we
obtain that $u$ satisfies \eqref{eq:85} and the local stability
condition. In view of Proposition~\ref{prop:useful}, this concludes
the proof.
\end{proof}

\section{Parametrized solutions}
\label{sec:parametrized}

In this section, we restart from the
discussions in Sections~\ref{ss:2.1} and \ref{subsec:pointwise}, and
adopt a different point of view, which relies on the
rate-independent structure of the limit problem. The main idea,
which was introduced by \cite{Efendiev-Mielke06}, is to rescale time
in order to gain a uniform Lipschitz bound on the (rescaled) viscous
approximations. Keeping track of the asymptotic behavior of time
rescalings, one can retrieve the $\BV$ limit we analyzed in Section
\ref{sec:BV-solutions}.  In particular, we shall recover that the
limiting jump pathes reflect the viscous approximation.

\subsection{Vanishing viscosity analysis: a rescaling argument.}
\label{subsec:vanishing2}

Let us recall that for every $\eps>0$
$u_\eps$ are the solutions of the viscous differential inclusion
\begin{equation}
  \label{eq:102} \partial\Psi_\eps(\dot u_\eps(t))
  + \mathrm{D} \cE_t (u_\eps(t)) \ni 0 \quad \text{in $\xfins$ \quad for a.a. $t \in (0,T),$}
  \tag{DN$_\eps$}
\end{equation}
which we split into the system
\[
  \begin{aligned}
    \partial\Psi_\eps(\dot u_\eps(t))&\ni w_\eps,\\
    \rmD\ene t{u_\eps(t)}&=-w_\eps,\qquad
    \partial_t\ene t{u_\eps(t)}=-\pt_\eps.
  \end{aligned}
\]
We follow the ideas
of~\cite{Efendiev-Mielke06,Mielke-Rossi-Savare08} to
capture the aforementioned limiting viscous jump pathes,. However, owing to the dissipation bound \eqref{e:est-1},
we use a different time rescaling $\sfs_\eps:[0,T]\to[0,\sfS_\eps]$
\begin{equation}
\label{e:resc1}
 \rescs_\eps (t):= t
 + \int_0^t \Big( \Psi_\eps(\dot{u}_\eps(r))+\Psi_\eps^*(w_\eps(r))\Big)\dd
 r\quad \text{and} \quad
 \rescT_\eps:=\rescs_\eps(T).
\end{equation}
Thus, $\rescs_\eps$ may be interpreted as some sort of ``energy
arclength'' of the curve $u_\eps$. Notice that,
 thanks to~\eqref{e:est-1}, the sequence $(\rescT_\eps)$ is uniformly bounded with respect
 to the parameter~$\eps$.
 Let us
 consider the rescaled functions  $(\resct_\eps,\rescu_\eps):
[0,\rescT_\eps] \to [0,T]\times \xfin$ and
$(\rescpt_\eps,\rescw_\eps): [0,\rescT_\eps] \to \R\times \xfin^*$
defined by
\begin{equation}
\label{e:resc2}
\begin{aligned}
  \resct_\eps (s)&:=\rescs_\eps^{-1}(s)\,, &\rescu_\eps(s)&:= u_\eps
  (\resct_\eps (s)),
  \\
   \rescpt_\eps(s)&:=\pt_\eps(\resct_\eps(s))=-\partial_t\ene{\resct_\eps(s)}{\rescu_\eps(s)}\,,\qquad
   &\rescw_\eps(s)&:=w_\eps(\resct_\eps(s))=-\rmD\ene{\resct_\eps(s)}{\rescu_\eps(s)}.
  \end{aligned}
\end{equation}
We now study the limiting behavior as $\eps \down 0$ of the
\emph{reparametrized trajectories}
\[
\begin{aligned}
   \left\{
    \big(\resct_\eps(s), \rescu_\eps(s)\big)\, : \ s
    \in [0,\rescT_\eps] \right\}&\subset \XX =[0,T] \times \xfin,\\
  \left\{
    \big(\dot\resct_\eps(s), \dot\rescu_\eps(s);\rescpt_\eps(s), \rescw_\eps(s)\big)\, : \ s
    \in [0,\rescT_\eps] \right\}&\subset\BipDom,
\end{aligned}
\]
where we use the notation
\begin{equation}
\label{eq:11-notation} \BipDom:=[0,+\infty) \times \xfin\times
\R\times
    \xfins.
\end{equation}
In order to rewrite the ``rescaled  energy identity'' fulfilled by
the triple $(\resct_\eps,\rescu_\eps,\rescw_\eps)$, we define the
viscous space-time {\bipotential} $\Bip_\eps: (0,+\infty) \times
\xfin \times \R\times \xfins \to [0,+\infty)$ by setting
\begin{equation}
  \label{def-meps} \Bip_\eps (\alpha,\sfv;\sfp,\sfw):= \alpha\Psi_\eps(\sfv/\alpha)+\alpha \Psi_\eps^*(\sfw)+
  \alpha \sfp=
  \frac \alpha\eps\Psi(\frac \eps\alpha\sfv)+
  \frac\alpha\eps\Psi^*(\rescw)+\alpha\sfp
\end{equation}
 Hence, \eqref{e:enid-eps} becomes for all $0 \leq s_1 \leq
s_2 \leq \rescS_\eps$
\begin{equation}
\label{resc-enid-eps}
  \int_{s_1}^{s_2} \Bip_\eps \big( \dot{\resct}_\eps(s),
  \dot{\rescu}_\eps(s);\rescpt_\eps(s),\rescw_\eps(s)\big)
 \dd s
 +\ene {\resct_\eps(s_2)}{\rescu_\eps(s_2)}
 =  \ene {\resct_\eps(s_1)}{\rescu_\eps(s_1)},
\end{equation}
and \eqref{e:resc1} yields
\[
  \Bip_\eps\big(
  \dot{\resct}_\eps(s),\dot{\rescu}_\eps(s);1,\rescw_\eps(s)\big)=1\quad
  \forae\, s \in (0,\rescT_\eps)\,.
\]
\subsection*{A priori estimates and passage to the limit.}
Due to estimate~\eqref{e:est-1}, there exists $\rescS>0$ such that,
along a (not relabeled) subsequence, we have $ \rescs_\eps(T) \to
\rescS$ as $\eps\down 0$. Exploiting again~\eqref{e:est-1}, the
Arzel\`{a}-Ascoli compactness theorem, and the fact that $\xfin$ is
finite-dimensional  (see also the proof
of~\cite[Thm.~3.3]{Mielke-Rossi-Savare08}),  we find two curves
$\resct \in W^{1,\infty} (0,\rescS)$ and $\rescu \in
W^{1,\infty}([0,\rescS];\xfin)$ such that,  along the same
subsequence,
\begin{subequations}
    \label{e:convergences-1}
    \begin{align}
   \resct_\eps \to \resct\ &\text{in ${\rm C}^0 ([0,\rescS])$,}
    \quad&\dot{\resct}_\eps \weaksto \dot{\resct}\ &\text{in $L^\infty
      (0,\rescS)$,}
    \\
    \rescu_\eps \to \rescu \ &\text{in ${\rm C}^0
      ([0,\rescS];\xfin)$,}\quad &\dot{\rescu}_\eps \weaksto
    \dot{\rescu} \ & \text{in $L^\infty (0,\rescS;\xfin)$,}\\
    \rescpt_\eps\to \rescpt \ &\text{in ${\rm C}^0 ([0,\rescS])$,}
    \quad
    &\rescw_\eps\to\rescw \ &\text{in }{\rm C}^0([0,\rescS];\xfins),
    \end{align}
    with
    \begin{equation}
    \ene{\resct_\eps }{\rescu_\eps} \to \ene{\resct }{\rescu},\quad 
    \rescpt(s)=-\partial_t\ene{\resct(s)}{\rescu(s)},\quad 
    \rescw(s)=-\rmD
    \ene{\resct(s)}{\rescu(s)}\label{eq:189}
  \end{equation}
\end{subequations}
for all $s \in [0,\rescS]$. Then,  to pass to the limit in
\eqref{resc-enid-eps} we exploit a lower semicontinuity result (see
Proposition \ref{prop:techn2}), based on the fact that the sequence
of functionals $(\Bip_\eps)$ $\Gamma$-converges to the
\emph{augmented} {\bipotential} $\Bip: [0,+\infty) \times \xfin
\times \R\times \xfins \to [0,+\infty]$ (see Lemma \ref{le:techn1})
defined by
  \begin{equation}
    \label{e:gamma-conv1} \Bip(\alpha, v;p,w):=
    \begin{cases}
      \Diss(v)+\rmI_{\Kx}(w)+\alpha\,p&\text{if }\alpha>0,\\
      \bipo(v,w)&\text{if }\alpha=0.
    \end{cases}
  \end{equation}
 By \eqref{e:convergences-1} and
 Proposition \ref{prop:techn2}, we take the $\liminf$ as $\eps \down 0$
of~\eqref{resc-enid-eps} and conclude that the pair
$(\resct,\rescu)$ fulfils, for all $0 \leq s_1 \leq s_2 \leq \rescS$, the estimate
\begin{equation}
\label{e:enid-lim} \int_{s_1}^{s_2} \Bip \big(
\dot{\resct}(s),\dot{\rescu}(s);\rescpt(s),\rescw(s) \big) \dd s
+\ene {\resct(s_2)}{\rescu(s_2)} \leq  \ene
{\resct(s_1)}{\rescu(s_1)}\,.
\end{equation}

\subsection{{\Bipotential}s and rate-independent evolution}
\label{subsec:Bip_properties}

The augmented space-time {\shortbipotential} $\Bip$ is closely related
to $\bipo$ introduced by \eqref{eq:68bis}. The following result fixes
some properties of $\Bip$. Its proof, which we choose to omit, can be
easily developed starting from Theorems \ref{thm:bipotential} and
\ref{thm:auxiliary} for the {\bipotential} $\bipo$. 

\begin{lemma}[General properties of $\Bip$]
  \label{le:auxiliary2}
  \
  \begin{enumerate}[\rm (1)]
  \item $\Bip$ is lower semicontinuous, $1$-homogeneous and convex in the
    pair $(\alpha,v)$; for every $(\alpha,v) \in [0,+\infty) \times \xfin$ the function $\Bip(\alpha,v;\cdot,\cdot)$ has convex sublevels.
  \item For all $(\alpha,v,p,w) \in \BipDom$ (cf.~\eqref{eq:11-notation})
   it  satisfies
    \begin{gather}
      \label{eq:13}
      \Bip(\alpha,v;p,w)\ge \la w, v \ra+\alpha p,\quad
      \Bip(0,v;p,w)\ge\bipo(v,w)\ge \Diss(v),\\
      \label{eq:14}
      \Bip(0,v;p,w)=\Diss(v)\quad \Leftrightarrow\quad
      w\in \Kx.
    \end{gather}
  \item The \emph{contact set} of $\Bip$
    \begin{equation}
      \label{eq:53}
      \Contact\Bip:=\Big\{(\alpha, v;p,w)\in \BipDom: \Bip(\alpha,v;p,w)=\langle w,v\rangle+\alpha p\Big\}
    \end{equation}
    does not impose any constraint on $p$.
      It can be characterized by
      \begin{equation}
        \label{eq:195}
        (\alpha, v;p,w)\in \Contact\Bip\quad\Leftrightarrow\quad
        w\in \partial\,\Bip(\alpha,\cdot\,;p,w)(v).
      \end{equation}
    We also have
    \begin{align}
      \label{eq:15}
      \text{for $\alpha>0$,}&\quad (\alpha,v;p,w)\in \Contact\Bip\quad
      \text{if and only if}
      \quad w\in \partial\Diss(v),\\
      \label{eq:15bis}
      \text{for $\alpha=0$,}& \quad (\alpha,v;p,w)\in \Contact\Bip\quad
      \text{if and only if}\quad (v,w)\in \Contact\bipo.
    \end{align}
    Equivalently, $ (\alpha,v;p,w)\in \Contact\Bip$ if and only if
    \begin{equation}
      \label{eq:32}
      w\in \partial\Diss(v)\subset \Kx\qquad\text{or}\qquad
      \Big(w\not\in \Kx,\quad \alpha=0,\quad \exists\,\eps\in \fl(v,w): \
      \ w\in \partial\Psi(\eps v)\Big),
    \end{equation}
    where $\fl(v,w)$ is defined in \eqref{eq:111}.
    In particular, in the additive viscosity case~\eqref{eq:5}, we simply have
    \begin{equation}
      \label{eq:33bis}
      (\alpha,v;p,w)\in \Contact\Bip\quad
      \Longleftrightarrow\quad
      \exists\,\lambda\ge 0:\quad
      w\in \partial\Diss(v)+\partial\Psi_V(\lambda v)\quad\text{and}\quad \alpha\lambda=0.
    \end{equation}
  \end{enumerate}
\end{lemma}
\paragraph{\textbf{Conclusion of the vanishing viscosity analysis.}}
We are now going to show that~\eqref{e:enid-lim} is in fact  an
 equality. This can be easily checked relying on
 the chain rule~\eqref{e:classical-chain_rule}, which
yields $\forae\, s \in (0,\rescS)$
\begin{equation}
\label{e:crucial-1}
\begin{aligned}
 \frac{\rmd}{\rmd s}\ene{\resct(s)}{\rescu(s)}&=-\partial_t
\ene{\resct(s)}{\rescu(s)}\, \dot{\resct}(s) -\langle -\mathrm{D}
\ene{\resct(s)}{\rescu(s)}, \dot{\rescu}(s) \rangle\\ & \topref{e:resc2}=
-\rescpt(s)\dot\resct(s)-\la \rescw(s),\dot\rescu(s)\ra\ge
-\Bip (\dot{\resct}(s), \dot{\rescu}(s),\rescpt(s),\rescw(s))
\,.
\end{aligned}
\end{equation}
Collecting~\eqref{e:crucial-1} and~\eqref{e:enid-lim}, we conclude
that the latter holds with an equality sign and, with an elementary
argument, that such equality also holds in the differential form,
namely $\forae\, s \in (0,\rescS)$
\begin{equation}
\label{def:mprif-1}
\begin{aligned}
  \rescpt(s)&=-\partial_t \ene{\resct(s)}{\rescu(s)},\qquad
  \rescw(s)=-\mathrm{D} \ene{\resct(s)}{\rescu(s)}\\
  \frac{\rmd}{\rmd s}\ene{\resct(s)}{\rescu(s)}&=
  -\rescpt(s)\dot\resct(s)-\la \rescw(s),\dot\rescu(s)\ra=-\Bip (\dot{\resct}(s), \dot{\rescu}(s),\rescpt(s),\rescw(s))
\end{aligned}
\end{equation}
which yields
\begin{equation}
  \label{eq:25bis}
  \Big(\dot\resct(s),\dot\rescu(s);-\partial_t \ene {\resct(s)}{\rescu(s)},-\rmD\ene {\resct(s)}{\rescu(s)}\Big)
  \in \Contact\Bip\quad\text{for a.a.\ }s\in (0,\rescS).
\end{equation}
Finally, we take the $\limsup$ as $\eps \down 0$
of~\eqref{resc-enid-eps}, using~(\ref{e:convergences-1})
and~\eqref{def:mprif-1}, whence
\[
\begin{aligned}
  \limsup_{\eps\down0}\int_0^\rescS \Bip_\eps(\dot{\resct}_\eps(s),\dot{\rescu}_\eps(s),
  \rescpt_\eps(s),\rescw_\eps(s))\,\d s\le
  \int_0^\rescS \Bip(\dot{\resct}(s), \dot{\rescu}(s),\rescpt(s),\rescw(s))\,\d s.
\end{aligned}
\]
In particular, we find that $\forae\, s \in
 (0,\rescS)$
\begin{equation}
  \label{eq:26bis}
  \Bip\big( \dot{\resct}(s),\dot{\rescu}(s);1,\rescw(s)\big)=1.
\end{equation}
\subsection{Parametrized solutions of rate-independent systems.}
\label{subsec:RIF} Motivated by the discussion of the previous
section,
 we now  give  the notion of
parametrized rate-independent evolution,   driven by a general
{\bipotential} $\Bip$, satisfying conditions $(1),(2)$ of Lemma
\ref{le:auxiliary2}.
\begin{definition}[Parametrized solutions of rate-independent systems]
\label{def:2.1} Let $\Bip:\BB\to ({-}\infty,+\infty]$ be
  the {\bipotential} \eqref{e:gamma-conv1}.
We say
that a Lipschitz continuous curve $(\resct,\rescu): [a,b] \to [0,T]
\times \xfin$ is a {\em parametrized rate-independent solution for the system {\PVRIS}}
if $\resct$ is nondecreasing and, setting
$\rescpt(s)=-\partial_t \ene{\resct(s)}{\rescu(s)},\ \rescw(s)=-\rmD
\ene{\resct(s)}{\rescu(s)}$ for all $s \in [a,b]$, we have
\begin{equation}
  \label{eq:16}
  \int_{s_1}^{s_2} \Bip(\dot \resct(s),\dot \rescu(s);\rescpt(s),\rescw(s))\,\d s+\ene{\resct(s_2)}{\rescu(s_2)}\le
  \ene{\resct(s_1)}{\rescu(s_1)}\quad
  \forall\, a\le s_1\le s_2\le b.
\end{equation}
Furthermore,
\begin{enumerate}
\item
 if $\dot\resct(s)+\Diss(\dot\rescu(s))>0$ for a.a.
$s\in (a,b)$ we say that $(\resct,\rescu)$ is \emph{nondegenerate};
\item
 if $\resct(a)=0,\resct(b)=T$ we say that $(\resct,\rescu)$ is
\emph{surjective};
\item
 if $(\resct,\rescu)$ satisfies \eqref{eq:26bis},
we say that it is \emph{normalized}.
\end{enumerate}
\end{definition}
Definition~\ref{def:2.1}  generalizes to the present setting the
notion which we first introduced in~\cite{Mielke-Rossi-Savare08}.
\begin{remark}
\upshape \label{rmk:crucial-feat} The nice feature of the
previous definition is its invariance with respect to
(nondecreasing, Lipschitz) time rescalings. Namely,  if
$(\resct,\rescu):[a,b]\to [0,T]\times \xfin$ is a parametrized solution
and $\rescs:[\alpha,\beta]\to [a,b]$ is a Lipschitz nondecreasing
map, then $(\resct\circ\rescs,\rescu\circ\rescs)$ is a parametrized
solution in $[\alpha,\beta]$. \end{remark}

 The next result provides
equivalent characterizations of parametrized solutions.
\begin{proposition}
  \label{prop:smooth_parametrized} 
  A Lipschitz continuous curve $(\resct,\rescu): [a,b]
  \to [0,T] \times \xfin$,  with $\resct$ nondecreasing, is a
  parametrized solution of {\PVRIS} if
  and only if one of the following (equivalent) conditions (involving as usual
  $\rescpt=-\partial_t\ene{\resct}{\rescu},\rescw=-\rmD\ene\resct\rescu$)
is satisfied:
\begin{enumerate}[\rm (1)]
\item The energy inequality \eqref{eq:16} holds just for $s_1=a$ and $s_2=b$, i.e.
  \begin{equation}
    \label{eq:35}
    \int_{a}^{b} \Bip(\dot \resct(s),\dot \rescu(s);\rescpt(s),\rescw(s))\,\d s+\ene{\resct(b)}{\rescu(b)}\le
    \ene{\resct(a)}{\rescu(a)}.
  \end{equation}
\item The energy inequality \eqref{eq:16}  holds in the differential form
  \begin{equation}
    \label{eq:17}
    \frac\d{\d s}\ene{\resct(s)}{\rescu(s)}+\Bip(\dot\resct(s),\dot\rescu(s);\rescpt(s),\rescw(s))\le 0\quad
    \text{for a.a.\ $s\in (a,b)$}.
  \end{equation}
\item The energy identity holds, in the differential form
  \begin{equation}
    \label{eq:17bis}
    \frac\d{\d s}\ene{\resct(s)}{\rescu(s)}+\Bip(\dot\resct(s),\dot\rescu(s);\rescpt(s),\rescw(s))= 0
    \quad
    \text{for a.a.\ $s\in (a,b)$,}
  \end{equation}
  or in the integrated form
\begin{equation}
    \label{eq:RIF:13-0}
      \int_{s_1}^{s_2} \Bip(\dot \resct(s),\dot \rescu(s);\rescpt(s),\rescw(s))\,\d s+\ene{\resct(s_2)}{\rescu(s_2)}=
      \ene{\resct(s_1)}{\rescu(s_1)}\quad
  \text{for $ a\le s_1\le s_2\le b.$}
  \end{equation}
\item
There holds
\[
 \big(\dot\resct(s),\dot\rescu(s);-\partial_t
\ene{\resct(s)}{\rescu(s)},-\rmD\ene{\resct(s)}{\rescu(s)}\big)
  \in \Contact\Bip
\quad \forae\, s \in (a,b)\,.
\]
\item The pair $(\resct,\rescu)$ satisfy the differential inclusion
\begin{equation}\label{e:newer}
    \partial\,\Bip\big(\dot\resct(s),\cdot\,;-\partial_t
    \ene{\resct(s)}{\rescu(s)},-\rmD\ene{\resct(s)}{\rescu(s)}\big)(\dot\rescu(s))+\rmD\ene{\resct(s)}{\rescu(s)}\ni0
    \quad \text{ a.e.\ in } (a,b)\,.
\end{equation}
In particular, for a.a.\ $s\in (a,b)$  we have the implications
  \begin{equation} \label{eq:31}
     \begin{array}{rcl}
    \dot\resct (s)>0& \Rightarrow &
    -\rmD\ene {\resct(s)}{\rescu(s)}\in \Kx,\\
     -\rmD\ene {\resct(s)}{\rescu(s)}\in \Kx & \Rightarrow &
    -\rmD\ene{\resct(s)}{\rescu(s)}\in \partial\Diss(\dot
    \rescu(s)),
    \end{array}
  \end{equation}
  and for every Borel
  map 
  $\lambda$ defined in the open set $\mathsf J$ by
  \begin{equation}
    \label{eq:124}
    \begin{gathered}
    \mathsf J:=\big\{s\in (a,b):  -\rmD\ene{\resct(s)}{\rescu(s)}\not\in  \Kx\big\},
    \\
    \text{with} \ \
    \lambda(s)\in \fl(\dot\rescu(s),-\rmD\ene{\resct(s)}{\rescu(s)})
    \ \ \forae\, s\in \mathsf J,
    \end{gathered}
     \end{equation}
  we have
  \begin{equation}
    \label{eq:54}
     -\rmD\ene{\resct(s)}{\rescu(s)}\in \partial\Psi(\lambda(s)\dot\rescu(s)),\quad
     \dot\resct(s)=0\quad\text{for a.a.\ }s\in \mathsf J.
  \end{equation}
\end{enumerate}
\end{proposition}
The \emph{proof} follows from the chain rule
\eqref{e:classical-chain_rule} (arguing as for \eqref{e:crucial-1},
\eqref{def:mprif-1}, \eqref{eq:25bis}), and from the
characterization of the contact set $\Contact\Bip$ of Lemma
\ref{le:auxiliary2} (see also
\cite[Prop.~2]{Mielke-Rossi-Savare08}). 
\begin{corollary}[Differential characterization in the additive viscosity case] 
  \label{p:2.1}
 Let $\Bip:\BB\to ({-}\infty,+\infty]$ be a {\bipotential}
satisfying conditions $(1),(2)$ of Lemma \ref{le:auxiliary2}, and
suppose also
  that the contact set of  $\Bip$ satisfies
the characterization~\eqref{eq:33bis} of
  Lemma \ref{le:auxiliary2} in the additive viscosity case~\eqref{eq:5}
  $\Psi=\Diss+\Psi_V$.

  Then,
   a Lipschitz continuous curve $(\resct,\rescu): [a,b]
  \to [0,T] \times \xfin$ is a parametrized solution of {\PVRIS} if
  and only if there exists a Borel function $\lambda : (a,b) \to[0,+\infty)
  $ such that $\forae\, s \in (a,b)$
  \begin{equation}
    \label{eq:55}
    \partial\Diss(\dot \rescu(s))+\partial\Psi_V(\lambda(s)\dot\rescu (s))+\rmD\ene{\resct(s)}{\rescu(s)}\ni0,\quad
    \lambda(s)\dot\resct(s)=0\quad\text{for a.a.\ }s\in (a,b).
  \end{equation}
\end{corollary}
The vanishing viscosity analysis developed in Sections
\ref{subsec:vanishing2} and \ref{subsec:Bip_properties}
provides the following convergence result.
\begin{theorem}[Convergence to parametrized solutions]
  \label{thm:convergence2}
  Let $(u_{n})$ be viscous solutions of \eqref{eq:102}
   corresponding to a vanishing sequence $(\eps_n)$, let
  $\sft_n:[0,\rescS]\to [0,T]$ be uniformly Lipschitz and surjective
  time rescalings and let $\rescu_n:[0,\rescS]\to\xfin$ be defined as $\rescu_n(s):=
  u_n(\resct_n(s))$ for all $s \in [0,\rescS]$.
  Suppose that
  \[
    \exists\,\alpha>0\, \ \ \forall\, n \in \N\, : \ \
    \sfm_n(s):=\Bip_{\eps_n}(\dot\resct_n(s),\dot\rescu_n(s);1,-\rmD\ene{\resct_n(s)}{\rescu_n(s)})
    \in [\alpha,\alpha^{-1}]
  \]
  for a.a. $s\in (0,\rescS)$.
  If the functions $(\resct_n,\rescu_n,\sfm_n)$ pointwise converge to $(\resct,\rescu,\sfm)$ as $n\to\infty$, then
  $(\resct,\rescu)$ is a (nondegenerate, surjective) parametrized rate-independent solution
  according to Definition \ref{def:2.1}, and
  \[
  \Bipo(\dot\resct(s),\dot\rescu(s);1,-\rmD\ene{\resct(s)}{\rescu(s)})=\sfm(s)
  \qquad \forae\, s \in (0,\rescS).
  \]
\end{theorem}
 The following remark, to be compared with
Remark~\ref{rem:obvious}, highlights  the different mechanical
regimes encompassed in the notion of parametrized rate-independent
solution.
\begin{remark}[Mechanical interpretation]
\label{rmk:mech} \upshape
 The evolution described by \eqref{e:newer} in Proposition~\ref{prop:smooth_parametrized}
  bears the following mechanical interpretation (cf.
  with~\cite{Efendiev-Mielke06} and~\cite{Mielke-Rossi-Savare08}):
\begin{itemize}
\item the regime $(\dot{\resct}>0,\, \dot{\rescu}=0)$ corresponds to
  {\em sticking},
\item the regime $(\dot{\resct}>0,\dot\rescu\neq 0)$ corresponds to
  {\em rate-independent sliding}.  In both these two regimes
  $-\rmD\ene{\resct}{\rescu}\in \Kx$.
\item when $-\rmD\ene{\resct}\rescu$ cannot obey the constraint $\Kx$,
  then the system switches to a \emph{viscous regime}.  The time is
  frozen (i.e., $\dot{\resct}=0$), and the solution follows a viscous
  path.  In the additive viscosity case~\eqref{eq:5} it is governed by
  the rescaled viscous equation \eqref{eq:55} with $\lambda>0$. These
  viscous motions can be seen as a jump in the (slow) external time
  scale.
  \end{itemize}
\end{remark}
We conclude this section with the main equivalence result between
parametrized and $\BV$ solutions of rate-independent systems
(compare  with the analogous \cite[Prop.~6]{Mielke-Rossi-Savare08}).
We postpone its proof at the end of the next section.
\begin{theorem}[Equivalence between $\BV$ and parametrized solutions]
  \label{thm:equivalence}
  Let $(\resct,\rescu):[0,\rescS]\to [0,T]\times\xfin$ be a (nondegenerate, surjective) parametrized solution
  of the rate independent system {\PVRIS}.
  For every $t\in [0,T]$ set
  \begin{equation}
    \label{eq:127}
    \rescs(t):=\big\{s\in [0,\rescS]: \resct(s)=t\big\}
  \end{equation}
  Then, any curve $u:[0,T]\to\xfin$ such that
  \begin{equation}
    \label{eq:126}
    u(t)\in \big\{\rescu(s):s\in \rescs(t)\big\}
  \end{equation}
  is a $\BV$ solution of the rate-independent system {\VRIS}.\\
  Conversely, if $u:[0,T]\to\xfin$ is a $\BV$ solution, then there exists a parametrized solution
  $(\resct,\rescu)$ such that \eqref{eq:126} holds for a time-rescaling function $\rescs$ defined as
  in~\eqref{eq:127}.
\end{theorem}
\section{Auxiliary results}
\label{sec:technical} After proving some lower semicontinuity
results for {\bipotential}s, in Section~\ref{subsec:pseudo} we develop
some auxiliary results concerning the total variation induced by
time-dependent (and possibly asymmetric) Finsler norms.
\subsection{Lower semicontinuity for {\bipotential}s.}
\label{subsec:technical1} Let us start with a lemma which shows that
$\Bip_\eps$, which is defined in \eqref{def-meps}, $\Gamma$-converges to $\Bipo$ as $\eps\down0$ (compare
with \cite[Lemma~3.1]{Mielke-Rossi-Savare08}), where $\Bipo$ is defined in \eqref{e:gamma-conv1}.
\begin{lemma}[$\Gamma$-convergence of $\Bip_\eps$]
  \label{le:techn1}
  \
  \begin{description}
  \item[$\Gamma$-liminf estimate]
 For every choice of sequences $\eps_n\downarrow0$
  and $ (\alpha_n,v_n,p_n,w_n)\to (\alpha,v,p,w)$ in $\BB$, we have
  \begin{equation}
    \label{eq:106bis}
    \liminf_{n\to\infty}\Bip_{\eps_n}(\alpha_n,v_n;p_n,w_n)\ge \Bipo(\alpha,v;p,w).
  \end{equation}
  \item[$\Gamma$-limsup estimate] For every $(\alpha, v;p,w)\in\BB$ there exists $(\alpha_\eps,v_\eps,p_\eps,w_\eps)_{\eps>0}$
    such that
    \begin{equation}
      \label{eq:118}
      \limsup_{\eps\down0}\Bip_{\eps}(\alpha_\eps,v_\eps;p_\eps,w_\eps)
      \leq
       \Bipo(\alpha,v;p,w).
    \end{equation}
  \end{description}
\end{lemma}
{\renewcommand{\qed}{}
\begin{proof}
  The $\Gamma$-liminf estimate is easy: if $\alpha>0$ then,
   also recalling~\eqref{eq:70}, one verifies  that
  \begin{equation}
    \label{eq:119}
    \begin{aligned}
      \liminf_{n\to\infty}\Bip_{\eps_n}(\alpha_n,v_n;p_n,w_n)&\ge
      \liminf_{n\to\infty}\left(\Diss(v_n)+\alpha_n\eps_n^{-1}\DualDiss(w_n)+\alpha_n
      p_n\right)\\&\ge\Bip_{0}(\alpha,v;p,w),
      \end{aligned}
      \end{equation}
      where we have used the notation
\begin{equation}
\label{e:novel} \Bip_{0}(\alpha,v;p,w):= \Diss(v)+\rmI_\Kx(w)+\alpha
p.
\end{equation}
The first inequality in~\eqref{eq:119} is also due
to~\eqref{eq:21-BIS}.  If $\alpha=0$, we use the obvious lower bound
  \[
    \Bip_{\eps_n}(\alpha_n,v_n;p_n,w_n)\ge \bipo(v_n,w_n)+\alpha_n p_n
  \]
  and the continuity of $\bipo$ (cf. Theorem~\ref{thm:auxiliary}).

  To show the limsup estimate \eqref{eq:118} for $w\in \Kx$, we simply choose $\alpha_\eps:=\alpha+\eps,
  v_\eps:=v, p_\eps:=p,w_\eps:=w$, observing that in this case
  \begin{equation}
    \Bip_\eps(\alpha_\eps,v_\eps;p_\eps,w_\eps)\le \eps(\alpha+\eps)\Psi(v/(\eps(\alpha+\eps))+(\alpha+\eps)p
    \stackrel{\eps\down0}\to \Diss(v)+\alpha
    p=\Bipo(\alpha,v;p,w)\,,\nonumber
  \end{equation}
 the first passage due to~\eqref{eq:21-TER}.
  If $w\not\in \Kx$, we choose a coefficient $\lambda\in \fl(v,w)$ as in
  \eqref{eq:111},
  and we set $\alpha_\eps:=\lambda\eps$, $v_\eps:=v,p_\eps:=p,w_\eps:=w$, obtaining
  \[
    \Bip_\eps(\alpha_\eps,v_\eps;p_\eps,w_\eps)=\bipo(v,w)+\lambda\eps p\stackrel{\eps\down0}\to \bipo(v,w)=
    \Bipo(\alpha,v;p,w).
    \mathqed
  \]
\end{proof}}

An important consequence of the previous Lemma is provided by the
following lower-semicontinuity result for the integral functional
associated with $\Bip_\eps$.

\begin{proposition}[Lower-semicontinuity of the $\eps$-energy]
\label{prop:techn2}
Let us fix an interval $(s_0,s_1)$. For every choice of a vanishing
sequence $\eps_n>0$ and of functions $\alpha_n\in L^\infty(s_0,s_1),\
p_n\in L^1(s_0,s_1),\ v_n\in L^1(0,T;\xfin),\ w_n\in L^1(0,T;\xfins)$
such that
\[
\begin{aligned}
      \alpha_n&\weaksto \alpha\ &&\text{in }L^\infty(s_0,s_1),
      &p_n&\to p\ &&\text{in
        $L^1(0,T)$},\\
      \quad
      \ v_n&\weakto
      v\ &&\text{in }L^1(0,T;\xfin),\ &w_n &\to w\ &&\text{in }L^1(s_0,s_1),
    \end{aligned}
\]
we have the liminf estimates
\begin{align}
    \label{eq:123}
    \liminf_{n\to\infty}\int_{s_0}^{s_1} \Bip_{\eps_n}(\alpha_n(s),v_n(s);p_n(s),w_n(s))\, \d s&\ge
    \int_{s_0}^{s_1}\Bipo(\alpha(s),v(s);p(s),w(s))\,\d s,\\
    \label{eq:136}
    \liminf_{n\to\infty}\int_{s_0}^{s_1}
    \Bip_{0}(\alpha_n(s),v_n(s);p_n(s),w_n(s))\, \d s&\ge
    \int_{s_0}^{s_1}\Bipo(\alpha(s),v(s);p(s),w(s))\,\d s,
\end{align}
where $\Bip_{0}$ is defined in \eqref{e:novel}.
\end{proposition}
\begin{proof}
  It is sufficient to prove this result in the case $p_n\equiv p=0$.
  Then we notice that, by Lemma \ref{le:techn1}, the integrand
  \begin{displaymath}
    \tilde \Bip(\eps,\alpha,v,w):=\Bip_\eps(\alpha,v;0,w)\quad
    \text{for}\ (\eps,\alpha,v,w) \in [0,+\infty)\times [0,+\infty)\times\xfin\times\xfins
  \end{displaymath}
  is lower semicontinuous and convex in the pair $(\alpha,v)$.
  Then, inequality~\eqref{eq:123} follows from Ioffe's Theorem (see e.g.\
  \cite[Thm. 5.8]{Ambrosio-Fusco-Pallara00}). A similar argument
  yields
  \eqref{eq:136}.
\end{proof}

\subsection{Asymmetric dissipations, pseudo-total variation, and extended space-time curves}
\label{subsec:pseudo}
\paragraph{\textbf{Notation.}}  Hereafter,
   $\XX$ shall stand for the extended
space-time domain $ [0,T]\times \xfin$,  with elements $\xx=(t,u)$
denoted by bold letters. We shall denote by
  $\VV$ the  tangent cone
$[0,+\infty)\times\xfin$ to $\XX$ and by $ \vv=(\alpha,v)$ the
elements in $\VV$.

 We shall  consider  lower
semicontinuous dissipation functionals $\cR:\XX\times \VV\to
[0,+\infty)$ satisfying the following properties:
\begin{subequations}\label{eq:R123}
\begin{gather}
  \label{eq:161}  \forall\, \xx\in \XX: \ \FNorm\xx\cdot{ }\ \text{is convex and positively $1$-homogeneous;}
\\
  \label{eq:162}
  \exists\, C>0 \  \forall\,\xx\in \XX,\ \vv=(\alpha,v)\in \VV\,:
  \quad
  \FNorm\xx\vv{ }\ge C\|v\|_\xfin\,
\\
  \label{eq:163}
   \text{$\cR$ is lower semicontinuous on $\XX\times
   \VV$.}
  \end{gather}
\end{subequations}
In order to keep track of the time-component of $\vv$ we also set,
  for all $\beta \geq 0$,
\[
  \FNorm\xx\vv\beta=\alpha\beta+\FNorm\xx\vv{ }\quad \text{for all} \, \xx\in \XX,\ \vv=(\alpha,v)\in \VV.
\]
Notice that, for  any dissipation $\cR$ complying with  properties
\eqref{eq:R123}, the corresponding
functional $\cR_\beta$ satisfies the subadditivity property for all
$\xx \in \XX$ and $\vv_1,\, \vv_2 \in \VV$
\[
  \FNorm\xx{\vv_1+\vv_2}\beta\le
  \FNorm\xx{\vv_1}\beta+\FNorm\xx{\vv_2}\beta.
\]
\begin{example}[Dissipations induced by $\Psi_0$ and $\Bip$]
  \label{ex:main}
  \
  \begin{enumerate}
  \item
  Our first trivial example of a  dissipation fulfilling
  properties \eqref{eq:R123} is given by
  \begin{equation}
    \label{eq:164}
    \cP(\xx,\vv):=\Psi_0(v)\quad \text{for} \ \xx\in \XX,\ \vv=(\alpha,v)\in \VV.
  \end{equation}
  \item
  Our \textbf{main example} will be provided by the dissipation induced
  by the {\bipotential} $\Bip$,
   namely
\begin{equation}
\label{e:new} \cB(\xx;\vv):= \cP(\alpha,v,0,-\rmD\ene t u) \quad
\text{for} \ \xx=(t,u)\in \XX,\ \vv=(\alpha,v)\in \VV\,.
\end{equation}
 It is not difficult to check that $\cB$ satisfies all of
 assumptions \eqref{eq:R123}. Hence, for all $\beta\geq 0$ we set
\begin{equation}
  \label{eq:152}
  \Norm\xx\vv\beta:=\Bip(\alpha,v;\beta,-\rmD\ene t u)\quad  \text{for} \
  \xx=(t,u)\in\XX,\ \vv=(\alpha,v)\in \VV.
\end{equation}
\end{enumerate}
\end{example}

\begin{definition}[Pseudo-Finsler distance induced by $\cR$]
\label{def:bip-distance}
Given a dissipation function $\cR:\XX\times \VV\to
[0,+\infty)$ complying with \eqref{eq:R123},
for every $\xx_i=(t_i,u_i)\in \XX,\ i=0,1,$ with $0\leq t_0\le t_1 \leq T$, we  set
  \begin{equation}
    \label{eq:128}
    \begin{aligned}
      \FDist{\xx_0}{\xx_1}\beta:=
      \inf\Big\{&\int_{r_0}^{r_1} \FNorm{\sfxx(r)}{\dot\sfxx(r)}\beta\,\d
      r\,:
      \\&
      \sfxx=(\resct,\rescu)\in \mathrm{Lip}(r_0,r_1;\XX),\ \xx(r_i)=\xx_i,\
      i=0,1,
      \quad \dot\resct\ge0\Big\}.
    \end{aligned}
  \end{equation}
  If $t_0>t_1$ we set $  \FDist{\xx_0}{\xx_1}\beta:=+\infty$.
  We also define
  \[
   \TriCost{\cR_0}t{u_0}{u_1}{u_2}:=
    \FDist{(t,u_0)}{(t,u_1)}{0}+    \FDist{(t,u_1)}{(t,u_2)}{0}.
  \]
  (notice that this quantity is independent of $\beta$).
\end{definition}
\begin{remark}
  \label{rem:link}
  The link with the Finsler cost $\Delta_{\bipcE}$~\eqref{eq:69} induced by $(\bipcE)$
  is clear. For $\cR=\cB$ given by \eqref{e:new}, using
  $\Bip(0,v;0,w)=\bipo(v,w)$ we have, for $t_0=t_1=t$,
  \begin{equation}
    \label{eq:131}
    \Dist{(t,u_0)}{(t,u_1)}0=\Cost\bipcE t{u_0}{u_1}\quad
    \text{for every }u_0,u_1\in \xfin.
  \end{equation}
  When $\cR=\cP$ is given by~\eqref{eq:164},   we simply have
  \[
    \cost{\cP_\beta}{(t_0,u_0)}{(t_1,u_1)}=\beta(t_1-t_0)+\Psi_0(u_1-u_0)\quad \text{for }u_0,u_1\in \xfin,\
    0\le t_0<t_1\le T.
  \]
\end{remark}
\paragraph{\textbf{General properties of $ \FDist{\cdot}{\cdot}\beta$.}}
It is not difficult to check that the infimum in \eqref{eq:128} is
attained and, by the usual rescaling argument  (cf.~Remark~\ref{rem:newrmk}),  one can always choose an optimal
Lipschitz curve $ \sfxx=(\resct,\rescu)$ defined in $[0,1]$ such
that
\begin{equation}
  \label{eq:137}
   \FNorm\sfxx{\dot\sfxx}1
  \quad \text{is essentially constant and equal to }
  \FDist{\xx_0}{\xx_1}1=(t_1-t_0)+\FDist{\xx_0}{\xx_1}0.
\end{equation}
Properties~\eqref{eq:162}--\eqref{eq:163} yield, for every
$u_0,u_1\in \xfin$ and $0 \leq t_0\le t_1 \leq T,$
 the estimate
 \begin{equation}
  \label{eq:129}
  \beta(t_1-t_0)+C\|u_1-u_0\|_\xfin\le \FDist{(t_0,u_0)}{(t_1,u_1)}\beta.
\end{equation}
Notice that $\FDist\cdot\cdot\beta$ is not symmetric but still
satisfies the triangle inequality: for $\xx_i=(t_i,u_i)\in
\XX$ with $t_0\le t_1\le t_2$,  there holds
\[
  \FDist{\xx_0}{\xx_2}\beta\le \FDist{\xx_0}{\xx_1}\beta +\FDist{\xx_1}{\xx_2}\beta.
\]
Another useful property,   direct consequence of~\eqref{eq:163},
 is the lower semicontinuity 
with respect to convergence in $\XX$: if
$ \xx_{i,n}= (t_{i,n},u_{i,n})\to \xx_i=(t_i,u_i)$ in $\XX$ as $n\up+\infty$, $i=0,1$, then
\begin{equation}
  \label{eq:135}
  \liminf_{n\up+\infty}\FDist{\xx_{0,n}}{\xx_{1,n}}\beta\ge \FDist{\xx_0}{\xx_1}\beta.
\end{equation}
Indeed, assuming that the $\liminf$ in~\eqref{eq:135} is finite and
that, up to the extraction of a suitable subsequence, that it is a
limit, it is sufficient to choose an optimal sequence
$\sfxx_n=(\resct_n,\rescu_n)$ of Lipschitz curves as in
\eqref{eq:137}, which therefore satisfies a uniform Lipschitz bound
and, up to the extraction of a further subsequence, converges to some
Lipschitz curve $\sfxx=(\resct,\rescu)$. Then, \eqref{eq:135} can be
proved in the same way as \eqref{eq:136}.

In the case of the cost induced by $\cB$ induced by the {\bipotential}
$\Bip$, we have a refined lower-semicontinuity result:
\begin{lemma}
  \label{le:extra-lsc}
  Let $u_n,w_n:[t_0,t_1]\to \xfin$ be Borel maps and $(\eps_n)$ be
  a vanishing sequence.
Suppose that $u_n$ is absolutely continuous for every $n \in \N$
and that the following convergences hold as $n \to \infty$
\begin{gather*}
  u_n(t) \to u(t)\quad\text{and}\quad
  w_n(t) \to w(t) \quad \text{for all $t \in [t_0,t_1]$};\qquad
  \sup_{t  \in [t_0,t_1]} \|w_n(t)+\rmD\ene
t{u_n(t)}\|_{\xfins}\to0.
\end{gather*}
 Then,
  \begin{equation}
    \label{eq:182}
    \liminf_{n\up+\infty}\int_{t_0}^{t_1}\Big(\Psi_{\eps_n}(\dot u_n(t))+\Psi_{\eps_n}^*(w_n(t))\Big)\,\d t\ge
    \cost{\cB_0}{(t_0,u(t_0))}{(t_1,u(t_1))}.
  \end{equation}
\end{lemma}
\begin{proof}
  Up to extracting a further subsequence, it is not restrictive to assume
  that the $\liminf$ in \eqref{eq:182} is in fact a limit.
  We set as in \eqref{e:resc1}, \eqref{e:resc2}
  \begin{gather*}
    \rescs_n(t):=t-t_0+\int_{t_0}^t \Big(\Psi_{\eps_n}(\dot u_n(r))+\Psi_{\eps_n}^*(w_n(r))\Big)\,\d r,\quad
    \rescS_n:=\rescs_n(t_1),\\
    \resct_n(s):=\rescs_n^{-1}(s),\quad
    \rescu_n(s):=u_n(\resct_n(s)),\quad
    \rescw_n(s):=w_n(\resct_n(s))\quad \text{for all }
    s\in [0,\rescS_n]
  \end{gather*}
  so that
  \begin{equation}
    \label{eq:160}
    \int_{t_0}^{t_1}\Big(\Psi_{\eps_n}(\dot u_n(t))+\Psi_{\eps_n}^*(w_n(t))\Big)\,\d t=
    \int_0^{\rescS_n} \Bip_{\eps_n}(\dot\resct_n(s),\dot\rescu_n(s);0,\rescw_n(s))\,\d s.
  \end{equation}
  Since the sequences $(\resct_n)$ and $(\rescu_n)$ are uniformly Lipschitz,
  applying  the Ascoli-Arzel\`a Theorem
  we can extract a (not relabeled)  subsequence
  such that $\rescS_n \to \rescS$, and find functions
  $\resct:[0,\rescS] \to [t_0,t_1]$, $\rescu: [0,\rescS] \to \xfin$,
  and $\rescw: [0,\rescS] \to \xfins$,
  such that
  \begin{displaymath}
    \resct_n\to \resct,\quad
    \rescu_n\to\rescu,\quad
    \rescw_n\to\rescw=-\rmD\ene{\resct}{\rescu}\quad
    \text{uniformly in }  [0,\rescS].
  \end{displaymath}
  By construction, we have
  $\resct(0)=t_0$, $ \rescu(0)=u(t_0)$,
  $\resct(\rescS)=t_1$, and $ \rescu(\rescS)=u(t_1)$.
  Applying Proposition~\ref{prop:techn2}, we have
  \begin{equation}
  \label{liminf}
  \begin{aligned}
    \liminf_{n\up+\infty}\int_0^{\rescS_n} \Bip_{\eps_n}(\dot\resct_n(s),\dot\rescu_n(s);0,\rescw_n(s))\,\d s&\ge
    \int_0^{\rescS}\Bip_{0}(\dot\resct(s),\dot\rescu(s);0,\rescw(s))\,\d s\\&\ge
    \cost{\cB_0}{(\resct(0),\rescu(0))}{(\resct(\rescS),\rescu(\rescS))}.
  \end{aligned}
  \end{equation}
  Combining \eqref{eq:160} and \eqref{liminf}, we conclude
  \eqref{eq:182}.
\end{proof}
\paragraph{\textbf{The total variation associated with
$\Delta_{\cR_\beta}$.}} In the same way as in
Definition~\ref{def:Finsler_var} we introduced the total variation
$\mathop{\text{\sl Var}}_{\bipcE}$ associated with the Finsler cost
$\Delta_{\bipcE}$, it is now natural to define the total variation
associated with $\Delta_{\cR_\beta}$.
\begin{definition}[Total variation for the pseudo-Finsler distance $\Delta_{\cR_\beta}$]
  \label{def:total2}
  For every curve $\sfxx=(\resct,\rescu):$ $[0,\rescS]\to \XX$ such that $\resct$ is \emph{nondecreasing} and
  every interval $[\sfa,\sfb]\subset [0,\rescS]$ we set
  \begin{equation}
    \label{eq:132}
    \begin{aligned}
      \Var{\cR_\beta}{\sfxx}{\sfa}{\sfb}:=
      \sup\Big\{&\sum_{m=1}^M\FDist{\sfxx(s_m)}{\sfxx(s_{m-1})}\beta:\\
      &\sfa=s_0<s_1<\cdots<s_{M-1}<s_M=\sfb\Big\}.
    \end{aligned}
  \end{equation}
  For a \emph{non-parametrized} curve $u:[0,T]\to\xfin$ and $[a,b]\subset [0,T]$, we simply set
\[
  \Var{\cR_\beta}uab:=\Var{\cR_\beta}{\uu}ab, \quad \text{with} \ \ \uu(t):=(t,u(t))\in \XX,  \ \ t\in [0,T].
\]
\end{definition}
In view of~\eqref{eq:162}, it is immediate to
check that a curve $u$ with $\Var{\cR_0}u0T<+\infty$
belongs to $\BV([0,T];\xfin)$.

In contrast to the (pseudo)-total variation defined
in~\eqref{eq:81}, the above notion of total variation is lower
semicontinuous with respect to pointwise convergence (compare with
Remark \ref{rem:lsc-pseudo}).
\begin{proposition}[Lower semicontinuity of $\Var{\cR_\beta}\cdot\sfa\sfb$]
  \label{prop:lsc-Var}
  If $\sfxx_n=(\resct_n,\rescu_n):[0,\rescS]\to\XX$ is a sequence of curves pointwise converging to $\sfxx=(\resct,\rescu)$
  as $n\up\infty$, we have
  \begin{equation}
    \label{eq:133}
    \liminf_{n\up\infty}\Var{\cR_\beta}{\sfxx_n}\sfa\sfb\ge
    \Var{\cR_\beta}{\sfxx}\sfa\sfb.
  \end{equation}
\end{proposition}
\begin{proof}
  The argument is standard: for an arbitrary subdivision
  $\sfa=s_0<s_1<\cdots<s_{M-1}<s_M=\sfb$,
  \eqref{eq:135} yields
  \begin{displaymath}
    \sum_{m=1}^M\FDist{\sfxx(s_m)}{\sfxx(s_{m-1})}\beta\le
    \liminf_{n\up+\infty}\sum_{m=1}^M\FDist{\sfxx_n(s_m)}{\sfxx_n(s_{m-1})}\beta
    \le
    \liminf_{n\up\infty}\Var{\cR_\beta}{\sfxx_n}\sfa\sfb.
  \end{displaymath}
  Taking the supremum with respect to all  subdivisions of $[\sfa,\sfb]$ we obtain \eqref{eq:133}.
\end{proof}
\paragraph{\textbf{ Lipschitz curves.}}
The next result shows that, for Lipschitz curves, the total
variation can be calculated by integrating the corresponding
dissipation potential.
\begin{proposition}[The total variation for Lipschitz curves]
  Given $\beta,\, L>0$,
  a bounded curve $\sfxx:=(\resct,\rescu):[0,\rescS]\to\XX$ satisfies the $\Delta_{\cR_\beta}$--Lipschitz condition
  with Lipschitz constant $L$
  \begin{equation}
    \label{eq:145}
    \FDist{\sfxx(s_1)}{\sfxx(s_2)}\beta\le L(s_2-s_1)\quad \text{for every }0\le s_1\le s_2\le \rescS,
  \end{equation}
  if and only if it is Lipschitz continuous (with respect to the usual distance in $\XX$),
  $\resct$ is nondecreasing, and
  \begin{equation}
    \label{eq:146}
    \FNorm{\sfxx(s)}{\dot\sfxx(s)}\beta
    \le L\quad\text{for a.a.\ }s\in (0,\rescS).
  \end{equation}
  In this case, for every $\gamma\ge0$
  \begin{equation}
    \label{eq:147}
    \Var{\cR_\gamma}\sfxx\sfa\sfb=\gamma(\resct(\sfb)-\resct(\sfa))+
    \int_\sfa^\sfb \FNorm{\sfxx(s)}{\dot\sfxx(s)}0\,\d s.
  \end{equation}
\end{proposition}
\begin{proof}
  The sufficiency of condition \eqref{eq:146} is clear. Let us now consider  a
  curve $\xx$ satisfying \eqref{eq:145}: by  the coercivity~\eqref{eq:129}, $\xx$ is a
  Lipschitz curve in the usual sense and \cite[Prop. 2.2]{Rossi-Mielke-Savare08} yields
  \[
    \FDist{\sfxx(s_0)}{\sfxx(s_1)}\beta\le \int_{s_0}^{s_1} \sfm(s)\,\d s,\quad\text{where}\quad
    \sfm(s):=\lim_{h\downarrow0}\frac{\FDist{\sfxx(s)}{\sfxx(s+h)}\beta}h
  \]
  is the so-called \emph{metric derivative} of $\xx$ (see~\cite{Ambrosio95,
  Ambrosio-Gigli-Savare08}).
  The minimality of $\sfm$ ensures that
  \begin{equation}
    \label{eq:149}
    \sfm(s)\le  \FNorm{\sfxx(s)}{\dot\sfxx(s)}\beta
    \quad\text{for a.a.\ }s\in [0,\rescS].
  \end{equation}
  On the other hand, since $\cR_\beta$ is lower semicontinuous and $1$-homogeneous in $\vv$,
  for every $0<\sigma<1$ and $s\in [0,\rescS]$
  we find a constant $\delta>0$ such that
  \[
     \FNorm{\sfxx(r)}{\vv}\beta
    \ge
    \sigma\FNorm{\sfxx(s)}{\vv}\beta\quad
    \text{for every $\vv\in \VV$}
    \quad\text{if $|r-s|\le \delta$},
  \]
  so that a comparison with the linear segment joining $\sfxx(s)$ with $\sfxx(s+h)$ yields
  \[
    \begin{aligned}
      \FDist{\sfxx(s)}{\sfxx(s+h)}\beta&
      \le \sigma^{-1}
      \FNorm{\sfxx(s)}{\sfxx(s+h)-\sfxx(s)}\beta
    \end{aligned}
  \]
  Dividing by $h$ and passing to the limit first as $h\downarrow0$ and eventually as $\sigma\uparrow1$, we
  obtain the opposite inequality of
  \eqref{eq:149}.  Combining~\eqref{eq:149} (which holds
  as an equality) with~\eqref{eq:145}, we infer~\eqref{eq:146},
  and~\eqref{eq:147} ensues.
\end{proof}
\begin{proposition}[Reparametrization]
  \label{prop:rep}
  Let $u:[0,T]\to \xfin$ be a curve with finite total variation
  $V:=\Var{\cR_0}u0T<+\infty$, and let us set
  \begin{equation}
    \label{eq:157}
    \rescs(t):=t+\Var{\cR_0}u0t=\Var{\cR_1}u0t\quad \text{for every $t\in [0,T]$}.
  \end{equation}
  Then, there exists a Lipschitz parametrization $\sfxx=(\resct,\rescu):[0,\rescS]\to \XX$,
with
  $\rescS=V+T$,
  such that
  \begin{equation}
    \label{eq:139}
    \FNorm{\sfxx(s)}{\dot\sfxx(s)}1=1\quad \forae\, s \in (0,\rescS),
  \end{equation}
  \begin{equation}
    \label{eq:140}
    \resct(\rescs(t))=t,\quad \rescu(\rescs(t))=u(t)\quad\text{for every }t\in [0,T].
  \end{equation}
  In particular,
  \begin{equation}
    \label{eq:141}
    b-a+\Var {\cR_0}uab=\rescs(b)-\rescs(a)=\int_{\rescs(a)}^{\rescs(b)}
    \FNorm{\sfxx(s)}{\dot\sfxx(s)}1
    \,\d s.
  \end{equation}
\end{proposition}
\begin{proof}
  The proof is classical, at least when the dissipation $\cR$ is continuous and even in its second argument:
  we briefly sketch the main ideas and refer to \cite[Lemma 4.1]{Mielke-Rossi-Savare08}.

  Notice that the jump set $\rmJ_\rescs$ of the curve $\rescs$ given by~\eqref{eq:157}
   coincides with the jump set $\rmJ_u$ of $u$,  and
  $\rescs$ is injective in $\rmC_u:=(0,T)\setminus \rmJ_u$.
  We denote by $\resct$ its inverse, defined on
  $\sfC_u:=\rescs(\rmC_u)$ and extended to $\overline \sfC_u$ by its (Lipschitz) continuity; we also
  set $\rescu(s):=u(t)$ if $s=\rescs(t)\in \sfC_u$.
  Suppose now that $(\sfs_-,\sfs_+)$ is a connected component of $[0,\rescS]\setminus
  \overline\sfC_u$,
  corresponding to some time $\bar t\in [0,T]$ with $\rescs_\pm=\rescs(\bar t_\pm)$ and
  $\bar \rescs=\rescs(\bar t)\in [\sfs_-,\sfs_+]$.
  We have
  \[
  \begin{aligned}
   &  \rescu(\sfs_-)=\lim_{s\up \sfs_-}\rescu(s)=u(\bar t_-),\qquad
  \rescu(\sfs_+)=\lim_{s\down\sfs_+}\rescu(s)=u(\bar t_+),
  \\ &
    \label{eq:158}
    \bar\rescs-\sfs_-=\FDist{(\bar t,u(\bar t_-))}{(\bar t,u(\bar t))}0,\quad
    \sfs_+-\bar \rescs=\FDist{(\bar t,u(\bar t))}{(\bar t,u(\bar
    t_+))}0.
  \end{aligned}
  \]
  By Definition~\ref{def:bip-distance}, we can join $(\bar t,\rescu(\sfs_-))$ to
  $(\bar t,\rescu(\sfs_+))$ by a $\Delta_{\cR_0}$-Lipschitz curve (still denoted by $(\resct,\rescu)$)
  defined in $[\sfs_-,\sfs_+]$ with constant
  first component $\resct(s)=\bar t$,  and satisfying \eqref{eq:137} as well as $\rescu(\bar \rescs)=u(\bar t)$.

  It is then easy to check that the final
   curve $\sfxx=(\resct,\rescu)$ obtained by ``filling'' in this way all
  the (at most countable) holes in $[0,\rescS]\setminus \sfC_u$ satisfies \eqref{eq:140} and the Lipschitz condition
  \eqref{eq:145} with $L\le 1$. Applying \eqref{eq:146} and \eqref{eq:147} we get
  \begin{align*}
    \int_{\sfs(a)}^{\sfs(b)}\FNorm{\sfxx(s)}{\dot\sfxx(s)}1\,\d s \le
     \sfs(b)-\sfs(a)=\Var{\cR_1}u ab & \le
    \Var{\cR_1}\sfxx{\sfs(a)}{\sfs(b)}\\&=\int_{\sfs(a)}^{\sfs(b)}\FNorm{\sfxx(s)}{\dot\sfxx(s)}1\,\d s
  \end{align*}
   where the first inequality follows from the
  $1$-Lipschitz condition, the subsequent identity from the
  definition of $\sfs$, and the last one from~\eqref{eq:141}.
  \end{proof}
The reparametrization of Proposition \ref{prop:rep} is also useful
to express the distributional derivative of $u$. If
 $\Var{\cR_0}u0T<+\infty$,  we can introduce the distributional
derivative $  \MU{\cR_1}u:=\rescs'$ of $\rescs$, which is a finite
positive measure satisfying
 \[
  \MU{\cR_1}u([a,b])=\rescs(b)-\rescs(a),\quad
  \int_0^T \zeta(t)\,\d\MU{\cR_1}u(t)=-\int_0^T \dot\zeta(t)\rescs(t)\,\d t\quad
  \text{for every }\zeta\in \rmC^0_0(0,T).
\]
Notice that a singleton $\{t\}$ has strictly positive measure if and
only if $t\in \rmJ_u$; more precisely
\[
  \begin{aligned}
  &
  \MU{\cR_1}u(\{t\})=\TriCost{\cR_0}t{u(t_-)}{u(t)}{u(t_+)}\quad\text{if }t\in \rmJ_u;
  \\
  &
  \MU{\cR_1}u(\{t\})=0\quad\text{if }t\in \rmC_u=(0,T)
   \setminus \rmJ_u\,,
  \end{aligned}
\]
with obvious modification for $t=0,T$. As a general fact we have the
representation formula  (recall that $\sft$ is the inverse of
$\sfs$)
\begin{equation}
  \label{eq:168}
  \resct_\#\left(\Leb 1\Restr{(0,\rescS)}\right)=\MU{\cR_1}u,\quad
  \text{i.e.~}
  \int_0^T \zeta(t)\,\d\MU{\cR_1}u(t)=\int_0^\rescS\zeta(\resct(s))\,\d
  s,
\end{equation}
for every bounded Borel function $\zeta:[0,T]\to \R$. Since $\resct$ is
 injective in $\sfC_u:=\resct^{-1}(\rmC_u) \subset (0,\rescS)$, a Borel subset $A$ of
$\sfC_u$ is $\Leb 1$-negligible if and only if $\resct(A)$ has
$\MU{\cR_1}u$-measure $0$. Therefore, as the derivatives
$\dot\resct,\dot\rescu$ are Borel functions defined up to a $\Leb
1$-negligible subset of $(0,\rescS)$, the compositions
$\dot\resct\circ \rescs, \dot\rescu\circ\rescs$ are well defined in
$\rmC_u$. The next lemma shows that they play an important role.
\begin{proposition}
  \label{prop:calculus}
  The Lebesgue measure $\Leb 1\Restr{(0,T)}$
  and the vector measure $u'_\co=u'_\Le+u'_\rmC$
   are absolutely
   continuous w.r.t.~$\MU{\cR_1}u$, and we have
  \begin{equation}
    \label{eq:171}
    \frac{\d \Leb 1}{\d\MU{\cR_1}u}=\dot\resct\circ\rescs\quad\text{and}\quad
    \frac{\d u'_\co}{\d\MU{\cR_1}u}=\dot\rescu\circ\rescs\qquad \MU{\cR_1}u\text{-a.e.\ in }\rmC_u.
  \end{equation}
\end{proposition}
\begin{proof}
  The absolute continuity of both measures is easy, since $\Leb 1\le \MU{\cR_1}u$
  by~\eqref{eq:168} and the total variation
  $\|u'\|_\xfin$ is absolutely continuous w.r.t.\ $\MU{\cR_1}u$ thanks to \eqref{eq:129}.
  The first identity of \eqref{eq:171} can be proved as in \cite[Lemma 4.1]{Mielke-Rossi-Savare08}.
  Concerning the second one, let us set for every smooth function $\zeta\in \rmC^\infty_0(0,T)$
  \[
    J_u(\zeta):=\sum_{t\in \rmJ_u} \zeta(t)\big(u(t_+)-u(t_-)\big),
  \]
  and let us observe that we have
  \begin{equation}
    \label{eq:173}
    -\int_{\sfC_u}(\zeta\circ \resct)'(s)\, \rescu(s)\,\d s=
    \int_{\sfC_u}\zeta(\resct(s))\,\dot \rescu(s)\,\d s+J_u(\zeta).
  \end{equation}
  Indeed, denoting by $A_t=(\sfa_t,\sfb_t)=\resct^{-1}(t)$, $t\in \rmJ_u$, the connected components of
  $[0,\rescS]\setminus\sfC_u$, and
   recalling that $\rescu(\sfa_t)=u(t_-),\rescu(\sfb_t)=u(t_+)$,
  we have
  \begin{align*}
    &-\int_{\sfC_u}(\zeta\circ \resct)'(s)\, \rescu(s)\,\d s=
    -\int_0^\rescS(\zeta\circ \resct)'(s)\, \rescu(s)\,\d s+\sum_{t\in \rmJ_u}
    \int_{\sfa_t}^{\sfb_t} (\zeta\circ \resct)'(s)\, \rescu(s)\,\d s
    \\&=
    \int_0^\rescS \zeta(\resct(s))\,\dot \rescu(s)\,\d s
    -\sum_{t\in \rmJ_u} \int_{\sfa_t}^{\sfb_t} (\zeta\circ \resct)(s)\, \dot\rescu(s)\,\d s+J_u(\zeta)
    =\int_{\sfC_u}\zeta(\resct(s))\,\dot \rescu(s)\,\d s+J_u(\zeta).
  \end{align*}
  Therefore, there holds
  \begin{align*}
    \int_0^T \zeta(t)\,\d u'(t)&=
    -\int_0^T \dot\zeta(t)\,u(t)\,\d t=
    -\int_0^\rescS \dot\zeta(\dot\resct(s))\, u(\resct(s))\,\dot\resct(s)\,\d s=
    -\int_{\sfC_u}\dot\zeta(\resct(s))\, u(\resct(s))\,\dot\resct(s)\,\d s
    \\&=
    -\int_{\sfC_u}(\zeta\circ \resct)'(s)\, \rescu(s)\,\d s=
    \int_{\sfC_u} \zeta(\resct(s))\,\dot \rescu(s)\,\d s+J_u(\zeta)\\&=
    \int_{\rmC_u} \zeta(t)\, \dot\rescu(\rescs(t))\,\d\MU{\cR_1}u(t)+J_u(\zeta).
 \end{align*}
where the fifth identity ensues from~\eqref{eq:173} and the last one
from~\eqref{eq:168}.  Since
 \begin{displaymath}
    \int_0^T \zeta(t)\,\d u_\co'(t)= \int_0^T \zeta(t)\,\d u'(t)-J_u(\zeta)
 \end{displaymath}
 we conclude the second of~\eqref{eq:171}.
\end{proof}
\begin{corollary}[Integral expression for ${\rm Var}_\cR$]
  \label{cor:integral}
  Let $u:[0,T]\to \xfin$ fulfil $\Var{\cR_0}u0T<+\infty$, let $\mu$ be a positive finite measure
  such that $\Leb 1 \ll \mu$ and $u_\co'\ll\mu$, and let us set
  \begin{equation}
    \label{eq:174}
    \begin{aligned}
    \JVar{\cR_0}uab:=
    \Cost{\cR_0} a{u(a)}{u(a_+)} & +
    \Cost{\cR_0} b{u(b_-)}{u(b)}\\ & +
   \sum_{t\in \rmJ_u\cap (a,b)}\TriCost{\cR_0}t{u(t_-)}{u(t)}{u(t_+)}.
\end{aligned}
  \end{equation}
  Then,
  \begin{equation}
    \label{eq:155}
    \Var{\cR_0}uab=\int_a^b \FNorm
    {\big(t,u(t)\big)}{\Big(\frac{\d\Leb 1}{\d\mu}(t),\frac{\d u_\co'}{\d \mu}(t)\Big)}{0}\,\d\mu(t)+
    \JVar{\cR_0}uab.
  \end{equation}
\end{corollary}
\begin{proof}
  Since the expression on the right-hand side is independent of the measure $\mu$,
  it is not restrictive to choose $\mu=\MU{\cR_1}u$; by \eqref{eq:141} we have
  \begin{align*}
    b-a+\Var{\cR_0}uab&=\int_{(\rescs(a),\rescs(b))\cap\sfC_u}\FNorm{\sfxx(s)}{\dot\sfxx(s)}1\,\d s+
    \Leb 1((\rescs(a),\rescs(b))\setminus \sfC_u)\\&=
    \int_{(a,b)\cap\rmC_u}\FNorm{\sfxx(\rescs(t))}{\dot\sfxx(\rescs(t))}1\,\d\mu+
    \mu([a,b]\cap \rmJ_u)\\&=
    \int_{(a,b)\cap \rmC_u}\FNorm{\big(t,u(t)\big)}{\vphantom{\Big(}(\dot\resct(\rescs(t)),\dot\rescu(\rescs(t)))}1+
    \JVar{\cR_0}uab,
  \end{align*}
  and we conclude by \eqref{eq:171}.
\end{proof}

\subsection{Total variation for $\BV$ solutions}
\label{ss:6.3}

We focus now on the particular case \eqref{eq:152} of
Example \ref{ex:main}, when the dissipation $\cR$ is associated with
the {\bipotential} $\Bip$.

\begin{theorem}[Comparison between $\Var{\cB_0}u\cdot\cdot$ and
  $\pVar\bipcE u\cdot\cdot$]
  For every curve $u\in \BV([0,T];\xfin)$ and every interval $[a,b]\subset [0,T]$ we have
  \begin{equation}
    \label{eq:159}
    \pVar\bipcE u ab\le \Var{\cB_0}uab,
  \end{equation}
  and equality holds in \eqref{eq:159} if and only if $u$ satisfies the local stability condition
  \eqref{eq:65bis} on $(a,b)$.  Furthermore, if
  $\Var{\cB_0}uab<+\infty$, then $u$ satisfies
\eqref{eq:65bis} on $(a,b)$.
\end{theorem}
\begin{proof}
  Let us first notice that the jump contributions to the total variations
$\mathop{\text{\sl Var}}_{\bipcE}$ and $\mathop{\rm Var}_{\cB_0}$
   are the same by \eqref{eq:131}.
 Inequality \eqref{eq:159} then follows by applying \eqref{eq:155} and observing
 that  for $\mu$-a.a.\ $t\in [0,T]$
  \begin{equation}
    \label{eq:183}
    \cB_0\Big(\big(t,u(t)\big),\Big(\frac{\d\Leb 1}{\d\mu}(t),\frac{\d u_\co'}{\d \mu}(t)\Big)\Big)=
    \Bip\Big(\frac{\d\Leb 1}{\d\mu}(t),\frac{\d u_\co'}{\d \mu}(t);0,w(t)\Big)
    \ge\Psi_0\Big(\frac{\d u_\co'}{\d \mu}(t)\Big)
  \end{equation}
  (where we  have used the notation $w(t)=-\rmD\ene t{u(t)}$), the latter inequality ensuing
  from~\eqref{eq:13}.
  On the other hand,
in view of~\eqref{eq:14},  \eqref{eq:183} is an identity if and only
if  $w(t) \in K^*$ for $\mu$-a.a $t \in (0,T)$, i.e.
  if the local stability property~\eqref{eq:65bis} holds.

  Finally, since $\Leb 1\ll\mu$, $\frac{\d\Leb
    1}{\d\mu}(t)>0$ for $\Leb 1$-a.a $t \in (0,T)$.
  Therefore, on account
  of~\eqref{e:gamma-conv1} we conclude the last
  part of the statement.
\end{proof}
\begin{corollary}
  \label{cor:equivalent-BV}
  A curve $u:[0,T]\to\xfin$ is a $\BV$ solution if and only if it satisfies one of
   the following (equivalent) two conditions:
   \begin{equation}
    \label{eq:176}
    \Var{\cB_0}u{t_0}{t_1}+\ene {t_1}{u(t_1)}= \ene {t_0}{u(t_0)}+
    \int_{t_0}^{t_1} \partial_t \ene t{u(t)}\,\d t\quad\text{for every }0\le t_0\le t_1\le
    T,
  \end{equation}
  \begin{equation}
    \label{eq:175}
    \Var{\cB_0}u0T+\ene T{u(T)}\le \ene 0{u(0)}+
    \int_0^t \partial_t \ene s{u(s)}\,\d s.
  \end{equation}
\end{corollary}

\begin{lemma}
  \label{le:main-lsc}
  Suppose that $u_\eps\in \AC([0,T]; \xfin)$, $\eps>0$, is a family pointwise converging to $u$ as $\eps\downarrow0$,
  and $w_\eps:[0,T]\to\xfins$ satisfies $\|w_\eps(t)+\rmD\ene t{u_\eps(t)}\|_{\xfins}\to0$
  uniformly in $[0,T]$.
  Then,
  \begin{equation}
    \label{eq:177}
    \liminf_{\eps\downarrow0}\int_0^T \Big(\Psi_\eps(\dot u_\eps)+\Psi_\eps^*(w_\eps(t))\Big)\,\d t
    \ge
    \Var{\cB_0}u0T\ge\pVar{\bipcE}u0T.
  \end{equation}
\end{lemma}
\begin{proof}
  Choosing a finite partition $0=t_0<t_1<t_2<\cdots<t_N=T$ of the time interval $[0,T]$,
  Lemma \ref{le:extra-lsc} yields
  \begin{displaymath}
     \liminf_{\eps\downarrow0}\int_0^T \Big(\Psi_\eps(\dot u_\eps)+\Psi_\eps^*(w_\eps(t))\Big)\,\d t\ge
     \sum_{j=1}^N\cost{\cB_0}{(t_{j-1},u(t_{j-1}))}{(t_j,u(t_j))}.
  \end{displaymath}
  Taking the supremum of the right-hand side with respect to all  partitions of $[0,T]$, we end up with \eqref{eq:177}.
\end{proof}
We conclude this section with the \textbf{proof of Theorem
\ref{thm:equivalence}.}
\begin{proof}
 Let $(\resct,\rescu)$ be a parametrized
solution as in the statement of the theorem. It is easy to check
directly from definitions \eqref{eq:128} and \eqref{eq:132} that
\[
\begin{aligned}
  \Var{\cB_0}u0T&\le \int_0^\rescS \Bip(\dot\resct(s),\dot\rescu(s);0,-\rmD\ene{\resct(s)}{\rescu(s)})\,\d s
 \\ & \le
  \ene0{\rescu(0)}-\ene{\resct(\rescS)}{\rescu(\rescS)}+
  \int_0^\rescS \partial_t\ene{\resct(s)}{\rescu(s)}\dot\resct(s)\,\d s\\&=
  \ene0{u(0)}-\ene T{u(T)}+
  \int_0^T \partial_t \ene{t}{u(t)}\,\d t\,,
  \end{aligned}
\]
where the second inequality ensues from~\eqref{eq:35}. Thus,
\eqref{eq:175} holds, so that $u$ is  a $\BV$ solution by
Corollary~\ref{cor:equivalent-BV}. The converse implication follows
from Proposition \ref{prop:rep}.  \end{proof}
\bibliographystyle{siam}
\def\cprime{$'$} \def\cprime{$'$} \def\cprime{$'$} \def\cprime{$'$}

\end{document}